\documentclass[11pt]{article}
\usepackage[utf8]{inputenc}

\usepackage{amsmath} 
\usepackage[colorlinks=true,bookmarks=false,linkcolor=blue,urlcolor=blue,citecolor=blue,breaklinks=true]{hyperref}

\usepackage{breakcites} 

\usepackage[numbers]{natbib}
\usepackage{doi}
\usepackage{graphicx}
\usepackage{amssymb} 
\usepackage{amsfonts} 
\usepackage{amsthm} 
\usepackage[shortlabels]{enumitem} 
\usepackage{comment} 
\usepackage{bbm}
\usepackage{mathtools}
\usepackage{algorithm}
\usepackage{algpseudocode}
\usepackage[dvipsnames,table]{xcolor} 
\usepackage[margin=1in]{geometry} 
\usepackage{subfig}
\usepackage{multirow} 
\usepackage{pifont} 
\usepackage{booktabs} 
\usepackage{mathrsfs}
\usepackage{tikz-cd}
\usepackage[nottoc,numbib]{tocbibind}
\usepackage{multicol}

\makeatletter
\newcommand{\myitem}[1]{%
\item[#1]\protected@edef\@currentlabel{#1}%
}
\makeatother

\makeatletter
\renewcommand{\paragraph}{%
  \@startsection{paragraph}{4}%
  {\z@}{1ex \@plus 1ex \@minus .2ex}{-.5em}%
  {\normalfont\normalsize\bfseries}%
}
\makeatother

\newtheorem{theorem}{Theorem}[section]
\newtheorem{lemma}[theorem]{Lemma}
\newtheorem{proposition}[theorem]{Proposition}
\newtheorem{corollary}[theorem]{Corollary}
\newtheorem{example}[theorem]{Example}
\newtheorem{definition}[theorem]{Definition}
\newtheorem{remark}[theorem]{Remark}

\newcommand{\be}{\begin{equation}}
\newcommand{\ee}{\end{equation}}
\newcommand{\bthm}{\begin{theorem}}
\newcommand{\ethm}{\end{theorem}}
\newcommand{\blem}{\begin{lemma}}
\newcommand{\elem}{\end{lemma}}
\newcommand{\bpof}{\begin{proof}}
\newcommand{\epof}{\end{proof}}
\newcommand{\bcor}{\begin{corollary}}
\newcommand{\ecor}{\end{corollary}}
\newcommand{\bprop}{\begin{proposition}}
\newcommand{\eprop}{\end{proposition}}

\newcommand{\id}{\mathrm{id}}
\newcommand{\diag}{\mathrm{diag}}

\newcommand{\mc}[1]{\mathcal{#1}}
\newcommand{\mbb}[1]{\mathbb{#1}}
\newcommand{\mfk}[1]{\mathfrak{#1}}
\newcommand{\mscr}[1]{\mathscr{#1}}
\newcommand{\msf}[1]{\mathsf{#1}}

\newcommand{\argmin}[1]{\underset{#1}{\operatorname{argmin}}}

\newcommand{\D}{\mathrm{D}}

\newcommand{\RR}{\mathbb{R}}

\newcommand{\NN}{\mathbb{N}}

\newcommand{\CC}{\mathbb{C}}

\newcommand{\vct}[1]{\mathbb{#1}}  
\newcommand{\cvx}[1]{\mathcal{#1}} 

\providecommand{\keywords}[1]
{
  \small	
  \textbf{Keywords:} #1
}

\newcommand{\new}[1]{#1}
\title{\new{Dimension-}Free Descriptions of Convex Sets}
\date{June 12, 2024}

\author{Eitan Levin$^\dag$ and Venkat Chandrasekaran$^{\dag,\ddag}$ \thanks{Email: \texttt{\{eitanl, venkatc\}@caltech.edu}} \vspace{0.25in} \\ $^\dag$ Department of Computing and Mathematical Sciences\\ $^\ddag$ Department of Electrical Engineering \\ California Institute of Technology \\ Pasadena, CA 91125}

\begin{document}

\maketitle

\begin{abstract}
Convex sets arising in a variety of applications are well-defined for every relevant dimension. Examples include the simplex and the spectraplex that correspond to probability distributions and to quantum states; combinatorial polytopes and their relaxations such as the cut polytope and the elliptope in integer programming; and unit balls of regularizers such as the $\ell_p$ and Schatten norms in inverse problems. 
Moreover, these sets are often specified using conic descriptions that can be obviously instantiated in any dimension. 
We develop a systematic framework to study such \new{dimension-}free descriptions of convex sets. 
We show that \new{dimension-}free descriptions arise from a recently-identified phenomenon in algebraic topology called representation stability, which relates invariants across dimensions in a sequence of group representations. 
Our framework yields structural results for \new{dimension-}free descriptions pertaining to the relations between the sets they describe across dimensions, extendability of a single set in a given dimension to a freely-described sequence, and continuous limits of such sequences. We also develop a procedure to obtain parametric families of freely-described convex sets whose structure is adapted to a given application; illustrations are provided via examples that arise in the literature as well as new families that are derived using our procedure. 
We demonstrate the utility of our framework in two contexts. First, we develop an algorithm for a \new{dimension-}free analog of the convex regression problem, where a convex function is fit to input-output data; by searching over our parametric families, we can fit a function to low-dimensional inputs and extend it to any other dimension. Second, we prove that many sequences of symmetric conic programs can be solved in constant time, which unifies and strengthens several results in the literature.

\end{abstract}


\tableofcontents

\section{Introduction}\label{sec:intro}

Convex sets play a central role in numerous areas of the mathematical sciences such as optimization, statistical inference, control, inverse problems, and information theory.  
The convex sets arising in all these domains are often well-defined in every relevant dimension.  Indeed, unit balls of standard regularizers used in inverse problems (e.g., $\ell_p$ or Schatten norms) are defined for vectors and matrices of any size. 
Polytopes associated to graph problems and their relaxations (e.g., the cut polytope and the elliptope approximating it) are defined for graphs of any size. 
Information-theoretic quantities (e.g., relative entropy) as well as their quantum analogues, are defined for distributions and states on any number of (qu)bits. 
Consequently, all these convex sets and many others should be viewed not as single sets but as sequences indexed by dimension. 
Furthermore, there is often a single ``\new{dimension-}free'' description of all the sets in such a sequence, as the following example illustrates. 
\begin{example}\label{ex:free_descr_intro}
Here are a few examples of sequences of convex sets and their descriptions.
\begin{enumerate}[(a)]
    \item The simplex in $n$ dimensions is $\Delta^{n-1} = \{x\in\RR^n:x\geq0,\ \mathbbm{1}_n^\top x=1\}$, which is the set of probability distributions over $n$ items. Here $x\geq 0$ denotes an entrywise nonnegative vector $x$, and $\mathbbm{1}_n\in\RR^n$ is the vector of all-1's. 
    
    \item The spectraplex, or the set of density matrices, of size $n$ is $\mc D^{n-1} = \{X\in\mbb S^n: X\succeq0,\ \mathrm{Tr}(X)=1\}$. It is the set of density matrices describing mixed states in quantum mechanics~\cite[\S2.4]{nielsen_chuang_2010}. Here $X\succeq0$ denotes a symmetric positive-semidefinite (PSD) matrix $X$.
    
    \item The $\ell_2$ ball in $\RR^n$ is given by $\cvx B_{\ell_2}^n = \left\{x\in\RR^n: \begin{bmatrix} 1 & x^\top \\ x & I_n\end{bmatrix}\succeq0\right\}$ where $I_n$ denotes the $n\times n$ identity. 
    
    \item The elliptope of size $n$ is $\{X\in\mbb S^n: X\succeq0, \mathrm{diag}(X)=\mathbbm{1}_n\}$. It arises in a standard relaxation of the max-cut problem~\cite{goemans1995improved}.
\end{enumerate}
\end{example}
Note that all of the above descriptions are clearly defined for any relevant dimension. 
Indeed, these descriptions are composed of elements such as $\mathbbm{1}_n$ and $I_n$, linear maps such as $\mathrm{diag}(\cdot)$ and $\mathrm{Tr}(\cdot)$, and inequalities such as $\geq$ and $\succeq$, all of which are \emph{``\new{dimension-}free''} in the sense that they are well-defined in every relevant dimension. 
In this paper, we develop a systematic framework to study such \emph{freely-described convex sets}.

One motivation for this effort is to facilitate structural understanding of convex sets that can be instantiated in any dimension and of sequences of optimization problems over such sets.
Such sequences arise in several applications. For example, in extremal combinatorics~\cite{raymond2018symmetric, raymond2018symmetry} it is of interest to certify inequalities between graph homomorphism densities involving graphs of every size. In quantum information and control theory~\cite{huber2021positive, klep2018positive}, many problems reduce to optimizing (traces of) polynomials in matrix variables of every size.  
Despite the ubiquity of sequences of convex sets in these problem domains, the existence and interplay between the sets in different dimensions is typically not explicitly discussed or exploited in the literature.  
In the present paper, we explicitly consider sequences of convex sets that are related in a concrete fashion across dimensions, and we present a framework to derive several structural results about such sets as well as optimization programs over them; as one notable illustration, we show that sequences of invariant conic programs (including some arising in the above applications) can be solved in time independent of dimension.


A second motivation for our effort stems from the growing interest in learning solution methods for various problem families from data. 
Here one is typically given input-output data, and the goal is to fit a mapping that approximately fits the data.
This framework has been fruitfully applied to domains including integer programming, inverse problems, and numerical solvers for PDEs~\cite{kristo2020case,GNN_substruct_count,GNN_maxcut,ongie2020deep, alg_unrolling,learn_branch}.  
However, a fundamental limitation in much of this literature is that the mappings learned from data are typically only defined for inputs of the same dimension as the ones in the training data, and extension to inputs of different dimensions is handled on a case-by-case basis.  
In contrast, we wish to learn \emph{algorithms} that should be defined for inputs of any relevant size, i.e., we aim to identify a sequence of solution maps, one for each input size.  Identifying a freely-described convex set offers a convenient approach to learning an algorithm specified as linear optimization over convex sets (convex programs).  
Further, to facilitate numerical search over a family of freely-described convex sets---for example, to fit an element of this family to training data---it is natural to seek \emph{finitely-parametrized} families of such sets, such as the following examples.
\begin{example}\label{ex:free_descr_intro_param_reorg}
The following are examples of finitely-parametrized families of freely-described convex sets. 
\begin{enumerate}[(a)]
    \item The sequence of subsets $\left\{\alpha_1 \mathrm{diag}(X) + \alpha_2 X\mathbbm{1}_n: \allowbreak X\succeq0,\ \alpha_3\mathrm{Tr}(X)\mathbbm{1}_n + \alpha_4 \mathrm{diag}(X) + \alpha_5 X\mathbbm{1}_n = \alpha_6\mathbbm{1}_n\right\}$ of $\RR^n$ is parametrized by $\alpha\in\RR^6$.

    \item Free spectrahedra are sequences of the form $\left\{(X_1,\ldots,X_d)\in(\mbb S^n)^d: L_0\otimes I_n \allowbreak + \sum_{i=1}^dL_i\otimes X_i \succeq 0\right\}$ that are parametrized by $L_0,\ldots,L_d\in \mbb S^k$ for $k \in \NN$.  They arise in the theory of matrix convexity and free convex algebraic geometry~\cite{free_AG_chap,kriel2019introduction}.

    \item The sequence $\cvx C_n = \Bigg\{X\in\mbb S^n:\frac{\mathbbm{1}_n^\top X\mathbbm{1}_n}{n^2} L_1\otimes\mathbbm{1}_n\mathbbm{1}_n^\top + \frac{\mathrm{Tr}(X)}{n}L_2\otimes \mathbbm{1}_n\mathbbm{1}_n^\top + L_3\otimes\frac{1}{n}\left(X\mathbbm{1}_n\mathbbm{1}_n^\top + \mathbbm{1}_n\mathbbm{1}_n^\top X\right) + L_4\otimes \left(\diag(X)\mathbbm{1}_n^\top + \mathbbm{1}_n\diag(X)^\top\right) + L_5\otimes X + L_6\otimes \mathbbm{1}_n\mathbbm{1}_n^\top + L_7\otimes (nI_n) \succeq0\Bigg\}$ is parametrized by $L_1,\ldots,L_7\in\mbb S^k$. 
    We obtain this sequence from our framework in Section~\ref{sec:graphons} by viewing a matrix $X\in\mbb S^n$ representing a weighted graph as a step \emph{graphon}~\cite{lovasz2012large}, and deriving a parametric family of \new{dimension-}free descriptions extending continuously to more general spaces of graphons.
\end{enumerate}
Each example here is a parametric family, and constitutes a freely-described sequence of convex sets for each value of its associated parameters.
\end{example}
Note that the parametric \new{dimension-}free descriptions in Example~\ref{ex:free_descr_intro_param_reorg} are composed of linear combinations of the \new{dimension-}free elements and linear maps that appear in the descriptions of Example~\ref{ex:free_descr_intro}. In fact, the parameters above are simply the coefficients in these linear combinations. 
Thus, we can obtain finitely-parametrized families of freely-described convex sets by deriving \emph{finite-dimensional spaces} of \new{dimension-}free elements.

The central thesis of this paper is that the \new{dimension-}free elements which constitute \new{dimension-}free descriptions arise from a recently-identified phenomenon in algebraic topology called \emph{representation stability}, which relates invariants across dimensions in a sequence of group representations~\cite{CHURCH2013250}.  
We use these relations to formally define the \new{dimension-}free elements that constitute our \new{dimension-}free descriptions, and investigate various structural properties of such descriptions as well as the associated sequences of optimization problems.
Our definitions together with results from representation stability also yield finite-dimensional spaces of \new{dimension-}free elements, which we use to derive parametric families of freely-described convex sets adapted to specific applications and to fit elements of these families to data.
We outline the contributions of this paper in the remainder of the introduction.  We do not assume prior familiarity with representation stability, and the relevant background is presented in Section~\ref{sec:backgrnd_rep_stab}.

\subsection{Our Framework and Contributions}

This paper consists of four main contributions.  
First, we formally define \new{dimension-}free descriptions of convex sets by generalizing the insights derived from Examples~\ref{ex:free_descr_intro} and \ref{ex:free_descr_intro_param_reorg} using representation stability.  
Second, we prove structural results pertaining to freely-described sequences of convex sets, the relationship between sets in such a sequence in different dimensions, and limits of such sequences. 
Our proofs combine concepts from convex analysis and representation theory.  
Third, we derive consequences of our framework for invariant conic programs, showing that sequences of such programs of growing dimension can often be solved in constant time. Our results unify and generalize existing work in the literature.  
Finally, we apply our framework and its structural results to derive an algorithm for fitting a freely-described convex function to given data in different dimensions, a problem we call \new{dimension-}free convex regression. The latter two contributions are aimed at addressing our mathematical and algorithmic motivations above.


\subsubsection{Freely-described Convex Sets}
To formalize \new{dimension-}free descriptions of convex sets, we begin by briefly reviewing \emph{conic descriptions}, which express a convex set as an affine section of a convex cone.  
Conic descriptions are the most popular approach to specifying convex sets in the optimization literature, and they have played a central role in the development of modern convex optimization~\cite{BN_lects}. Indeed, we often classify convex sets based on their conic descriptions---polyhedra are affine sections of nonnegative orthants and spectrahedra are affine sections of PSD cones.  
Formally, if $\vct V,\vct W,\vct U$ are (finite-dimensional) vector spaces and $\cvx K\subseteq \vct U$ is a convex cone, then a convex subset $\cvx C\subseteq \vct V$ can be described using linear maps $A\colon \vct V\to \vct U$, $B\colon \vct W\to \vct U$ and a vector $u\in \vct U$ as follows:
\begin{equation}\label{eq:conic_descrip_preim}\tag{Conic}
    \cvx C = \{x\in \vct V: \exists y\in \vct W \textrm{ s.t. } Ax + By + u\in \cvx K\}.
\end{equation}
We will refer to the spaces $\vct W$ and $\vct U$ as the \emph{description spaces} associated to the conic description.  If the cone $\cvx K$ is a nonnegative orthant (resp., PSD cone), then linear optimization over $\cvx C$ is a linear (resp., semidefinite) program.  
The type of the cone as well as its dimension determine the computational complexity of optimization over $\cvx C$.

Each sequence of convex sets $\cvx C_n \subseteq \vct V_n$ in Examples~\ref{ex:free_descr_intro} and \ref{ex:free_descr_intro_param_reorg} is given by~\eqref{eq:conic_descrip_preim} for suitable sequences of description spaces $\vct W_n, \vct U_n$, of vectors $u_n \in \vct U_n$, of linear maps $A_n\colon \vct V_n \rightarrow \vct U_n$, $B_n\colon \vct W_n \rightarrow \vct U_n$, and of convex cones $\cvx K_n \subseteq \vct U_n$.  
%
\new{The vector spaces and cones are typically fixed in advance and are dictated by the application domain.}
\new{Having fixed the vector spaces and cones,} we seek to formalize the vectors $u_n$ and linear maps $A_n,B_n$ that are defined for any dimension, such as those appearing in Examples~\ref{ex:free_descr_intro} and~\ref{ex:free_descr_intro_param_reorg}. Furthermore, we wish to obtain finite-dimensional spaces of these \new{dimension-}free elements, which yield finitely-parametrized families of \new{dimension-}free descriptions. 
To that end, we observe that these \new{dimension-}free elements are sequences of invariants under sequences of groups related in a particular way across dimensions.  
For example, the all-1's vector $\mathbbm{1}_n$ of length $n$ is invariant under the group of permutations on $n$ letters acting by permuting coordinates. Further, the all-1's vectors of different lengths are related to each other: extracting the first $n$ entries of $\mathbbm{1}_{n+1}$ yields $\mathbbm{1}_n$.  Similarly, the $n\times n$ identity matrix $I_n$ is invariant under the orthogonal group of size $n$ acting by conjugation, and extracting the top left $n\times n$ submatrix of $I_{n+1}$ yields $I_n$.  
Thus, to give a formal definition for \new{dimension-}free vectors and linear maps, we consider sequences of groups acting on sequences of vector spaces, and we require the spaces in the sequence to be related to each other -- specifically, we embed lower-dimensional spaces into higher-dimensional ones and project higher-dimensional spaces onto lower-dimensional ones. 
Such sequences of group representations are called \emph{consistent sequences} and they were first defined in the seminal paper~\cite{CHURCH2013250} introducing representation stability.


\begin{definition}[Consistent sequences~\cite{CHURCH2013250}]\label{def:consistent_seqs}
Fix a family of compact\footnote{Our theory can be extended to reductive groups.} groups $\mscr G=\{\msf G_n\}_{n\in\NN}$ such that $\msf G_n\subseteq \msf G_{n+1}$. 
A \emph{consistent sequence} of $\mscr G$-representations is a sequence $\mscr V = \{(\vct V_n,\varphi_n)\}_{n\in\NN}$ satisfying the following properties:
\begin{enumerate}[(a)]
    \item $\vct V_n$ is an orthogonal $\msf G_n$-representation;
    \item $\varphi_n\colon \vct V_n\hookrightarrow \vct V_{n+1}$ is a linear $\msf G_n$-equivariant isometry.
\end{enumerate}
Unless we want to emphasize the embeddings $\varphi_n$, we shall identify $\vct V_n$ with its image inside $\vct V_{n+1}$. We then write $\mscr V = \{\vct V_n\}$ and take $\varphi_n$ to be inclusions $\vct V_n\subseteq \vct V_{n+1}$.\footnote{Formally, we obtain such inclusions inside the direct limit of the sequence.} 
\end{definition}

As $\varphi_n$ is an isometry, its adjoint $\varphi_n^*$ defines the orthogonal projection $\mc P_{\vct V_n}$ of higher-dimensional spaces onto lower-dimensional ones.  
We now formally define \new{dimension-}free vectors or linear maps as sequences of invariants projecting onto each other.

\begin{definition}[Freely-described elements]\label{def:freely_descr_elt}
    A \emph{freely-described element} in a consistent sequence $\{(\vct V_n,\varphi_n)\}$ of $\{\msf G_n\}$-representations is a sequence $\{v_n\in \vct V_n^{\msf G_n}\}$ of invariants satisfying $\varphi_n^*(v_{n+1})=v_n$ for all $n$.
\end{definition}
%

Importantly, the set of freely-described elements in a given consistent sequence $\mscr V$ is naturally a linear space,\footnote{Formally, the space of freely-described elements is the inverse limit $\varprojlim_n \vct V_n^{\msf G_n}$.} because if $\{v_n\}$ and $\{v_n'\}$ are freely-described elements then so is $\{\alpha v_n+\beta v_n'\}$ for any $\alpha,\beta\in\RR$.  
Furthermore, fundamental results in representation stability imply that these spaces of freely-described elements are often finite-dimensional. More precisely, when the consistent sequence $\mscr V$ is \emph{finitely-generated}, the restricted projections $\mathcal P_{\vct V_n}|_{\vct V_{n+1}^{\msf G_{n+1}}}\colon \vct V_{n+1}^{\msf G_{n+1}}\to \vct V_n^{\msf G_n}$ become isomorphisms for all $n$ exceeding the \emph{presentation degree} of the sequence, see Section~\ref{sec:backgrnd_rep_stab} for precise definitions.

Definition~\ref{def:freely_descr_elt} also defines freely-described linear maps. Indeed, if $\mscr V=\{(\vct V_n,\varphi_n)\}$ and $\mscr U = \{(\vct U_n,\psi_n)\}$ are consistent sequences of $\{\msf G_n\}$-representations, then the sequence of spaces of linear maps between them can be naturally identified with $\{(\vct V_n\otimes \vct U_n,\varphi_n\otimes\psi_n)\}$ where $\otimes$ is the tensor (or Kronecker) product, and this is also a consistent sequence. A freely-described element of $\mscr V\otimes\mscr U$ is a sequence of equivariant maps $\{A_n\in\mc L(\vct V_n,\vct U_n)^{\msf G_n}\}$ satisfying $\psi_n^*A_{n+1}\varphi_n = A_n$. When $\varphi_n,\psi_n$ are inclusions, this simplifies to $\mc P_{\vct U_n}A_{n+1}|_{\vct V_n} = A_n$ where $\mc P_{\vct U_n}$ is orthogonal projection onto $\vct U_n$. 
From results in representation stability, it follows that for a large class of consistent sequences, the sequence $\mscr V\otimes\mscr U$ is finitely-generated when $\mscr V,\mscr U$ are, and the presentation degree of the former can be bounded in terms of those of the latter; see Theorem~\ref{thm:calc_for_pres_degs_FIW}.  Thus, the space of freely-described linear maps is often finite-dimensional as well.

\begin{example}[Vectors with zero-padding]\label{ex:basic_consist_seq}
Let $V_n=\RR^n$ with the standard inner product, and let $\varphi_n(x)=[x^\top, 0]^\top$ correspond to padding a vector with a zero. This is a consistent sequence for many standard sequences $\{\msf G_n\}$ of groups, including the groups of $n\times n$ orthogonal matrices, and (signed) permutation matrices. Here $\msf G_n$ is embedded in $\msf G_{n+1}$ by sending $g\in \msf G_n$ represented as an $n\times n$ matrix to $\mathrm{blkdiag}(g,1)$. 
For each of these sequences of groups, we get a consistent sequence $\mscr V=\{(\vct V_n,\varphi_n)\}$.  When $\msf G_n=\msf{S}_n$ is the group of permutations on $n$ letters, the space of freely-described elements in $\mscr V$ is one-dimensional and consists of sequences $\{\alpha\mathbbm{1}_n\}$ for $\alpha\in\RR$. The space of freely-described linear maps from $\vct V_n$ to itself is two-dimensional, and consists of sequences $\{\alpha_1\mathbbm{1}_n\mathbbm{1}_n^\top + \alpha_2 I_n\}$ for $\alpha\in\RR^2$.
\end{example}

Having formalized freely-described vectors and linear maps, we are ready to define \new{dimension-}free descriptions of convex sets.  These are just sequences of conic descriptions specified using freely-described elements.
\begin{definition}[\new{Dimension-}free conic descriptions]\label{def:free_descriptions}
    Let $\mscr V=\{\vct V_n\},\ \mscr W= \{\vct W_n\},\ \mscr U=\{\vct U_n\}$ be consistent sequences of $\{\msf G_n\}$-representations, and $\{\cvx K_n\subseteq \vct U_n\}$ be a sequence of convex cones. A sequence of conic descriptions
    \begin{equation}\tag{ConicSeq}\label{eq:preim_seq}
        \cvx C_n = \left\{x\in \vct V_n: \exists y\in \vct W_n \textrm{ s.t. } A_nx + B_ny + u_n\in \cvx K_n\right\},
    \end{equation}
    is called \emph{dimension-free} if $\{A_n\},\{B_n\}$, and $\{u_n\}$ are freely-described elements of the consistent sequences $\mscr V\otimes\mscr U$, $\mscr W\otimes\mscr U$, and $\mscr U$, respectively.
\end{definition}
All the descriptions in Examples~\ref{ex:free_descr_intro} and \ref{ex:free_descr_intro_param_reorg} become \new{dimension-}free when the relevant sequences of vector spaces are endowed with natural consistent sequence structure.
Moreover, when $\mscr V\otimes\mscr U,\mscr W\otimes\mscr U,\mscr U$ are all finitely-generated, Definition~\ref{def:free_descriptions} yields a finitely-parametrized family of \new{dimension-}free descriptions, obtained by choosing bases for the spaces of freely-described $\{A_n\},\{B_n\},\{u_n\}$ and viewing the coefficients in these bases as the parameters. 
These parametric families generalize Example~\ref{ex:free_descr_intro_param_reorg}.  More broadly, they can be adapted to specific applications via the choice of embeddings and group actions involved in the consistent sequences, which formally relate instances of different sizes and their symmetries.

We can use \new{dimension-}free descriptions to describe convex \emph{functions} as well as sets, using the many correspondences between convex sets and functions from convex analysis. For example, given a freely-described sequence of convex sets we can consider the sequence of their support and gauge functions, and given a sequence of convex functions we can consider \new{dimension-}free descriptions for the sequence of their epigraphs~\cite[\S4]{rockafellar1970convex}. Thus, Definition~\ref{def:free_descriptions} yields finitely-parametrized infinite sequences of convex functions as well.

\begin{example}[Parametric convex graph invariants]\label{ex:graph_params}
    A convex graph invariant is a convex function over symmetric matrices (viewed as adjacency matrices of weighted graphs) that is invariant under conjugation of its argument by permutation matrices (viewed as relabelling the vertices of the graphs); for example, the max-cut value of a (weighted) graph is a convex graph invariant~\cite{chandrasekaran2012convex}. These examples and others are defined for graphs of any size, and therefore they correspond to a sequence of convex functions $\{f_n\colon \mbb S^n\to\RR\}$ such that $f_n$ is invariant under conjugation of its argument by permutation matrices for each $n$. 

    Our framework yields parametric families of convex graph invariants as support or gauge functions of parametric freely-described convex sets. 
    Set $\vct V_n=\mbb S^n$ with the action of $\msf G_n=\msf{S}_n$ by conjugation, since we seek convex subsets of symmetric matrices invariant under this group action. We choose embeddings $\varphi_n\colon \mbb S^n\hookrightarrow \mbb S^{n+1}$ by padding with a zero row and column, which corresponds to appending isolated vertices to a graph. This yields the consistent sequence $\mscr V=\{(\vct V_n,\varphi_n)\}$.
    
    To obtain convex sets, we also need to choose description spaces and cones. For simplicity, we choose description spaces $\vct W_n=0$ and $\vct U_n=\vct V_n$ with the same embeddings and group actions as for $\mscr V$, and the positive semidefinite cones $\cvx K_n = \mbb S^n_+$.
    Once the relevant consistent sequences and cones have been chosen, the parametric family of freely-described sets, parametrized in this case by $\alpha\in\RR^{11}$, are obtained transparently from our framework:
    \begin{equation}\begin{aligned}\label{eq:parametric_graph_param}
        \cvx C_n &= \Bigg\{X\in\mbb S^n: \alpha_1(\mathbbm{1}_n^\top X\mathbbm{1}_n)\mathbbm{1}_n\mathbbm{1}_n^\top + \alpha_2(\mathbbm{1}_n^\top X\mathbbm{1}_n)I_n + \alpha_3 \mathrm{Tr}(X)\mathbbm{1}_n\mathbbm{1}_n^\top + \alpha_4\mathrm{Tr}(X)I_n + \alpha_5\Big(X\mathbbm{1}_n\mathbbm{1}_n^\top + \mathbbm{1}_n\mathbbm{1}_n^\top X\Big)\\ &+  \alpha_6\Big(\mathrm{diag}(X)\mathbbm{1}_n^\top + \mathbbm{1}_n\mathrm{diag}(X)^\top\Big) + \alpha_7 X + \alpha_8 \mathrm{diag}(X\mathbbm{1}_n) +  \alpha_9I\odot X + \alpha_{10}\mathbbm{1}_n\mathbbm{1}_n^\top + \alpha_{11}I_n \succeq0\Bigg\}.
    \end{aligned}\end{equation}
    We can obtain a larger parametric family by enlarging the description space. For example, taking $k$ copies of the description spaces and cones above yields convex sets described by $k$ linear matrix inequalities of the form~\eqref{eq:parametric_graph_param}, depending on $11k$ parameters.
\end{example}


\subsubsection{Structural Aspects of \new{dimension-}free descriptions}\label{sec:intro_structural}

We have defined freely-described convex sets in Definition~\ref{def:free_descriptions} by relating the \emph{descriptions} of the convex sets in a sequence $\{\cvx C_n\}$ across dimensions. 
However, in many applications it is further desirable to directly relate the sets $\cvx C_n$ themselves or the functions that arise from them across dimensions. 
We study the following two relations (which are dual to each other by~\cite[Cor.~16.3.2]{rockafellar1970convex}).
\begin{definition}[Compatibility conditions]\label{def:compatibility_conds}
    Let $\{\vct V_n\}$ be a nested sequence of vector spaces and let $\{\cvx C_n\subseteq \vct V_n\}$ be a sequence of convex sets (respectively, let $\{f_n\colon \vct V_n\to\RR\cup\{\infty\}\}$ be a sequence of convex functions). We say that $\{\cvx C_n\}$ (resp., $\{f_n\}$) satisfies 
    \begin{enumerate}[align=left, font=\emph]
        \item[Intersection compatibility] if $\cvx C_{n+1}\cap \vct V_n = \cvx C_n$ (resp., $f_{n+1}|_{\vct V_n}=f_n$);
        \item[Projection compatibility] if $\mc P_{\vct V_n}\cvx C_{n+1} = \cvx C_n$ (resp., $\mc P_{\vct V_n}f_{n+1}=f_n$).
    \end{enumerate}
\end{definition} 
Here $(\mc P_{\vct V_n}f_{n+1})(x)=\inf_{x'\in\mc P_{\vct V_n}^{-1}(x)}f_{n+1}(x')$ defines a convex function on $\vct V_n$~\cite[Thm.~5.7]{rockafellar1970convex}.
One can check that a sequence of convex functions is intersection or projection compatible if and only if its sequence of epigraphs is correspondingly compatible, and a similar correspondence holds for gauge and support functions of compatible sequences of sets; see Section~\ref{sec:intro_notation}. In what follows, the nested sequence of spaces in Definition~\ref{def:compatibility_conds} will always be a consistent sequence.

The conditions in Definition~\ref{def:compatibility_conds} are natural in a variety of applications.
For example, many graph parameters remain unchanged when an isolated vertex is appended to the graph, such as the max cut value; hence these are intersection-compatible with respect to the embeddings in Example~\ref{ex:graph_params}. 
Other parameters, such as the stability number, are nonincreasing when taking induced subgraphs, and any small graph occurs as an induced subgraph of a larger one that has the same parameter value. 
Hence such parameters are projection-compatible with respect to the same embeddings.  
In inverse problems, it is desirable to use regularizers that are both intersection and projection compatible. Indeed, if $\{\vct V_n\}$ is a consistent sequence and we are given data about a signal $x\in\vct V_N$ that only depends on its projection $\mc P_{\vct V_n}x$ onto $\vct V_n$ for $n<N$, then compatibility of the regularizer ensures that the recovered signal also lies in $\vct V_n$, see Section~\ref{sec:inverse_probs}.  
Compatibility plays a central role in the analysis of approximations for the copositive cones, and the failure of a variant of intersection compatibility for the natural sums-of-squares relaxation of copositive cones underlies several results in this area~\cite{gvozdenovic2007semidefinite,laurent2023exactness}.
Compatibility conditions also arise in noncommutative geometry, where matrix-convex sets are defined as sequences of sets of matrices of each size related across dimensions by conditions stronger, in general, than Definition~\ref{def:compatibility_conds} (see Proposition~\ref{prop:compatible_cones_mtx_cvx}); in particular, the free spectrahedra in Example~\ref{ex:free_descr_intro_param_reorg} satisfy both intersection and projection compatibility. 
A freely-described sequence of convex sets need not satisfy either intersection or projection compatibility; see Example~\ref{ex:bad_cube} below.
Therefore, we investigate freely-described convex sets that additionally satisfy compatibility conditions, and we obtain three structural results pertaining to such sequences of sets.
Each of these results is obtained by combining the relevant concepts from representation stability along with notions from convex analysis. We now describe these three results in more detail.


Our first structural result gives conditions under which \new{dimension-}free descriptions \emph{certify} compatibility, i.e., under which the sequence of sets derived from these descriptions is evidently compatible.  Consider a freely-described sequence $\{\cvx C_n\}$ of convex sets \eqref{eq:preim_seq}.
Assuming that the sequence of cones $\{\cvx K_n\}$ underlying the description of $\{\cvx C_n\}$ is both intersection and projection compatible, we prove in Proposition~\ref{prop:compatible_descriptions} that the sequence of convex sets $\{\cvx C_n\}$ satisfies compatibility conditions if the underlying freely-described elements in~\eqref{eq:preim_seq} lie in a convex cone.
This result serves as the foundation for our subsequent developments---we use it to show in Section~\ref{sec:examples} that many sets arising in applications naturally admit \new{dimension-}free descriptions certifying their compatibility. We also use this result in Section~\ref{sec:numerical_examples} to design an algorithm fitting a freely-described and compatible sequence of sets to data.

Our second structural result is central to the development of our computational framework in Section~\ref{sec:numerical_examples} and it addresses the following question. Given a convex set with a conic description~\eqref{eq:conic_descrip_preim} in a fixed dimension $n_0$, when does this description \emph{extend} to a \new{dimension-}free description of a sequence of sets satisfying compatibility conditions?  As a concrete example, suppose we are given a convex relaxation for a combinatorial optimization problem, or a convex regularizer for an inverse problem, in a particular problem dimension; under what conditions can these be extended to problems in any desired dimension?  Leveraging our preceding result along with properties of presentation degrees of consistent sequences, 
we give conditions in Theorem~\ref{thm:extending_compatible_descriptions} on the dimension $n_0$ and the conic description in that dimension ensuring the desired extendability.  
We use this result in Section~\ref{sec:numerical_examples} to computationally extend a set fitted to data in a fixed dimension to any other dimension.

Finally, our third structural result develops a notion of a limiting object for a freely-described convex set. 
Given a consistent sequence $\{\vct V_n\}$ of $\msf G_n$-representations, consider the vector space $\vct V_\infty = \cup_n \vct V_n$ viewed as a representation of\footnote{Formally, these are the direct limits $\varinjlim \vct V_n$ and $\varinjlim \msf G_n$.} $\msf G_\infty = \cup_n \msf G_n$ and let $\overline{\vct V_\infty}$ denote the completion of $\vct V_\infty$ with respect to some norm. 
Consider now a freely-described sequence of convex sets $\{\cvx C_n\}$ that is intersection- and projection-compatible. The set $\cvx C_\infty = \cup_n \cvx C_n$ is a convex subset of $\vct V_\infty$, and we would like to describe its continuous limit $\overline{\cvx C_{\infty}}$ inside $\overline{\vct V_{\infty}}$. This is natural in many applications where problems of different sizes can be naturally viewed as suitable finite-dimensional discretizations of infinite-dimensional problems. Examples include finite graphs obtained via discretization of continuous graphons,  vectors obtained via discretizations of continuous-time signals, and matrices obtained via discretizations of operators between infinite dimensional spaces.
We show in Theorem~\ref{thm:descriptions_of_lims} that if the \new{dimension-}free description underlying the sequence $\{\cvx C_n\}$ certifies our compatibility conditions (in the sense of Proposition~\ref{prop:compatible_descriptions}) and if the freely-described elements constituting the description extend continuously to their respective limits, then the \new{dimension-}free description extends to a description of a dense subset of $\overline{\cvx C_{\infty}}$.
In particular, this result yields an infinite-dimensional conic program for optimizing a continuous linear functional over $\overline{\cvx C_{\infty}}$. 
\begin{example}[Convex graphon parameters]\label{ex:graphon_limit_set}
    Consider the freely-described sequence in Example~\ref{ex:free_descr_intro_param_reorg}(c), which is intersection- and projection-compatible with respect to the graphon consistent sequence (see Section~\ref{sec:graphons}). The elements of $\vct V_{\infty}$ are step functions on $[0,1]^2$ known as step graphons, and we endow them with the $L_{\infty}$ norm as is common in the literature~\cite{borgs2019graphons}. The completion $\overline{\vct V_{\infty}}$ then consists of certain piecewise-continuous graphons, and Theorem~\ref{thm:descriptions_of_lims} shows that the limit of the sequence $\{\cvx C_n\}$ in Example~\ref{ex:free_descr_intro_param_reorg}(c) with $L_7=0$ is
    \begin{equation*}\begin{aligned}
        \overline{\cvx C_{\infty}} = &\Bigg\{W\in \overline{\vct V_{\infty}}: 0\preceq \Big[(L_1)_{i,j}\int_{[0,1]^2}W(s,t)\, ds\, dt + (L_2)_{i,j}\int_{[0,1]}W(t,t)\, dt\\ &+ (L_3)_{i,j}\int_{[0,1]}[W(x,t) + W(t,y)]\, dt + (L_4)_{i,j}[W(x,x) + W(y,y)] + (L_5)_{i,j}W(x,y)\Big]_{i,j=1}^k + L_6\Bigg\}.
    \end{aligned}\end{equation*}
    Here positive-semidefiniteness of a matrix-valued function on $[0,1]^2$ is meant in the sense of matrix-valued kernels~\cite{6788407}, see~\eqref{eq:psd_kernel} for a definition. We require $L_7=0$ because the corresponding term does not extend continuously with respect to the $L_{\infty}$ norm. We derive this limiting description in Proposition~\ref{prop:graphon_descrip_prop}. 
\end{example}

We also apply Theorem~\ref{thm:descriptions_of_lims} to permutahedra and Schur--Horn orbitopes in Section~\ref{sec:DS_permuta} to obtain descriptions of limits in analogy to Example~\ref{ex:graphon_limit_set}. The resulting infinite-dimensional conic descriptions yield an infinite-dimensional generalization of the Schur--Horn theorem, which we give in Proposition~\ref{prop:SH_theorem_special}. In fact, we obtain limiting descriptions and a Schur--Horn theorem more generally in approximately finite-dimensional (AF) algebras in Appendix~\ref{apdx:SH_in_AF_algebras}.


\subsubsection{Invariant Conic Programs}\label{sec:intro_conic_progs}
Our framework yields structural results not only for descriptions of convex sets but also for sequences of optimization problems over them. 
Specifically, we use representation stability to unify and generalize a number of previous results which show that various sequences of invariant conic programs indexed by dimension can be solved in constant time.
Such sequences of programs arise in several applications including extremal combinatorics~\cite{raymond2018symmetric, raymond2018symmetry} and quantum information~\cite{huber2021positive, klep2018positive}.  Concretely, the programs that arise in certifying homomorphism density inequalities over graphs are invariant under symmetric groups of increasing sizes that relabel the vertices of the graphs involved.  Similarly, optimizing traces of matrix polynomials yields programs that are invariant under the unitary or orthogonal groups of increasing sizes that conjugate the matrices involved.
We now recall symmetry reductions of invariant programs and explain how constant-sized reductions arise from representation stability.

Consider optimizing a linear functional over a convex subset $\cvx C \subseteq \vct V$ specified as an affine section of a convex cone $\cvx K$ as in~\eqref{eq:conic_descrip_preim}. 
If the linear functional, the affine section, and the cone are invariant under the action of a group $\msf G$, then we can further restrict the constraint set to the invariant subspace $\vct V^{\msf G}$, thereby reducing the size of the program~\cite[\S3]{GATERMANN200495}.
%
%
%
Consider now a freely-described sequence of convex subsets $\{\cvx C_n\}$ of a consistent sequence $\{\vct V_n\}$ of $\{\msf G_n\}$-representations given by~\eqref{eq:preim_seq}.
The vectors and linear maps in~\eqref{eq:preim_seq} are $\msf G_n$-invariant as they are given by freely-described elements; if in addition the cones $\cvx K_n$ are $\msf G_n$-invariant, then the convex sets $\cvx C_n$ are also $\msf G_n$-invariant.  
When $\{\vct V_n\}$ is finitely-generated, so that the spaces of invariants in the sequence are all eventually isomorphic, optimizing a $\msf G_n$-invariant linear functional over $\cvx C_n$ for sufficiently large $n$ reduces as above to optimization over these constant-dimensional spaces of invariants.
However, even though the dimensionality of the variables in the symmetry-reduced programs stabilize, the complexity of the constraints might still grow with $n$.

To establish that the complexity of the constraints also stabilizes with $n$, we show that the invariant sections $\{\cvx K_n^{\msf G_n}\}$ of the cones have a constant-sized description in the following precise sense.
\begin{definition}[Constant-sized descriptions]\label{def:const_sized_descrp}
    Let $\{\vct U_n\}$ be a sequence of $\{\msf G_n\}$ representations and $\{\cvx K_n\subseteq \vct U_n\}$ a sequence of convex cones.
    For $t\in\NN$, we say that the sequence $\{\cvx K_n^{\msf G_n}\subseteq \vct U_n^{\msf G_n}\}$ admits a \emph{constant-sized description} for $n\geq t$ if there exists a single \new{finite-dimensional} vector space $\vct U$ containing a cone $\cvx K$, linear maps $T_n\colon \vct U\to \vct U_n^{\msf G_n}$, and subspaces $\vct L_n\subseteq \vct U$ such that $\cvx K_n^{\msf G_n}=T_n(\cvx K\cap \vct L_n)$ for all $n\geq t$. 
\end{definition}
Note that Definition~\ref{def:const_sized_descrp} does not require $\{\vct U_n\}$ to be a consistent sequence, nor in particular that $\vct U_n$ is embedded into $\vct U_{n+1}$. 
Proofs of constant-sized symmetry reductions in the literature have implicitly proceeded in a case-by-case manner by showing that the relevant cones, including symmetric PSD and relative entropy cones, have a constant-sized description in the sense of our Definition~\ref{def:const_sized_descrp}, see~\cite{riener2013exploiting,raymond2018symmetric,debus2020reflection,moustrou2021symmetry}.
In Section~\ref{sec:const_sized_progs}, we explain how these constant-sized descriptions can be generalized and derived systematically from an interplay between representation stability and the structure of the cones in question. To that end, we use a stronger form of representation stability known as uniform representation stability, which shows that the whole decomposition into irreducibles of the sequences of representations involved stabilize, rather than only their spaces of invariants; see Section~\ref{sec:uniform_rep_stab}. 
Our approach allows us to prove \new{the existence of constant-sized descriptions for invariant sections of the following cones under the actions of the classical Weyl groups. The latter consist of the sequence of permutation groups $\{\msf S_n\}$, signed permutation groups $\{\msf B_n\}$, and even-signed permutation groups $\{\msf D_n\}$, see Table~\ref{tab:gps_and_cones} for definitions.}
\begin{theorem}[Informal, see Theorem~\ref{thm:sym2_const_size}]\label{thm:informal}
	Consider the consistent sequence $\mscr V_0=\{\RR^n\}$ with embeddings by zero-padding and the standard actions of one of the sequences of classical Weyl groups as in Example~\ref{ex:basic_consist_seq}.  
    For any consistent sequence $\mscr V=\{\vct V_n\}$ obtained from $\mscr V_0$ by taking finitely-many direct sums, tensor products, symmetric algebras, or skew-symmetric algebras, the invariant sections of the cones $\{\mathrm{Sym}^2_+(\vct V_n)\}$ admit constant-sized descriptions for $n\geq d+k$, where $d$ and $k$ are the generation and presentation degrees of $\mscr V$, respectively. These degrees can be bounded using the calculus in Theorem~\ref{thm:calc_for_pres_degs_FIW}.
\end{theorem}
The sequences of classical Weyl groups are the groups of permutations $\{\msf{S}_n\}$, signed permutations $\{\msf{B}_n\}$, or even signed permutations $\{\msf{D}_n\}$.
Here $\mathrm{Sym}^2_+(\vct V)$ denotes the collection of nonnegative quadratic forms on $\vct V$. For example, Theorem~\ref{thm:informal} shows that the sections of the PSD cones $\{\mbb S^n_+\}$ invariant under the usual action of the classical Weyl groups by conjugation admit constant-sized descriptions.  
Theorem~\ref{thm:informal} is used to derive constant-sized descriptions for cones of invariant sums-of-squares in Theorem~\ref{thm:sym_sos_const_size} below. 
In the following result, we consider polynomials in \new{$\binom{n+k-1}{k}$} variables identified with $k$-\new{multisets} of $n$ letters, and show that the cones of sums of squares of such polynomials modulo any sequence of ideals admit constant-sized descriptions. The proof of the following result is given in Section~\ref{sec:psd_cone_const_size}.
\begin{theorem}\label{thm:sym_sos_const_size}
	Let $\{\msf G_n\}$ be one of the sequences of classical Weyl groups acting as usual on $\RR^n$. Let
	\begin{equation*}
	\mc I_n\subseteq \bigoplus_{d\geq 0}\mathrm{Sym}^d\left(\new{\mathrm{Sym}^k}\RR^n\right) \cong \RR[x_{i_1,\ldots,i_k}]_{\new{1\leq i_1\leq \cdots\leq i_k\leq n}} =:\vct V_n
	\end{equation*}
	be $\msf G_n$-invariant ideals, and let $\vct U_n=\mathrm{Sym}^{\leq 2d}(\new{\mathrm{Sym}^k}\RR^n)/\mc I_n$.  Consider the sums-of-squares cones $\mathrm{SOS}_{\vct U_n} = \{f\in \vct U_n: f \textrm{ is a sum of squares mod } \mc I_n\}$. Then the sequence $\{\mathrm{SOS}_{\vct U_n}^{\msf G_n}\}$ admits a constant-sized description for $n\geq 2kd$ if $\msf G_n=\msf S_n$ or $\msf B_n$, and for $n\geq 2kd+1$ if $\msf G_n=\msf D_n$.
\end{theorem}
\new{Theorem~\ref{thm:sym_sos_const_size} generalizes a number of results from the literature.
\begin{itemize}
    \item Setting $k=1$, $\mc I_n=(0)$, and $\msf G_n=\msf S_n$, we recover~\cite[Thms.~4.7, 4.10]{riener2013exploiting}. 
    \item Setting $k=1$, $\mc I_n=(0)$, and $\msf G_n=\msf B_n$ or $\msf D_n$, we recover~\cite[Cor.~3.23]{debus2020reflection}.
    \item For $k\geq 2$, setting $\mc I_n=(x_I-x_I^2, x_J)_{\substack{I\in\binom{[n]}{k}\\ J\not\in \binom{[n]}{k}}}$ where $\binom{[n]}{k}$ denotes the set $\{(i_1<\ldots<i_k): i_j\in[n]\}$, and setting $\msf G_n=\msf S_n$, we recover~\cite[Thm.~2.4]{raymond2018symmetric}.
\end{itemize}}
Theorem~\ref{thm:sym_sos_const_size} generalizes all of these to include any of the classical Weyl groups and any sequence of invariant ideals.

\new{Theorem~\ref{thm:sym_sos_const_size} can be applied to derive} constant-sized SDPs to certify graph homomorphism density inequalities~\cite{raymond2018symmetric,raymond2018symmetry}. 
Many problems in extremal combinatorics can be recast as proving polynomial inequalities between homomorphism densities of graphs, which is the fraction of maps between the vertex sets of two graphs that define graph homomorphisms. 
A simple example is Mantel's theorem, which states that the maximum number of edges in a triangle-free graph is $\lfloor n^2/4\rfloor$. Razborov proposed a method of certifying such inequalities using \emph{flag algebras}~\cite{razborov_2007}, which were shown in~\cite{raymond2018symmetric,raymond2018symmetry} to be sums-of-squares certificates of certain symmetric polynomial inequalities. 
Razborov's flags are interpreted in~\cite[\S3]{raymond2018symmetric} as ``free'' spanning sets for spaces of symmetric polynomials. Formally, they are freely-described elements in the sense of Definition~\ref{def:freely_descr_elt}. Our framework shows that both the existence of such freely-described spanning sets and the resulting constant-sized SDPs are consequences of representation stability.

We obtain similar results for invariant sections of cones derived from relative entropy inequalities that have recently played a role in polynomial optimization~\cite{sage,murray2021signomial}. Our main result for these cones is stated in Theorem~\ref{thm:sage_cones_const_size}, which generalizes~\cite[Thm.~5.3]{moustrou2021symmetry} from permutation groups to the other classical Weyl groups.

\subsubsection{\new{Dimension-}Free Convex Regression}\label{sec:intro_cvx_regr}

As our final contribution, we apply our framework to develop an algorithm for a \new{dimension-}free analog of the convex regression problem.  In the classic setting of this problem, the objective is to fit a convex function $f\colon\vct V\to\RR$ to a dataset $\{(x_i,y_i)\}_{i=1}^D$ of inputs $x_i\in \vct V$ and outputs $y_i\in \RR$ such that $f(x_i)\approx y_i$. 
In our \new{dimension-}free extension of this problem, we have a sequence $\{\vct V_n\}$ of vector spaces, usually of growing dimension, and data $\{(x_i,y_i)\in \vct V_{n_i}\oplus\RR\}$ in finitely-many dimensions $n_i$. We then seek an \emph{infinite} sequence of convex functions $\{f_n\colon \vct V_n\to\RR\}$ satisfying $f_{n_i}(x_i)\approx y_i$ in the dimensions in which data is available \new{and which, moreover, is freely-described.}

\new{To fit a freely-described sequence of functions}, we begin by endowing the vector spaces $\{\vct V_n\}$ containing the training data with the structure of a consistent sequence based on the symmetries of the problem at hand and the relations between problem instances in different dimensions.  
We then select description spaces $\{\vct W_n\},\{\vct U_n\}$ and cones $\{\cvx K_n\subseteq \vct U_n\}$ with respect to which our convex functions are to be described as in~\eqref{eq:preim_seq}, with freely-described vectors and linear maps parametrizing the desired family of freely-described functions; richer description spaces generally yield more expressive families of convex functions, but fitting a function from such a family is more expensive and requires data in higher dimensions by Theorem~\ref{thm:extending_compatible_descriptions}.  

Having chosen description spaces and cones as above, we have completely specified our family of freely-described functions and turn to the problem of fitting an element of this family to data. 
We do so in a fully algorithmic fashion, without needing to write down an explicit formula for members of our parametric family, using the following three steps. 
First, we compute a basis for invariant vectors and linear maps in the dimensions in which we have data using the algorithm of~\cite{pmlr-v139-finzi21a}. 
Second, we numerically identify coefficients in this basis that fit the data using an alternating minimization procedure based on convex duality. 
Third, we extend the invariant vectors and linear maps we identified in the data dimensions to other dimensions by solving linear systems arising from Definition~\ref{def:freely_descr_elt}. 
Our procedure is summarized in Algorithm~\ref{algo:free_descr} and detailed in Section~\ref{sec:numerical_examples}, and our implementation is publicly available at \url{https://github.com/eitangl/anyDimCvxSets}.
To summarize, once the choice of consistent sequence structure and description spaces is made based on the structure underlying an application, the dimensions of the available data, and the desired richness of the family of functions, the remainder of the procedure is fully computational.

We demonstrate our approach by obtaining semidefinite approximations of two non-semidefinite representable functions: the $\ell_{\pi}$ norm $\|x\|_{\pi}=\left(\sum_i|x_i|^{\pi}\right)^{1/\pi}$, and the following (nonnegative and positively homogeneous) variant of the quantum entropy function:
\begin{equation}\label{eq:quant_ent_mod}
f(X) = \mathrm{Tr}[(X + \mathrm{Tr}(X)I) \log(X/\mathrm{Tr}(X)+I)].
\end{equation}
The $\ell_{\pi}$ norm of a vector is defined for vectors of any length, is invariant under signed permutations, remains unchanged under zero-padding of its input, and can only increase if we append any other entry to a given input. It is therefore both intersection and projection compatible with respect to the embeddings and projections of Example~\ref{ex:basic_consist_seq}.
The function~\eqref{eq:quant_ent_mod} is well-defined for (positive-semidefinite) inputs of all sizes, is invariant under conjugation by orthogonal matrices, and its value is unchanged if we zero-pad its input.  Thus, it is intersection-compatible on $\{\mbb S^n\}$ with embeddings given by zero-padding.  Given evaluations of these two functions on low-dimensional inputs, we use our procedure to fit semidefinite approximations to them by choosing sequences of PSD cones for our descriptions, and fitting a freely-described and compatible sequence of convex functions. Thus, we ensure that our approximations are also group-invariant and satisfy the same compatibility properties as the underlying ground-truth functions.  
\new{Concretely, we fit an approximation to the $\ell_{\pi}$ norm using its values on 50 random vectors of length at most 2, and fit an approximation to~\eqref{eq:quant_ent_mod} using its values on 200 random PSD matrices of size at most $4\times 4$. To evaluate the quality of our approximations, we further generate $10^3$ random points in dimensions up to 20 and compare the value of our approximation to the ground truth functions on these points. We detail our description spaces, our fitting algorithm, and our data generation process in Section~\ref{sec:numerical_results}.}
Figure~\ref{fig:regr_results} shows the error in our semidefinite approximations for different dimensions when we search over all freely-described sets in our family when fitting to data, and when we search only over the smaller family of compatible sequences.  We see that imposing compatibility ensures that the error increases gracefully when extending to higher dimensions.

\begin{figure}[t]
    \centering
    \subfloat[\centering Learning $\ell_{\pi}$.]{{\includegraphics[width=.47\linewidth]{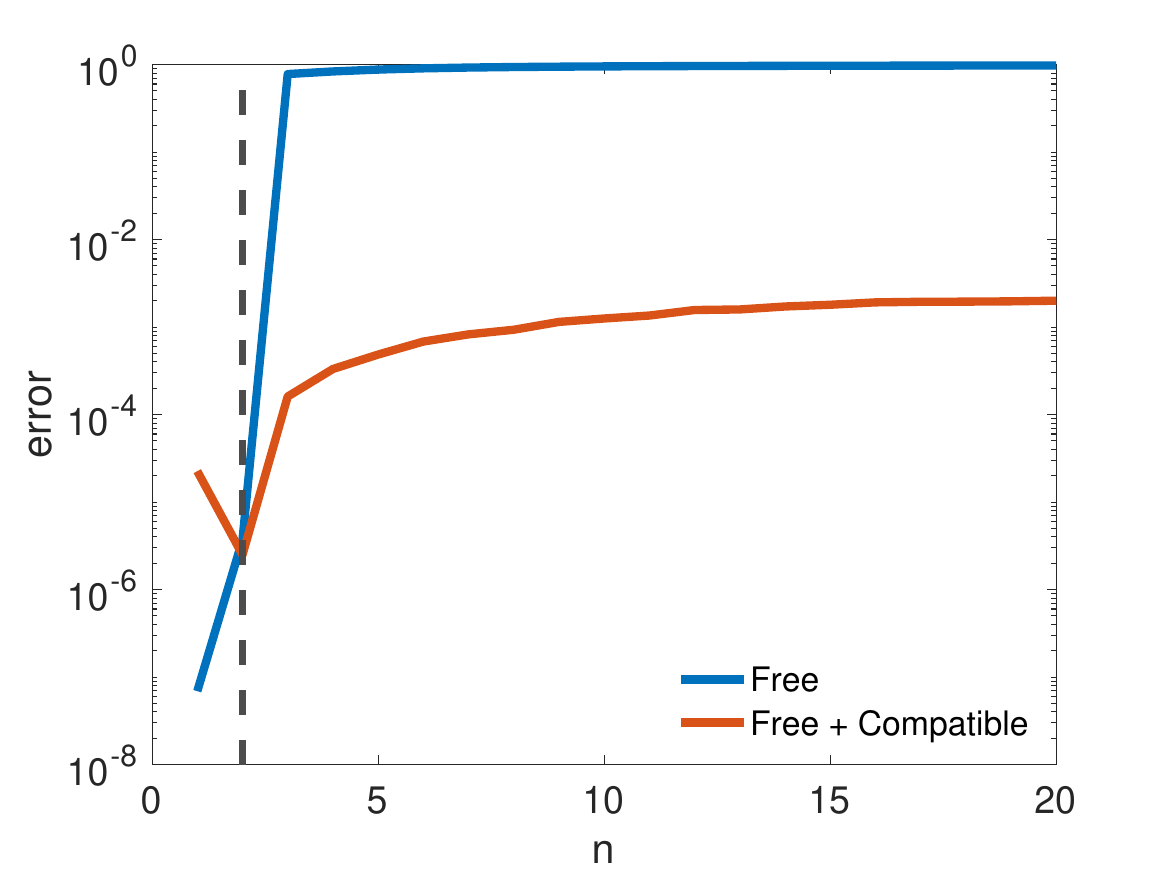} }}%
    \hfill
    \subfloat[\centering Learning quantum entropy.]{{\includegraphics[width=.47\linewidth]{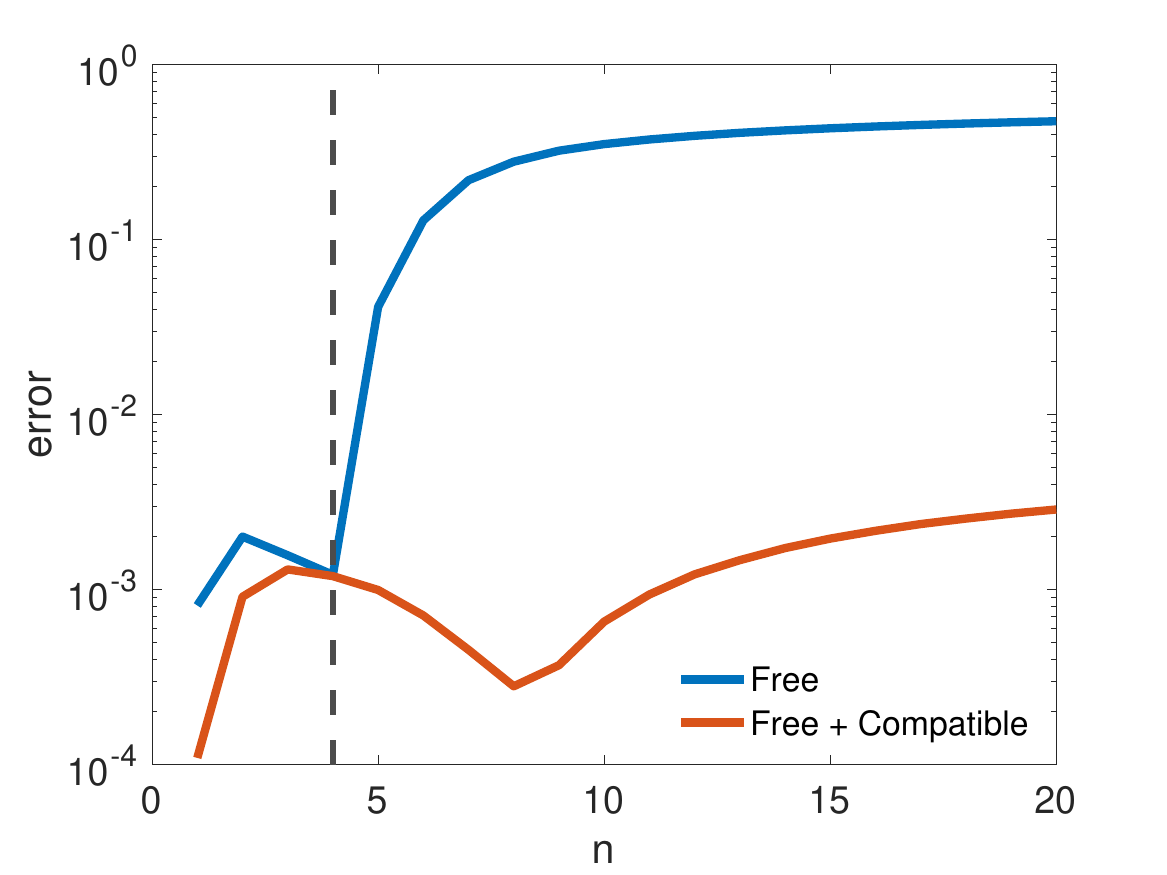} }}%
    \caption{Errors for learning the $\ell_{\pi}$ norm and the quantum entropy variant~\eqref{eq:quant_ent_mod}. The dashed vertical lines denote the max $n$ for which data is available.}%
    \label{fig:regr_results}%
\end{figure}

Obtaining semidefinite approximations for quantum information theoretic quantities such as~\eqref{eq:quant_ent_mod} facilitates the use of semidefinite programming solvers in convex optimization problems involving quantum entropy~\cite{fawzi2022geometry,fawzi2022optimal,fawzi2023certified,fawzi2024entropy}. We mention that the authors of~\cite{fawzi2019semidefinite} derived semidefinite approximations of the quantum entropy and related functions from an analytic perspective. We show in Section~\ref{sec:spectral_funcs} that their description is \new{dimension-}free but does not certify compatibility, in contrast to the description we obtain from data using the preceding procedure.

\subsection{Related Work}\label{sec:related_work}
We briefly survey several related areas.
\noindent\paragraph{Extended formulations of convex sets:} There is a large literature as well as a systematic framework on extended formulations in which conic descriptions of a convex set in a fixed dimension are investigated; see \cite{lifts_review} for a review.  The goal in this body of work is to express `complicated' convex sets in $\RR^d$ (e.g., polyhedra with many facets) as linear images of affine sections of `simple' convex cones in a space that is not much larger than $d$.  This framework is also applied to study \emph{equivariant lifts} of group-invariant convex sets, which are descriptions of the form~\eqref{eq:conic_descrip_preim} consisting of group-invariant cones, vectors, and linear maps, i.e., conic descriptions that make evident the group invariance of the convex set; see~\cite{equivariant_lifts} and~\cite[\S4.3]{lifts_review}.  These are precisely the type of descriptions we consider for each of the convex sets in our sequences.  Moreover, while the literature on extended formulations is typically articulated in the setting of a convex set in a \emph{fixed} dimension, many results in the area implicitly concern descriptions of a \emph{sequence} $\{\cvx C_n\}$ of convex sets and the complexity of these descriptions as a function of $n$~\cite{YANNAKAKIS1991441,equivariant_lifts,regular_polygons,sdp_of_rots}.  Thus, there are several points of contact with the present paper.  In fact, many (though not all) descriptions proposed in the literature on extended formulations can in fact be instantiated in any dimension, and are moreover \new{dimension-}free in our sense as we show in Section~\ref{sec:examples}. More broadly, to the best of our knowledge, descriptions of convex sets that arise from a systematic consideration of relations between dimensions have not been studied previously.  
By bringing such considerations to the fore, the present effort elucidates the representation-theoretic phenomena and their interaction with the convex-geometric aspects underpinning convex sets that can be instantiated in any dimension.

\noindent\paragraph{Free spectrahedra, noncommutative convex algebraic geometry:} 
A broad research program pursued in several areas involves the study of ``matrix'' or noncommutative analogues of classical ``scalar'' or commutative objects. 
Examples include random matrix theory and free probability~\cite{mingo2017free} in which matrix-valued random variables and their limits are the object of study as opposed to scalar-valued ones; and noncommutative algebraic geometry~\cite{free_AG_chap}, in which polynomials in noncommuting variables and their evaluations on matrices are studied as opposed to standard polynomials in commuting variables that are evaluated on scalars~\cite{NPA,burgdorf2016optimization,gribling2018bounds,gribling2019lower}.
Applying this program to convex sets yields matrix-convex sets and free spectrahedra as in Example~\ref{ex:free_descr_intro_param_reorg}(b). We refer the reader to~\cite{free_AG_chap, kriel2019introduction} for surveys and~\cite{freeSpectr_app1,freeSpectr_app2} for some applications.
In analogy to the present paper, results in this area explicitly pertain to sequences of sets which are ``freely-described'', in the sense that their description can be instantiated in any dimension.
For example, free spectrahedra are sequences of sets described by a single linear matrix inequality, and free algebraic varieties are defined by the same noncommutative polynomials instantiated on matrices of any size.
Another point of contact with our work is the consideration of relations between the sets in the sequence across dimensions, such as matrix-convex combinations which have been formalized and studied in this literature. 
Our notion of \new{dimension-}free descriptions is more general than the ones in this literature however, and it allows us to derive more flexible families of freely-described sets which are adapted to the structure underlying a broader range of applications. Further, the relations between sets in different dimensions we consider in this paper are less restrictive than matrix convexity, and they yield more general families of sets than free spectrahedra (free spectrahedra may be obtained in our framework via particular instantiations of description spaces, see Section~\ref{sec:free_spectr}).


\noindent\paragraph{Representation stability:} 
Representation stability arose out of the observation that the decomposition into irreducibles stabilizies for many sequences of representations.
This phenomenon has been formalized in~\cite{CHURCH2013250} using consistent sequences, and it has been further studied in~\cite{FImods,WILSON2014269,GADISH2017450} from a categorical perspective and in~\cite{sam_snowden_2015,sam2016gl,sam2017grobner} from a limits-based perspective. 
We relate the categorical and limits-based formalisms to our setting in appendix~\ref{sec:cat_reps} and Section~\ref{sec:lims_of_consist_seqs}, respectively, and we refer the reader to~\cite{farb2014representation,wilson2018introduction,sam2020structures} for introductions to this area.

Representation stability has been used to study sequences of polyhedral cones and their infinite-dimensional limits~\cite{van2021theorems}, as well as sequences of algebraic varieties, their defining equations, and their infinite-dimensional limits~\cite{Draisma2014,chiu2022sym,alexandr2023moment}. 
An important distinction between these works and ours is our application of representation stability to \emph{descriptions} of convex sets rather than to their extreme points or rays as in~\cite{van2021theorems}. Thus, we are able to study non-polyhedral sets such as spectrahedra and sets defined by relative entropy programs. 
Similarly, our study of infinite-dimensional limits in Section~\ref{sec:lim_of_sets} focuses on limiting descriptions and not just on limits of the sets themselves.

\subsection{Notation and Basics}\label{sec:intro_notation}
\begin{table}[ht]
    \centering
    \begin{tabular}{@{}ll@{}}
        \toprule
        Symmetric group & $\msf S_n = \{g\in\RR^{n\times n}: g \textrm{ is a permutation matrix}.\}$\\ \arrayrulecolor{black!30}\midrule
        Signed symmetric group & $\msf B_n = \{g\in\RR^{n\times n}: g \textrm{ is a signed permutation matrix}\}$ \\ \midrule
        Even-signed symmetric group & $\msf D_n = \{g\in B_n: g \new{\textrm{ has an even number of } -1 \textrm{ entries}}\}$\\ \midrule
        Orthogonal group & $\msf O_n = \{g\in\RR^{n\times n}: g^\top g=I_n\}$ \\ \midrule
        Space of linear maps & $\mc L(\vct V,\vct U) = \{A\colon \vct V\to \vct U \textrm{ linear}\}$;\qquad  $\mc L(\vct V) = \mc L(\vct V,\vct V)$.\\ \midrule 
        Direct sum & $\vct V\oplus \vct U = \vct V\times \vct U=\{(v,u): v\in \vct V, u\in \vct U\}$.\\ \midrule
        Direct powers & $\vct V^k = \vct V^{\oplus k} = \underbrace{\vct V\oplus\cdots\oplus \vct V}_{k \textrm{ times}}$.\\\midrule
        Tensor product & $\vct V\otimes \vct U = \mathrm{span}\{v\otimes u: v\in \vct V, u\in \vct U\}\cong \mc L(\vct V,\vct U)$. \\ \midrule
        Tensor power & $\vct V^{\otimes k} = \underbrace{\vct V\otimes\cdots\otimes \vct V}_{k \textrm{ times}}$.\\ \midrule 
        Symmetric algebra & $\begin{alignedat}{2} &\mathrm{Sym}^k(\vct V) &&= \mathrm{span}\{v_1\cdots v_k: v_i\in \vct V\}\\ &&&= \{\textrm{polynomials of degree } =k \textrm{ on } \vct V\}\\ &&&= \{\textrm{symmetric tensors of order } k \textrm{ over } \vct V\}\\ &\mathrm{Sym}^{\leq k}(\vct V)&&=\bigoplus_{i=0}^k\mathrm{Sym}^i(\vct V).\end{alignedat}$\\ \midrule 
        Alternating algebra & $\begin{alignedat}{2} &\bigwedge\nolimits^k\vct V &&= \mathrm{span}\{v_1\wedge\cdots \wedge v_k: v_i\in \vct V\}\\ &&&= \{\textrm{skew-symmetric tensors of order } k \textrm{ over } \vct V\}.\end{alignedat}$\\ \midrule
        Symmetric matrices & $\mbb S^n = \{X\in\RR^{n\times n}: X^\top = X\} = \mathrm{Sym}^2(\RR^n)$.\\ \midrule 
        Skew-symmetric matrices & $\mathrm{Skew}(n) = \{X\in\RR^{n\times n}: X^\top = -X\} = \bigwedge^2\RR^n.$ \\ \midrule 
        Spaces of invariants & $\begin{aligned} &\vct V^{\msf G} = \{v\in \vct V: g\cdot v = v \textrm{ for all } g\in \msf G\},\\ &\mc L(\vct V,\vct U)^{\msf G} = \{A\in\mc L(\vct V,\vct U): gA=Ag \textrm{ for all } g\in \msf G\}.\end{aligned}$
        \\\arrayrulecolor{black}\bottomrule
    \end{tabular}
    \caption{Commonly-used groups and vector spaces. Here $\vct V,\vct U$ are finite-dimensional vector spaces.}
    \label{tab:gps_and_cones}
\end{table}
We assume familiarity with the basics of representation theory and convex analysis, and we refer the reader to~\cite{fulton2013representation,serre1977linear} and~\cite{rockafellar1970convex}, respectively, for background. In what follows, we review a few basic notions from these areas and introduce notation used throughout the paper. We list several standard groups and constructions involving vector spaces in Table~\ref{tab:gps_and_cones}. 
\noindent\paragraph{Basics:} 
We \new{use $[n]$ to denote $\{1,\ldots,n\}$}. For $i\leq j$ we denote by $(i,j)$ the transposition permuting letters $i$ and $j$. For real numbers $a<b$, we \new{use $[a,b)$ to denote $\{x\in\RR:a\leq x<b\}$}.
For $i\in[n]$, we denote by $e_i\in\RR^n$ the $i$th standard basis vector with a 1 in the $i$th entry and zero everywhere else, and we write $e_i^{(n)}$ when we wish to emphasize the dimension. 
If $x\in\RR^n$, we denote by $\mathrm{diag}(x)\in\mbb S^n$ the diagonal matrix with $x$ on the diagonal. If $X\in\RR^{n\times n}$, we denote by $\mathrm{diag}(X)\in\RR^n$ the vector of its diagonal elements. 
All vector spaces in this paper are finite-dimensional real vector spaces equipped with an inner product $\langle\cdot,\cdot\rangle$ unless stated otherwise. We emphasize that some of the inner products we use are nonstandard, so the transpose of a matrix and the adjoint of the linear operator it represents may differ. 
Given a subspace $\vct W\subseteq \vct V$, we denote by $\mc P_{\vct W}\colon \vct V\to \vct W$ the orthogonal projection onto $\vct W$. 
We denote by $\RR^n_+$ the cone of entrywise nonnegative length-$n$ vectors, and by $\mbb S^n_+$ the cone of PSD $n\times n$ matrices. If $\vct V$ is a vector space, we let $\mathrm{Sym}^2_+(\vct V)\cong \mbb S^{\dim \vct V}_+$ denote the cone of PSD linear maps in $\mc L(\vct V)$. 

\noindent\paragraph{Representation theory:}
A (linear) action of a group $\msf G$ on a finite-dimensional vector space $\vct V$ is given by a group homomorphism $\rho\colon \msf G\to \mathrm{GL}(\vct V)$. Usually $\rho$ is clear from context and we omit it, writing $g\cdot v = \rho(g)v$ for $g\in \msf G$ and $v\in \vct V$ instead. All the groups we consider are compact and all group actions are orthogonal, meaning $\langle g\cdot x,g\cdot y\rangle=\langle v,v'\rangle$ for all $x,y\in \vct V$.
The action of $\msf G$ on $\vct V$ induces an action on $\vct V^{\otimes k}$ and $\mathrm{Sym}^k(\vct V)$ by setting $g\cdot v_1\otimes\cdots\otimes v_k=(gv_1)\otimes\cdots\otimes(gv_k)$ and $g\cdot (v_1\cdots v_k)=(gv_1)\cdots(gv_k)$ and extending by linearity. If $\vct V$ and $\vct U$ are both representations of $\msf G$, we have an action of $\msf G$ on $\vct V\otimes \vct U$ by $g\cdot(v\otimes u)=(gv)\otimes(gu)$ (and extending by linearity) and on $\mc L(\vct V,\vct U)$ by $g\cdot A = gAg^{-1}$, thus making the representations $\vct V\otimes \vct U$ and $\mc L(\vct V,\vct U)$ isomorphic. Linear maps invariant under this preceding group action are also called \emph{equivariant} or \emph{intertwining}, since they are precisely the linear maps commuting with the group elements. 
We denote the group ring of $\msf G$ by $\RR[\msf G] = \mathrm{span}\{e_g\}_{g\in \msf G}$, where $e_g$ is a basis element indexed by the group element $g$. Note that a representation of $\msf G$ is the same as a module over the ring $\RR[\msf G]$. 
    
If $\msf H\subseteq \msf G$ is a subgroup and $\vct V$ is a representation of $\msf H$, the \emph{induced representation} of $\msf G$ from $\vct V$ is $\mathrm{Ind}_{\msf H}^{\msf G}(\vct V)=\RR[\msf G]\otimes_{\RR[\msf H]}\vct V$. We have $\dim\mathrm{Ind}_{\msf H}^{\msf G}(\vct V)=|\msf G/\msf H|\dim \vct V$, and we apply this notion only when $\msf H$ has finite index in $\msf G$. If $g_1=\mathrm{id},g_2,\ldots,g_k$ are coset representatives for $\msf G/\msf H$, we have 
\begin{equation}\tag{Ind}\label{eq:ind_def}
    \mathrm{Ind}_{\msf H}^{\msf G}(\vct V) = \bigoplus_{i=1}^kg_i\vct V,    
\end{equation}
together with the following action of $\msf G$: \new{For each $g\in \msf G$ and $i\in[k]$, there is a unique $j\in[k]$ and $h\in\msf H$ satisfying $gg_i=g_jh$, and we let $g$ act on $g_iv$ for $v\in\vct V$ by $g\cdot g_iv = g_j(h\cdot v)$}. This construction is independent of the choice of coset representatives.

As vector spaces, we have an isomorphism $\mathrm{Ind}_{\msf H}^{\msf G}(\vct V)\cong \vct V^{|\msf G/\msf H|}$. Hence, an $\msf H$-invariant inner product $\langle\cdot,\cdot\rangle$ on $\vct V$ induces a $\msf G$-invariant inner product on $\vct V^{|\msf G/\msf H|}$ by setting $\langle g_iv,g_ju\rangle = \delta_{i,j}\langle v,u\rangle$ for $v,u\in \vct V$ and $i,j\in[|\msf G/\msf H|]$. Here $\delta_{i,j}=1$ if $i=j$ and zero otherwise.
We have an isomorphism $(\mathrm{Ind}_{\msf H}^{\msf G}\vct V)^{\msf G}\cong \vct V^{\msf H}$ sending $v\in \vct V^{\msf H}$ to $\sum_ig_iv\in (\mathrm{Ind}_{\msf H}^{\msf G}\vct V)^{\msf H}$ and $\widetilde v\in (\mathrm{Ind}_{\msf H}^{\msf G}\vct V)^{\msf H}$ to $\mc P_{\vct V}\widetilde v\in \vct V^{\msf H}$.

If $\msf H\subseteq \msf H'$ and $\msf G\subseteq \msf G'$ such that $\msf H'\cap \msf G=\msf H$, then we have the inclusion $\msf G/\msf H\hookrightarrow \msf G'/\msf H'$ sending $g\msf H\mapsto g\msf H'$; this in turn yields an inclusion $\mathrm{Ind}_{\msf H}^{\msf G}\vct V\hookrightarrow \mathrm{Ind}_{\msf H'}^{\msf G'}\vct V$ between induced representations by completing a set of coset representatives for $\msf G/\msf H$ to representatives for $\msf G'/\msf H'$. Here $\vct V$ is assumed to be an $\msf H'$-representation.

If $\vct V,\vct U$ are $\msf H$-representations and $A\in\mc L(\vct V,\vct U)^{\msf H}$, we can extend $A$ to a map $\mathrm{Ind}(A)\colon \mathrm{Ind}_{\msf H}^{\msf G}\vct V\to\mathrm{Ind}_{\msf H}^{\msf G}\vct U$ defined by $\mathrm{Ind}(A)(g_iv)=g_i(Av)$ where $g_i$ is one of the above coset representatives and $v\in \vct V$.
If $\vct V$ is a $\msf G$-representation and $\vct W\subseteq \vct V$ is an $\msf H$-subrepresentation, there is a $\msf G$-equivariant linear map $\mathrm{Ind}_{\msf H}^{\msf G}\vct W\to \vct V$ sending $g\otimes w\mapsto g\cdot w$ whose image is precisely $\RR[\msf G]\vct W = \mathrm{span}\{g\cdot w\}_{g\in \msf G, w\in \vct W}$.

\noindent\paragraph{Convex sets and functions:} The epigraph of a convex function $f\colon\vct V\to\RR\cup\{\infty\}$ is the convex set $\{(x,t)\in\vct V\oplus\RR: f(x)\leq t\}$. If $\cvx C\subseteq \vct V$ is a convex set then its gauge function (also called Minkowski functional) is $\gamma_{\cvx C}(x) = \inf\{t: x\in t\cvx C\}$ and its support function is $h_{\cvx C}(x) = \sup\{\langle y,x\rangle: y\in\cvx C\}$. 

Our compatibility conditions for convex sets in Definition~\ref{def:compatibility_conds} imply compatibility for convex functions derived from them using the above correspondences, and vice versa. Indeed, it is easy to check that a sequence of convex functions is intersection (resp., projection) compatible if and only if the sequence of their epigraphs is intersection (projection) compatible. Similarly, if a sequence of convex sets is intersection (resp., projection) compatible, then the sequence of their gauge functions is intersection (resp., projection) compatible. 
The correspondences between compatibility conditions between sets and their support functions is a bit subtler. If a sequence of sets is projection-compatible, then the sequence of their support functions is intersection-compatible. If a sequence of \emph{compact} sets is intersection-compatible, then their support functions are projection-compatible.

\section{Background on Representation Stability}\label{sec:backgrnd_rep_stab}
We review some fundamental definitions and results from the representation stability literature, which studies consistent sequences $\{\vct V_n\}$ of $\{\msf{G}_n\}$-representations as in Definition~\ref{def:consistent_seqs}.
We further require a notion of maps between consistent sequences, which enables us to define embeddings, quotients, and isomorphisms of consistent sequences. 
\begin{definition}[Morphisms of sequences]\label{def:morphisms_of_seqs}
    If $\mscr V=\{(\vct V_n,\varphi_n)\}$ and $\mscr U = \{(\vct U_n,\psi_n)\}$ are two consistent sequences of $\{\msf{G}_n\}$-representations, then a \emph{morphism of consistent sequences} $\mscr A\colon \mscr V\to\mscr U$ is a collection of linear maps $\mscr A=\{A_n\colon \vct V_n\to \vct U_n\}$ such that the following hold for each $n$:
    \begin{enumerate}[(a)]
        \item $A_n$ is $\msf{G}_n$-equivariant;
        \item $A_{n+1}\varphi_n = \psi_n A_n$.
    \end{enumerate}
\end{definition}
If $\varphi_n$ and $\psi_n$ are inclusions, condition (b) above becomes $A_{n+1}|_{\vct V_n}=A_n$. Note that morphisms are freely-described elements (as in Definition~\ref{def:freely_descr_elt}) of $\mscr V\otimes\mscr U$, but the converse may not hold.
Morphisms of sequences have appeared in the representation stability literature as the natural notion of maps between sequences, see~\cite[Def.~2.1.1]{FImods} and~\cite[\S3.2]{WILSON2014269}. 


\subsection{Generation Degree}\label{sec:gen_deg}
To relate invariants and equivariants across dimensions, we need canonical isomorphisms between spaces of invariants in a consistent sequence. Proposition~\ref{prop:gen_imp_surj} below shows that the projections $\mc P_{\vct V_n}$ are such isomorphisms, using the following parameter introduced in~\cite{FImods} to control the complexity of a consistent sequence.
\begin{definition}[Generation degree]\label{def:gen_deg}
    A consistent sequence $\mscr V=\{\vct V_n\}$ of $\{\msf{G}_n\}$-representations is \emph{generated in degree $d$} if $\RR[\msf{G}_n]\vct V_d = \vct V_n$ for all $n\geq d$. The smallest $d$ for which this holds is called the \emph{generation degree} of the sequence. A subset $\mc S\subseteq \vct V_d$ is called a \emph{set of generators} for $\mscr V$ if $\RR[\msf{G}_n]\mc S=\vct V_n$ for all $n\geq d$. 
    A sequence is \emph{finitely-generated} if it is generated in degree $d$ for some $d<\infty$.
\end{definition}
Note that $\RR[\msf{G}_n]\mc S = \mathrm{span}\{gx\}_{g\in \msf{G}_n,x\in\mc S}$, so that $\mscr V$ is generated in degree $d$ if the span of the $\msf{G}_n$-orbit of $\vct V_d$, when embedded in $\vct V_n$, is all of $\vct V_n$ for any $n\geq d$.
Note also that if $\mscr V$ is generated in degree $d$ then $\vct V_d$ is a set of generators for $\mscr V$.
\begin{proposition}\label{prop:gen_imp_surj}
    Suppose $\mscr V=\{(\vct V_n,\varphi_n)\}$ is a consistent sequence of $\{\msf{G}_n\}$-representations generated in degree $d$. Then the restrictions of the projections $\varphi_n^*|_{\vct V_{n+1}^{\msf G_{n+1}}}\colon \vct V_{n+1}^{\msf{G}_{n+1}}\to \vct V_n^{\msf{G}_n}$ to spaces of invariants are injective for all $n\geq d$, and are therefore isomorphisms for all large enough $n$.
\end{proposition}
\begin{proof}
    First, the map $\varphi_n^*$ is $\msf{G}_n$-equivariant because $\msf{G}_n$ acts orthogonally and $\msf{G}_n\subseteq \msf{G}_{n+1}$, so it maps $\msf{G}_{n+1}$-invariants in $\vct V_{n+1}$ to $\msf{G}_n$-invariants in $\vct V_n$. Second, suppose $\varphi_n^*(v)=0$ for some $v\in \vct V_{n+1}^{\msf{G}_{n+1}}$ with $n \geq d$. For any $u\in \vct V_{n+1}$, write $u = \sum_ig_i \varphi_n(u_i)$ where $u_i\in \vct V_n$ and $g_i\in \msf{G}_{n+1}$. Because $v$ is $\msf{G}_{n+1}$-invariant, we have $\langle v, u\rangle = \langle \varphi_n^*(v), \sum_iu_i\rangle = 0$. As $u\in \vct V_{n+1}$ was arbitrary, we conclude that $v=0$.  Thus, $\varphi_n^*$ maps $\vct V_{n+1}^{\msf{G}_{n+1}}$ injectively into $\vct V_n^{\msf G_n}$, so that $\dim \vct V_n^{\msf{G}_n}\geq \dim \vct V_{n+1}^{\msf{G}_{n+1}}$, for all $n\geq d$. Therefore, the sequence of dimensions $\dim \vct V_n^{\msf{G}_n}$ eventually stabilizes, at which point $\varphi_n^*$ becomes an isomorphism.
\end{proof}
Note that $\varphi_n^*=\mc P_{\vct V_n}$ is precisely the orthogonal projection onto $\vct V_n$.
Proposition~\ref{prop:gen_imp_surj} is stated in the representation stability literature in terms of the adjoints of the projections, viewed as maps between \emph{co}invariants, see~\cite[\S3]{FImods} for example.

\subsection{Presentation Degree}\label{sec:pres_deg}
While boundedness of the generation degree of a consistent sequence ensures that the projections eventually become isomorphisms, providing a precise quantification of this phenomenon requires a more sophisticated concept, namely the \emph{presentation degree}.  We describe this concept after giving some preliminary definitions.  Our presentation here is brief, and we refer the reader to Appendix~\ref{apdx:pres_deg} for a more detailed derivation of these notions motivated by our computational considerations.


\begin{definition}[Centralizing subgroups]\label{def:stab_subgps}
Let $\mscr V = \{\vct V_n\}$ be a consistent sequence of $\{\msf{G}_n\}$-representations. For any $d\leq n$, define the \emph{centralizing subgroups} of $\vct V_d$ by
\begin{equation*}
    \msf{H}_{n,d} = \{g\in \msf{G}_n: g\cdot v = v \textrm{ for all } v\in \vct V_d\}.
\end{equation*}
\end{definition}

Note that the subgroup generated by $\msf{G}_d$ and $\msf{H}_{n,d}$ inside $\msf{G}_n$ is the set of products $\msf{G}_d\msf{H}_{n,d}=\{gh: g\in \msf{G}_d,\ h\in \msf{H}_{n,d}\}$ since $ghg^{-1}\in \msf{H}_{n,d}$ for all $g\in \msf{G}_d$ and $h\in \msf{H}_{n,d}$. 
\begin{definition}[$\mscr V$-modules]\label{def:V_mods}
    Let $\mscr V=\{\vct V_n\}$ and $\mscr U=\{\vct U_n\}$ be consistent sequences of $\{\msf{G}_n\}$-representations, and let $\{\msf{H}_{n,d}\}_{d\leq n}$ be the centralizing subgroups of $\mscr V$ as in Definition~\ref{def:stab_subgps}.
    We say that $\mscr U$ is a \emph{$\mscr V$-module} if $\vct U_d\subseteq \vct U_n^{\msf{H}_{n,d}}$ for all $d\leq n$.
\end{definition}


\begin{definition}[Induction and algebraically free\footnote{Freeness here is meant in the algebraic sense of being generated by generators with no nontrivial relations between them (see Appendix~\ref{apdx:pres_deg}), in contrast to Definition~\ref{def:free_descriptions} where it is meant in the sense of dimension-free descriptions.} sequences]\label{def:free_mods}
Let $\mscr V$ be a consistent sequence of $\{\msf{G}_n\}$-representations, and for $d\leq n$ let $\msf{H}_{n,d}\subseteq \msf{G}_n$ be its centralizing subgroups. 
\begin{enumerate}[(a)]
    \item Fix $d\in\NN$ and a $\msf{G}_d$-representation $\vct W$ on which $\msf{H}_{d,d}$ acts trivially. For each $n\geq d$, we view $\vct W$ as a $\msf{G}_d\msf{H}_{n,d}$-representation on which $\msf{H}_{n,d}$ acts trivially.  The associated \emph{$\mscr V$-induction sequence} is defined as:
    \begin{equation*}
        \mathrm{Ind}_{\msf{G}_d}(\vct W) = \left\{\mathrm{Ind}_{\msf{G}_d\msf{H}_{n,d}}^{\msf{G}_n}\vct W\right\}_n,
    \end{equation*}
    where the induced representation is taken to be 0 when $n<d$. This is a $\mscr V$-module.

    \item A consistent sequence $\mscr F$ is an \emph{algebraically free $\mscr V$-module} if it is a direct sum of $\mscr V$-induction sequences. The sequence $\mscr V$ itself is \emph{algebraically free} if it is an algebraically free $\mscr V$-module.
\end{enumerate}
\end{definition}
\begin{definition}[Relation and presentation degrees]\label{def:relations}
Let $\mscr V$ be a consistent sequence of $\{\msf{G}_n\}$-representations.
    We say that a $\mscr V$-module $\mscr U$ is \emph{generated in degree $d$, related in degree $r$, and presented in degree $k=\max\{d,r\}$} if there exists an algebraically free $\mscr V$-module $\mscr F$ generated in degree $d$, and a surjective morphism of sequences $\mscr F\to \mscr U$ whose kernel is generated in degree $r$. The smallest $k$ for which this holds is called the \emph{presentation degree} of $\mscr U$.
\end{definition}
Note that the presentation degree is at least as large as the generation degree (cf.\ Definition~\ref{def:gen_deg}). 

\begin{example}\label{ex:basic_consist_seq_degs}
Let $\vct V_n=\RR^n$ with embeddings by zero-padding as in Example~\ref{ex:basic_consist_seq}. Recall that this is a consistent sequence for each of the sequences of groups $\msf{G}_n = \msf O_n,\msf B_n,\msf D_n,\msf S_n$ acting by their standard $n\times n$ matrix representations. Here $\msf{H}_{n,d}$ is the subgroup of $n\times n$ orthogonal or signed permutation matrices fixing the first $d$ coordinates.

This sequence is generated in degree 1 for all the groups listed above. Indeed, any of the canonical basis vectors $e_i$ are obtained from the first one $e_1$ via the action of $\msf S_n$.

If $\msf{G}_n=\msf B_n$ or $\msf S_n$ then this sequence is algebraically free and hence presented in degree 1 as well, while if $\msf{G}_n=\msf D_n$ then it is not free but presented in degree 2. Indeed, we have $|\msf D_n/\msf D_1\msf{H}_{n,1}|=2n$ when $n\geq 2$ with coset representatives $(1,i)s^p$ for $p\in\{0,1\}$, $i\in[n]$ where $s=\diag(-1,-1,1\ldots,1)$. Hence~\eqref{eq:ind_def} yields $\mathrm{Ind}_{\msf D_1\msf{H}_{n,1}}^{\msf D_n}\RR = \RR^n\oplus\RR^n$ on which $\sigma\in \msf S_n\subseteq\msf D_n$ acts by $\sigma(x,y)=(\sigma x,\sigma y)$ and $s(x,y)=([y_1,y_2,x_3,\ldots,x_n]^\top,[x_1,x_2,y_3,\ldots,y_n]^\top)$. We have equivariant linear maps $\mathrm{Ind}_{\msf D_1\msf{H}_{n,1}}^{\msf D_n}\RR\to\RR^n$ sending $(x,y)\mapsto x-y$ giving a morphism of sequences $\mathrm{Ind}_{\msf D_1}\RR\to \{\RR^n\}$ with kernel generated in degree 2.
\end{example}
The presentation degree enables us to strengthen Proposition~\ref{prop:gen_imp_surj} and to quantify more precisely when the projections there become isomorphisms.
\begin{proposition}\label{prop:presentation_implies_isom_of_canon}
    Let $\mscr V$ be a consistent sequence of $\{\msf{G}_n\}$-representations and $\mscr U$ be a $\mscr V$-module presented in degree $k$. Then the maps $\mc P_{\vct U_n}\colon \vct U_{n+1}^{\msf{G}_{n+1}}\to \vct U_n^{\msf{G}_n}$ are isomorphisms for all $n\geq k$.
\end{proposition}
\begin{proof}
    As $\mscr U$ is presented in degree $k$, there exists an algebraically free $\mscr V$-module $\mscr F=\{\vct F_n\}$ and a surjective morphism $\mscr F\to \mscr U$ such that both its kernel $\mscr K=\{\vct K_n\}$ and $\mscr F$ itself are generated in degree $k$. 
    Because each map $\vct F_n\to \vct U_n$ is a $\msf{G}_n$-equivariant surjection with kernel $\vct K_n$, its restriction to invariants $\vct F_n^{\msf{G}_n}\to \vct U_n^{\msf{G}_n}$ is surjective with kernel $\vct K_n^{\msf{G}_n}$. 

    As $\mscr F$ is an algebraically free $\mscr V$-module, there exist integers $d_j$ and $\msf{G}_{d_j}$-representations $\vct W_{d_j}$ satisfying $\mscr F = \bigoplus_j\mathrm{Ind}_{\msf{G}_{d_j}}\vct W_{d_j}$. Such $\mscr F$ has generation degree $\max_jd_j\leq k$. Therefore, letting $\{\msf{H}_{n,d}\}$ be the centralizing subgroups of $\mscr V$, we have for $n\geq k$ (see Section~\ref{sec:intro_notation})
    \begin{equation*}
        \vct F_n^{\msf{G}_n} = \bigoplus\nolimits_j\Big(\mathrm{Ind}_{\msf{G}_{d_j}\msf{H}_{n,d_j}}^{\msf{G}_n}(\vct W_{d_j})\Big)^{\msf{G}_n} \cong \bigoplus\nolimits_j\vct W_{d_j}^{\msf{G}_{d_j}},    
    \end{equation*}
    Thus, $\dim \vct F_n^{\msf{G}_n}$ is constant for $n\geq k$.
    Moreover, by Proposition~\ref{prop:gen_imp_surj} and the fact that $\mscr K$ and $\mscr U$ are generated in degree $k$, we have $\dim \vct K_n^{\msf{G}_n} \geq \dim \vct K_{n+1}^{\msf{G}_{n+1}}$ and similarly $\dim \vct U_n^{\msf{G}_n}\geq \dim \vct U_{n+1}^{\msf{G}_{n+1}}$ for all $n\geq k$. 
    
    By the rank-nullity theorem, we have $\dim \vct U_n^{\msf{G}_n} = \dim \vct F_n^{\msf{G}_n} - \dim \vct K_n^{\msf{G}_n}$. As $\dim \vct F_n^{\msf{G}_n}$ is constant while both $\dim \vct U_n^{\msf{G}_n}$ and $\dim \vct K_n^{\msf{G}_n}$ are nonincreasing for $n\geq k$, we conclude that they are all constant for $n\geq k$. 
    To conclude, we note that $\mc P_{\vct U_n}$ is injective when restricted to $\vct U_{n+1}^{\msf{G}_{n+1}}$ for all $n\geq k$ by Proposition~\ref{prop:gen_imp_surj}.
\end{proof}

\begin{remark}[$\mscr V$-modules vs.\ centralizing subgroups]\label{rmk:V-mods_vs_subgps}
The definition of presentation degree assumes a ``base'' consistent sequence $\mscr V$.
Note however that it depends only on the centralizing subgroups $\{\msf{H}_{n,d}\}$ of $\mscr V$.
In fact, any sequence of subgroups $\{\msf{H}_{n,d}\subseteq \msf{G}_n\}_{d\leq n}$ satisfying $\msf{H}_{n+1,d}\supseteq \msf{H}_{n,d}$, $\msf{H}_{n,d+1}\subseteq \msf{H}_{n,d}$, and $\msf{H}_{n+1,d}\cap \msf{G}_n=\msf{H}_{n,d}$ for $d\leq n$ arise as centralizing subgroups of such a consistent sequence. 

The centralizing subgroups play a central role because they determine embeddings $\{g\varphi_{n-1}\cdots\varphi_d\}_{g\in \msf{G}_n}\allowbreak \cong \msf{G}_n/\msf{H}_{n,d}$ of $\vct V_d$ into $\vct V_n$, and the combinatorics of these embeddings yields Theorem~\ref{thm:calc_for_pres_degs_FIW}. See the proof of~\cite[Prop.~2.3.6]{FImods} and Appendix~\ref{sec:cat_reps} for example.
We formulate our results in terms of $\mscr V$-modules rather than the subgroups $\{\msf{H}_{n,d}\}$ directly because the sequences we use are often constructed from the same base sequence as in Section~\ref{sec:consist_seq_construct} below, easing the application of our results. 
\end{remark}

\subsection{Constructions of Consistent Sequences}\label{sec:consist_seq_construct}
Expressive families of freely described convex sets require complex description spaces, and in turn complex consistent sequences.  In this section, we describe common operations that yield complicated consistent sequences from simpler ones, along with a calculus for bounding the generation and presentation degrees of the resulting sequences.


Fix a family of groups $\mscr G=\{\msf{G}_n\}_{n\in\NN}$ such that $\msf{G}_n\subseteq \msf{G}_{n+1}$. Suppose $\mscr V = \{(\vct V_n,\varphi_n)\}$ and $\mscr U = \{(\vct U_n,\psi_n)\}$ are consistent sequences of $\mscr G$-representations. Then the following are also consistent sequences of $\mscr G$-representations:
\begin{enumerate}[align=left, font=\emph]
    \item[(Sums)] The \emph{direct sum} of $\mscr V$ and $\mscr U$ is $\mscr V\oplus \mscr U =\{(\vct V_n\oplus \vct U_n,\varphi_n\oplus\psi_n)\}$. 

    If $\vct W$ is a fixed vector space, viewed as a trivial $\msf{G}_n$-representation for all $n$, denote $\mscr V\oplus \vct W = \{(\vct V_n\oplus \vct W,\varphi_n\oplus\mathrm{id}_{\vct W})\}$.
        
    \item[(Tensors)] The \emph{tensor product} of $\mscr V$ and $\mscr U$ is $\mscr V\otimes \mscr U = \{(\vct V_n\otimes \vct U_n,\varphi_n\otimes\psi_n)\}$. 

    This is also the sequence of spaces of \emph{linear maps} $\mc L(\vct V_n,\vct U_n)\cong \vct V_n\otimes \vct U_n$, where we embed $A_n\colon \vct V_n\to \vct U_n$ to $(\varphi_n\otimes\psi_n)A_n=\psi_nA_n\varphi_n^*\colon \vct V_{n+1}\to \vct U_{n+1}$.

    The \emph{order-$k$ tensors over $\mscr V$} is $\mscr V^{\otimes k}$.

    If $\vct W$ is a fixed vector space, viewed as a trivial $\msf{G}_n$-representation for all $n$, denote $\mscr V\otimes \vct W = \{(\vct V_n\otimes \vct W,\varphi_n\otimes\mathrm{id}_{\vct W})\}$. \new{Since $\vct V_n\otimes\vct W\cong \vct V_n^{\oplus\dim\vct W}$ as $\msf G_n$-representations, we can identify $\mscr V\otimes\vct W$ with $\mscr V^{\oplus \dim\vct W}$.}
    
    \item[(Polynomials)] The \emph{degree-$k$ polynomials over $\mscr V$} is $\mathrm{Sym}^k\mscr V=\{(\mathrm{Sym}^k\vct V_n,\varphi_n^{\otimes k})\}$, which is also the sequence of order-$k$ symmetric tensors over $\mscr V$. Here we view $\mathrm{Sym}^k\vct V_n\subseteq \vct V^{\otimes k}$ and restrict $\varphi_n^{\otimes k}$ to that subspace.
    The sequence of polynomials of degree at most $k$ is denoted $\mathrm{Sym}^{\leq k}\mscr V = \bigoplus_{j=1}^k\mathrm{Sym}^j\mscr V$.

    Similarly, we can form the sequence of $k$th exterior powers $\bigwedge^k\mscr V$.
    
    \item[(Moments)] The sequence of \emph{moment matrices of order $k$ over $\mscr V$} is $\mathrm{Sym}^2(\mathrm{Sym}^{\leq k}\mscr V)$. 
    Its elements can be viewed as symmetric matrices whose rows and columns are indexed by monomials of degree at most $k$ in basis elements for $\mscr V$. 

    \item[(Images and Kernels)] If $\mscr A=\{A_n\in\mc L(\vct V_n,\vct U_n)^{\msf{G}_n}\}$ is a morphism mapping $\mscr V\to \mscr U$, then the images $\mathrm{Im}\mscr A = \{(A_n(\vct V_n),\psi_n)\}$ and kernels $\mathrm{ker}\mscr A = \{(\ker A_n,\varphi_n)\}$ form consistent sequences.
\end{enumerate}
If $\mscr V,\mscr U$ are $\mscr V_0$-modules for some common consistent sequence $\mscr V_0$, then all the above are $\mscr V_0$-modules as well.


The group actions above are given in Section~\ref{sec:intro_notation}.
The following theorem gives a calculus that allows us to bound the presentation degrees of sequences constructed from certain simpler ones with known presentation degrees.
The following theorem is a consequence of results in~\cite{FImods,WILSON2014269,GADISH2017450} concerning calculus for generation degrees. We combine these results to obtain the following analogous calculus for presentation degrees, whose proof is given in Appendix~\ref{sec:cat_reps}. 
\begin{theorem}[Calculus for generation and presentation degrees]\label{thm:calc_for_pres_degs_FIW}
    Let $\mscr V$ be a consistent sequence of $\{\msf{G}_n\}$-representations and let $\mscr W,\mscr U$ be $\mscr V$-modules generated in degrees $d_W,d_U$ and presented in degrees $k_W,k_U$, respectively. Then
    \begin{enumerate}[align=left, font=\emph]
        \item[(Sums)] $\mscr W\oplus \mscr U$ is generated in degree $\max\{d_W,d_U\}$ and presented in degree $\max\{k_W,k_U\}$. 
        \item[(Images and kernels)] \new{If $\mscr A\colon\mscr W\to\mscr U$ is a morphism consisting of linear maps $\mscr A=\{A_n\in\mc L(\vct W_n,\vct U_n)^{\msf G_n}\}$, and if the sequence of their adjoints $\mscr A^*=\{A_n^*\in\mc L(\vct U_n,\vct W_n)^{\msf G_n}\}$ is a morphism $\mscr U\to\mscr W$}, then $\mathrm{im}\mscr A$ and $\ker\mscr A$ are generated in degree $d_W$ and presented in degree $k_W$.
    \end{enumerate}
    If $\mscr V=\{\RR^n\}$ with $\msf{G}_n=\msf B_n,\msf D_n,$ or $\msf S_n$ as in Example~\ref{ex:basic_consist_seq}, we further have
    \begin{enumerate}[align=left, font=\emph]
        \item[(Tensors)] $\mscr W\otimes\mscr U$ is generated in degree $d_W+d_U$ and presented in degree $\max\{k_W+d_U, k_U + d_W\}$.
        \item[(Sym and $\bigwedge$)] $\mathrm{Sym}^{\ell}\mscr W,\ \bigwedge^{\ell}\mscr W$ are generated in degree $\ell d_W$ and presented in degree $(\ell-1)d_W+k_W$.
    \end{enumerate}
\end{theorem}
\begin{proof}
    This follows from Theorem~\ref{thm:calculus_for_degrees} in the appendix.
\end{proof}
\begin{example}
    Suppose $\{\RR^n\}$ as in Example~\ref{ex:basic_consist_seq} with $\msf{G}_n=\msf B_n,\msf D_n,$ or $\msf S_n$. Then $(\RR^n)^{\otimes k}$ consists of $n\times\cdots\times n$-sized tensors with embeddings by zero padding and $\mathrm{Sym}^k\RR^n$ consists of homogeneous polynomials of degree $k$ in $n$ variables. These are generated in degree $k$ and presented in degree $k$ if $\msf{G}_n=\msf B_n,\msf S_n$ and in degree $k+1$ if $\msf{G}_n=\msf D_n$. 
    
\end{example}
\begin{remark}[Other groups]
    Note that the last two conclusions in Theorem~\ref{thm:calc_for_pres_degs_FIW} fail without the restriction to $\msf{G}_n=\mathrm{B}_n,\mathrm{D}_n$, or $\mathrm{S}_n$. For example, \new{consider the sequence of cyclic groups $\msf G_n=\mathrm{Cyc}_{2^n}$ embedded as usual in each other. Let $\mscr V=\{\vct V_n=\RR^{2^n}\}$ be the consistent sequence with embeddings $\varphi_n(x)=x\otimes [1,0]^\top$ and the standard action of $\msf G_n=\mathrm{Cyc}_{2^n}$ by cyclically shifting coordinates. Then $\mscr V$ is generated in degree 1 but $\mscr V^{\otimes 2}$ is not finitely-generated since $\dim\mc L(\RR^{2^n})^{\mathrm{Cyc}_{2^n}}=2^n$ does not stabilize. }
\end{remark}

\subsection{Permutation Modules}\label{sec:perm_mods}
We introduce a class of particularly simple consistent sequences on which the group acts by permuting basis elements. These consistent sequences arise in our study of relative entropy cones and their constant-sized descriptions in Section~\ref{sec:rel_ent_cones}. If a group $\msf G$ acts on a (finite) set $\mc A$, define $\RR^{\mc A}=\bigoplus_{\alpha\in\mc A}e_{\alpha}$ to be the vector space with orthonormal basis elements $\{e_{\alpha}\}_{\alpha\in\mc A}$, which is a $\msf G$-representation with action $g\cdot e_{\alpha}=e_{g\cdot \alpha}$. 
\begin{definition}[Permutation modules]\label{def:perm_rep}
    Let $\mscr V = \{\vct V_n\}$ be a consistent sequence of $\{\msf{G}_n\}$-representations. Let $\{\mc A_n\subseteq \vct V_n\}$ be finite $\msf{G}_n$-invariant sets satisfying $\mc A_n\subseteq \mc A_{n+1}$ for all $n$. Then the \emph{permutation $\mscr V$-module} generated by the sets $\{\mc A_n\}$ is the $\mscr V$-module $\{\RR^{\mc A_n}\}_n$.
\end{definition}
Permutation modules can be analyzed in terms of the orbits in the sets $\mc A_n$. In particular, indicators of orbits form a basis for the space of invariants in a permutation module.
\begin{proposition}\label{prop:pres_deg_perm_mod}
Let $\mscr V=\{\vct V_n\}$ be a consistent sequence of $\{\msf{G}_n\}$-representations, let $\{\mc A_n\subseteq \vct V_n\}$ be a nested sequence of finite group-invariant sets, and let $\mscr U = \{\vct U_n=\RR^{\mc A_n}\}$ be the associated permutation $\mscr V$-module. 
\begin{enumerate}[(a)]
    \item $\mscr U$ is generated in degree $d$ if and only if $\mc A_n=\bigcup_{g\in \msf{G}_n}g\mc A_d$ for all $n\geq d$.
    
    \item The projections $\mc P_{\vct U_n}\colon \vct U_{n+1}^{\msf{G}_{n+1}}\to \vct U_n^{\msf{G}_n}$ are isomorphisms for all $n\geq d$ if and only if \new{either of the equivalent conditions in } (a) holds and the number of orbits in $\mc A_n$, which equals $\dim \vct U_n^{\msf{G}_n}$, is constant for all $n\geq d$.
   
    \item Suppose $\mscr V = \{\RR^n\}$ and $\msf{G}_n=\msf B_n,\msf D_n,$ or $\msf S_n$ as in Example~\ref{ex:basic_consist_seq} and $\mscr U$ is generated in degree $d$. Then $\mscr U$ is free if $\msf{G}_n=\msf B_n,\msf S_n$, and agrees with a free module starting from degree $d+1$ if $\msf{G}_n=\msf D_n$.
\end{enumerate}
\end{proposition}
Here we say $\mscr U$ agrees with a free module starting from degree $d$ if there is a free $\mscr V$-module $\mscr F=\{\mbb F_n\}$ and a morphism of sequences $\mscr F\to \mscr U$ such that $\vct F_n\to\vct U_n$ is an isomorphism for all $n\geq d$.
\begin{proof}
\begin{enumerate}[(a)]
    \item We have $\RR[\msf{G}_n]\vct V_d = \sum_{\substack{\alpha\in\mc A_d\\ g\in \msf{G}_n}}\RR e_{g\alpha} = \bigoplus\limits_{\alpha\in \bigcup_{g\in \msf{G}_n}g\mc A_d}\RR e_{\alpha}$ which equals $\vct V_n=\bigoplus_{\alpha\in\mc A_n}\RR e_{\alpha}$ precisely under the stated condition.
    
    \item \new{Let $\mc O_1^{(n)},\ldots,\mc O_{D(n)}^{(n)}$ be the $\msf G_n$-orbits in $\mc A_n$ and set $\mc O_i^{(n+1)}=\msf G_{n+1}\mc O_i^{(n)}$ for $i\in[D(n)]$. 
    Note that $\vct U_n^{\msf G_n}\cong\RR^{D(n)}$ has a basis consisting of orbit indicators $\mathbbm{1}_{\mc O_i^{(n)}}=\sum_{\alpha\in\mc O_i^{(n)}}e_{\alpha}$. 
    Moreover, $\mc P_{\vct U_n}\mathbbm{1}_{\mc O_i^{(n+1)}}=\mathbbm{1}_{\mc O_i^{(n+1)}\cap\mc A_n}=\sum_{\alpha\in\mc O_i^{(n+1)}\cap\mc A_n}e_{\alpha}$. 
    If $\mc A_{n+1}=\bigcup_{g\in\msf G_{n+1}}g\mc A_n$ and the number of orbits in $\mc A_n$ and $\mc A_{n+1}$ are equal, then $\mc O_i^{(n+1)}$ are precisely the orbits in $\mc A_{n+1}$. 
    This implies that $\mc O_i^{(n+1)}\cap\mc A_n=\mc O_i^{(n)}$ for $i\in[D(n)]$, and hence thta $\mc P_{\vct U_n}$ is an isomorphism of invariants. Indeed, the inclusion $\supseteq$ is trivial while if $\alpha\in\left(\mc O_i^{(n+1)}\cap\mc A_n\right)\setminus\mc O_i^{(n)}$ then $\alpha\in\mc O_j^{(n)}$ for some $j\neq i$ and $\msf G_{n+1}\mc O_i^{(n)}=\msf G_{n+1}\mc O_j^{(n)}$ since these two $\msf G_{n+1}$-orbits are not disjoint (they both contain $\alpha$), and hence must be equal. This contradicts the fact that $\mc O_i^{(n+1)}$ are distinct. 
    Conversely, suppose $\mc P_{\vct U_n}$ is an isomorphism of invariants. If $\alpha\in\mc A_{n+1}\setminus\bigcup_{g\in\msf G_{n+1}}g\mc A_n$ then $\msf G_{n+1}\alpha\cap\mc A_n=\emptyset$ and $\mc P_{\vct U_n}\mathbbm{1}_{\msf G_{n+1}\alpha}=0$, a contradiction. Therefore, we have $\mc A_{n+1}=\bigcup_{g\in\msf G_{n+1}}g\mc A_n$. Furthermore, the number of orbits in $\mc A_n$ and $\mc A_{n+1}$ are equal, since these are the dimensions of $\vct U_n^{\msf G_n}$ and $\vct U_{n+1}^{\msf G_{n+1}}$, respectively. }
    
    
    \item Let $d_0=0$ if $\msf{G}_n=\msf S_n,\msf B_n$ and $d_0=1$ if $\msf{G}_n=\msf D_n$. Let $\hat{\mc A}\subseteq \mc A_d$ be a set of minimal degree $\msf{G}_{d+d_0}$-orbit representatives in $\mc A_{d+d_0}$. We argue that these are also the $\msf{G}_n$-orbit representatives of $\mc A_n$ for all $n\geq d+d_0$. 
    Indeed, since $\mc A_n=\bigcup_{g\in \msf{G}_n}g\mc A_d = \bigcup_{g\in \msf{G}_n}g\hat{\mc A}$, it suffices to show that distinct elements in $\hat{\mc A}$ lie in distinct $\msf{G}_n$-orbits.
    To that end, observe that $\mscr V$ satisfies
    \begin{equation}\label{eq:perm_rep_free_cond2}
        g\cdot \alpha\in \vct V_d \textrm{ for } \alpha\in \vct V_d,\ g\in \msf{G}_n \implies \exists \tilde g\in \msf{G}_{d+d_0} \textrm{ s.t. } g\cdot \alpha = \tilde g\cdot \alpha,
    \end{equation}
    since if $\msf{G}_n=\msf S_n,\msf B_n$ define $\tilde g$ to act as $g$ on the coordinates $I=\{i\in[d]:\alpha_i\neq0\}$ and act trivially on the others, and if $\msf{G}_n=\msf D_n$ and if $g$ flips oddly-many signs of coordinates in $I$ then in addition have $\tilde g$ flip the sign of coordinate $d+1$.
    Therefore, if $\alpha,\alpha'\in\mc A_d$ lie in the same $\msf{G}_n$-orbit for $n>d+d_0$, then they also lie in the same $\msf{G}_{d+d_0}$-orbit. 
    Thus, we have $\vct U_n=\bigoplus_{\alpha\in\hat{\mc A}}\RR[\msf{G}_n]e_{\alpha}$ for all $n\geq d+d_0$.

    Next, we argue that $\mscr U$ is free or agrees with a free module starting from degree $d+d_0$. Observe that $\mscr V$ satisfies the additional property 
    \begin{equation}\label{eq:perm_rep_free_cond}
        \alpha \in \vct V_d\setminus \vct V_{d-1} \textrm{ has min.\ degree in its orbit} \implies \msf{Stab}_{\msf{G}_n}(\alpha) = \msf{Stab}_{\msf{G}_d}(\alpha)\msf{H}_{n,d},
    \end{equation}
    Indeed, if $\alpha\in \vct V_d\setminus \vct V_{d-1}$ has minimal degree then all its $d$ entries are nonzero, hence any $g\in \msf{G}_n$ fixing $\alpha$ must fix the first $d$ coordinates.
    Therefore, if $\alpha\in \vct V_d\setminus \vct V_{d-1}$ has minimal degree and $g_1,g_2,\ldots,g_M$ are coset representatives for $\msf{G}_n/\msf{Stab}_{\msf{G}_d}(\alpha)\msf{H}_{n,d}$, then
    \begin{equation*}
        \RR[\msf{G}_n]e_{\alpha} = \bigoplus_{m=1}^M\RR e_{g_m\cdot \alpha} = \bigoplus_{m=1}^Mg_m\cdot \RR e_{\alpha} = \mathrm{Ind}_{\msf{Stab}_{\msf{G}_d}(\alpha)\msf{H}_{n,d}}^{\msf{G}_n}(\RR e_{\alpha}) = \mathrm{Ind}_{\msf{G}_d\msf{H}_{n,d}}^{\msf{G}_n}\left(\mathrm{Ind}_{\msf{Stab}_{\msf{G}_d}(\alpha)}^{\msf{G}_d}\RR e_{\alpha}\right),
    \end{equation*}
    where the first equality follows by~\eqref{eq:perm_rep_free_cond}, the second by the definition of a permutation representation (Definition~\ref{def:perm_rep}), and the third equality follows by~\eqref{eq:ind_def}, and the last equality follows because $\msf{G}_d/\msf{Stab}_{\msf{G}_d}(\alpha)\cong (\msf{G}_d\msf{H}_{n,d})/(\msf{Stab}_{\msf{G}_d}(\alpha)\msf{H}_{n,d})$.
    Thus, if $\alpha\in\hat{\mc A}$ has degree $d_{\alpha}$ and we define $\vct W_{\alpha}=\mathrm{Ind}_{\msf{Stab}_{\msf{G}_{d_{\alpha}}}(\alpha)}^{\msf{G}_{d_{\alpha}}}\RR e_{\alpha}$, then the map $\bigoplus_{\alpha\in\hat{\mc A}}\mathrm{Ind}_{\msf{G}_{d_{\alpha}}}(\vct W_{\alpha})_n\to \vct U_n$ sending $e_{\alpha}$ to itself is an isomorphism for all $n\geq d+d_0$. 
    Condition~\ref{eq:perm_rep_free_cond2} shows that $\hat{\mc A}\cap\mc A_j$ is a set of minimal degree $\msf{G}_j$-orbit representatives for each $j\leq d$ if $\msf{G}_n=\msf B_n,\msf S_n$, hence the above map is an isomorphism for all $n$.
    \qedhere
\end{enumerate}
\end{proof}
We remark that proposition~\ref{prop:pres_deg_perm_mod}(c) applies to any $\mscr V$ satisfying~\eqref{eq:perm_rep_free_cond2} and~\eqref{eq:perm_rep_free_cond}.  

\subsection{Uniform Representation Stability}\label{sec:uniform_rep_stab}
Proposition~\ref{prop:gen_imp_surj} shows that $\dim \vct V_n^{\msf{G}_n}$ stabilizes whenever $\{\vct V_n\}$ is a finitely-generated consistent sequence of $\{\msf{G}_n\}$-representations.
In fact, the theory of~\cite{FImods,WILSON2014269,sam_snowden_2015} and others shows that for many standard families $\{\msf{G}_n\}$ of groups, the entire decomposition of $\vct V_n$ into irreducibles stabilizes.
This phenomenon was called \emph{uniform representation stability} in~\cite{CHURCH2013250}, and we use it to derive constant-sized descriptions for many sequences of PSD cones in Section~\ref{sec:psd_cone_const_size}.
The following is a concrete instance of this phenomenon that we shall use there.
\begin{theorem}[{\cite[Thm.~1.13]{FImods},\cite[Thm.~4.27]{WILSON2014269}}]\label{thm:rep_stability_FIW}
    Let $\mscr V_0=\{\RR^n\}$ with $\msf{G}_n=\msf B_n,\msf D_n,$ or $\msf S_n$ be the consistent sequence from Example~\ref{ex:basic_consist_seq} and let $\mscr V = \{\vct V_n\}$ be a $\mscr V_0$-module generated in degree $d$ and presented in degree $k$. Then there exists a finite set $\Lambda$ and integers $(m_{\lambda})_{\lambda\in\Lambda}$, together with an assignment $\lambda\mapsto \vct W_{\lambda[n]}$ of a distinct $\msf{G}_n$-irreducible $\vct W_{\lambda[n]}$ to each $\lambda\in\Lambda$ for $n\geq k+d$ such that $\vct V_n \cong \bigoplus_{\lambda\in\Lambda}\vct W_{\lambda[n]}^{m_{\lambda}}$ as $\msf{G}_n$-representations.
\end{theorem}
\new{More precisely, there is a standard indexing for irreducibles of $\msf S_n, \msf B_n$, and $\msf D_n$ in terms of partitions of $n$, and a particular way to pad these partitions to form $\lambda[n]$, see~\cite[\S2]{WILSON2014269}.}
\begin{proof}
    The irreducibles of the groups $\msf S_n,\msf D_n,\msf B_n$ are indexed as in~\cite[\S2.1]{WILSON2014269}, and the consistent labelling of irreducibles for different $n$ is given in~\cite[\S2.2]{WILSON2014269}. Under this labelling, the $\mscr V_0$-module $\mscr V$ is uniformly representation stable with stable range $n\geq k+d$ by~\cite[Thms.~4.4, 4.27]{WILSON2014269}, which precisely says that the set of irreducibles appearing in the decomposition of the $\vct V_n$ and their multiplicities become constant for $n\geq k+d$ by~\cite[Def.~2.6]{CHURCH2013250}. 
\end{proof}
\begin{example}\label{ex:Rn_decomp_stab}
    Irreducibles of $\msf S_n$ are indexed by partitions of $n$. If $\lambda_1[n]=(n)$ is the trivial partition and $\lambda_2[n]=(n-1,1)$, then $\RR^n=\vct W_{\lambda_1[n]}\oplus \vct W_{\lambda_2[n]}$ for all $n\geq 1$, where $\vct W_{\lambda_1[n]}=\mathrm{span}\{\mathbbm{1}_n\}$ and $\vct W_{\lambda_2[n]}=\{x\in\RR^n:\mathbbm{1}_n^\top x=0\}$ are distinct irreducible representations of $\msf S_n$.
\end{example}

\subsection{Stabilization of Shifted Sequences}\label{sec:shifted_consistent_seqs}
Many of the phenomena in the representation stability literature, including Theorem~\ref{thm:rep_stability_FIW}, can be derived from properties of \emph{shifted} consistent sequences, which are sequences with group actions restricted to centralizing subgroups, see~\cite{gan2019inductive} and references therein.
In particular, we shall need the following result for such shifted sequences to derive constant-sized descriptions for relative entropy cones in Section~\ref{sec:rel_ent_cones}.
\begin{proposition}[{\cite[Lemma~4.19]{WILSON2014269}}]\label{prop:stab_deg}
    Let $\mscr V = \{\RR^n\}$ and $\msf{G}_n=\msf B_n,\msf D_n$, or $\msf S_n$ as in Example~\ref{ex:basic_consist_seq} and let $\mscr U$ be a $\mscr V$-module presented in degree $k$. If $\{\msf{H}_{n,d}\}$ are the centralizing subgroups of $\mscr V$ and $\ell\in\NN$, then the projections $\mc P_{\vct U_n}\colon \vct U_{n+1}^{\msf{H}_{n+1,\ell}}\to \vct U_n^{\msf{H}_{n,\ell}}$ are isomorphisms for all $n\geq \ell + k$. 
\end{proposition}
\begin{corollary}\label{cor:stab_deg_for_stabs}
    Let $\mscr V = \{\RR^n\}$ and $\msf{G}_n=\msf B_n,\msf D_n$, or $\msf S_n$ as in Example~\ref{ex:basic_consist_seq}, and let $\mscr U$ be a $\mscr V$-module presented in degree $k$. If $\beta\in \RR^d\setminus\RR^{d-1}$ has minimal degree in its $\msf{G}_d$-orbit, then the projections $\mc P_{\vct U_n}\colon \vct U_{n+1}^{\msf{Stab}_{\msf{G}_{n+1}}(\beta)}\to \vct U_n^{\msf{Stab}_{\msf{G}_n}(\beta)}$ are isomorphisms for all $n\geq d + k$. 
\end{corollary}
\begin{proof}
    As shown in~\eqref{eq:perm_rep_free_cond}, we have $\msf{Stab}_{\msf{G}_n}(\beta)=\msf{Stab}_{\msf{G}_d}(\beta) \msf{H}_{n,d}$ for all $n\geq d$, hence 
    \begin{equation*}
        \vct U_n^{\msf{Stab}_{\msf{G}_n}(\beta)}=\vct U_n^{\msf{H}_{n,d}}\cap \vct U_n^{\msf{Stab}_{\msf{G}_d}(\beta)}.
    \end{equation*}
    By Proposition~\ref{prop:stab_deg}, the projections $\mc P_{\vct U_n}\colon \vct U_{n+1}^{\msf{H}_{n+1,d}}\to \vct U_n^{\msf{H}_{n,d}}$ are isomorphisms for all $n\geq d+k$. It thus suffices to show that if $u\in \vct U_{n+1}^{\msf{H}_{n+1,d}}$ satisfies $\mc P_{\vct U_n}u\in \vct U_n^{\msf{Stab}_{\msf{G}_d}(\beta)}$, then $u\in \vct U_{n+1}^{\msf{Stab}_{\msf{G}_d}(\beta)}$. For any such $u$ and $g\in \msf{Stab}_{\msf{G}_d}(\beta)$ we have $g\cdot u\in \vct U_{n+1}^{\msf{H}_{n+1,d}}$ because $\msf{H}_{n+1,d}$ and $\msf{Stab}_{\msf{G}_d}(\beta)\subseteq \msf{G}_d$ commute for our specific $\mscr V$. As $\mc P_{\vct U_n}(u-g\cdot u)=0$ and $\mc P_{\vct U_n}$ is injective on ${\vct U_{n+1}^{\msf{H}_{n+1,d}}}$, we get $u=g\cdot u$. 
\end{proof}



\subsection{Limits of Consistent Sequences}\label{sec:lims_of_consist_seqs}
Lastly, we consider limits of consistent sequences. We shall use the following definitions and results to describe limits of freely-described sequences of convex sets in Section~\ref{sec:lim_of_sets}.
We define limits of consistent sequences, and interpret our definitions in terms of these limits.
\begin{definition}
    For a consistent sequence $\mscr V=\{\vct V_n\}$ of $\{\msf{G}_n\}$-representations, define its \emph{limiting representation} as the vector space $\vct V_{\infty}=\bigcup_n\vct V_n$, viewed as a representation of $\msf{G}_{\infty}=\bigcup_n\msf{G}_n$.
\end{definition}
There is an approach to representation stability studying limiting representations of limiting groups as above, instead of representations of categories as in Appendix~\ref{sec:cat_reps}. 
For example, the authors of~\cite{sam_snowden_2015} analyze representations of five standard infinite groups, including $\msf O_{\infty},\msf S_{\infty}$, that occur as quotients or subrepresentations of tensor powers of $\RR^{\infty}$ and its dual. 

We remark that freely-described elements and morphisms of sequences can equivalently be defined in terms of limits. Indeed, given two consistent sequences $\{\vct V_n\},\{\vct U_n\}$ of $\{\msf{G}_n\}$-representations, a sequence of equivariant linear maps $\{A_n\in\mc L(\vct V_n,\vct U_n)^{\msf{G}_n}\}$ is a morphism of sequences if and only if there exists $A_{\infty}\in\mc L(\vct V_{\infty},\vct U_{\infty})^{\msf{G}_{\infty}}$ satisfying $A_{\infty}|_{\vct V_n} = A_n$ for all $n$, i.e., iff $\{A_n\}$ extends to the limit. 
Similarly, a sequence of invariants $\{v_n\in \vct V_n^{\msf{G}_n}\}$ defines a sequence of invariant linear functionals $\ell_n(x)=\langle v_n,x\rangle\colon \vct V_n\to\RR$, and these extend to a $\msf{G}_{\infty}$-invariant linear functional on $\vct V_{\infty}$ if and only if $\{v_n\}$ is a freely-described element. Every invariant functional on $\vct V_{\infty}$ arises in this way, so freely-described elements are in one-to-one correspondence with invariant linear functionals on the limit of a consistent sequence.  

We also consider continuous limits of consistent sequences by taking the completion of $\vct V_{\infty}$ with respect to some norm.
It is then natural to consider sequences of linear maps which extend to \emph{continuous} maps between these limits.
The inner products on each $\vct V_n$ extend to the limit $\vct V_{\infty}$, so we can always complete with respect to the induced norm and obtain a Hilbert space. 
For many of the examples we consider, however, we obtain more meaningful completions with respect to other norms. 
For the purpose of obtaining limiting descriptions of convex sets in this completion, we consider the following class of norms.
\begin{definition}[Continuous limits]
    \label{def:cont_lims}
    Let $\mscr V = \{\vct V_n\}$ be a consistent sequence of $\{\msf{G}_n\}$-representations, let $\vct V_{\infty}=\bigcup_n\vct V_n$, and let $\mc P_n\colon \vct V_{\infty}\to \vct V_n$ be the orthogonal projection. Fix a norm $\|\cdot\|_c$ on $\vct V_{\infty}$ (not necessarily induced by the inner product) that satisfies $\|\mc P_nx\|_c \leq C\|x\|_c$ for some $C>0$ and  all $n$ and $x\in \vct V_{\infty}$. 
    We call the completion $\overline{\vct V_{\infty}}$ of $\vct V_{\infty}$ with respect to $\|\cdot\|_c$ a \emph{continuous limit} of the sequence $\mscr V$.

    Let $\{\vct V_n\}$ and $\{\vct U_n\}$ be consistent sequences of $\{\msf G_n\}$-representations with associated continuous limits $\overline{\vct V_{\infty}}$ and $\overline{\vct U_{\infty}}$.  A sequence of maps $\{A_n\colon \vct V_n\to \vct U_n\}$ \emph{extends continuously to the limit} if there exists a bounded linear operator $\overline{A_{\infty}}\colon \overline{\vct V_{\infty}}\to\overline{\vct U_{\infty}}$ such that $\overline{A_{\infty}}|_{\vct V_n}=A_n$ for all $n$. 
    A freely-described element $\{u_n\in \vct U_n^{\msf{G}_n}\}$ \emph{extends continuously to the limit} if there exists $u_{\infty}\in \overline{\vct U_{\infty}}$ satisfying $\mc P_nu_{\infty}=u_n$ for all $n$. 
\end{definition}
Note that a morphism of sequences $\{A_n\colon \vct V_n\to \vct U_n\}$ extends to the continuous limit if and only if the sequence of operator norms $\{\|A_n\|_{\mathrm{op},c}\}$ with respect to $\|\cdot\|_c$ is bounded.

Our condition on the norm $\|\cdot\|_c$ ensures that each projection $\mc P_n$ extends to a bounded linear map on the continuous limit, and that the $\mc P_n$ converge to the identity in the strong operator topology, as we show in the following lemma.
\begin{lemma}\label{lem:projections_converge}
    In the notation of Definition~\ref{def:cont_lims}, there exists $C>0$ such that $\|\mc P_nx\|_c\leq C\|x\|_c$ for all $n$ and all $x\in\overline{\vct V_{\infty}}$ if and only if $\mc P_n$ extends to a bounded linear map $\overline{\vct V_{\infty}}\to \vct V_n$ for all $n$ and $\|x-\mc P_nx\|_c\to 0$ as $n\to\infty$ for all $x\in \overline{\vct V_{\infty}}$. 
\end{lemma}
\begin{proof}
    Suppose $\|\mc P_nx\|_c\leq C\|x\|_c$ for all $n$ and $x$. Then $\|\mc P_n\|_{\mathrm{op},c}\leq C$ so that $\mc P_n$ extends to a bounded linear map on $\overline{\vct V_{\infty}}$ for all $n$. Furthermore, for each $x\in \overline{\vct V_{\infty}}$ there exists $x_n\in \vct V_n$ such that $\|x-x_n\|_c\to 0$. Since $x_n = \mc P_nx_n$, we have
    \begin{equation*}
        \|x-\mc P_nx\|_c \leq \|x-x_n\|_c + \|\mc P_nx_n - \mc P_nx\|_c\leq (1+C)\|x-x_n\|_c \to 0.
    \end{equation*}
    Conversely, if $\|\mc P_n\|_{\mathrm{op},c}<\infty$ for all $n$ and $\mc P_n$ converge strongly to the identity, then for any $x\in\overline{\vct V_{\infty}}$ we have $\|\mc P_nx\|_c<\infty$ for all $n$ and $\|\mc P_nx\|_c\to \|x\|_c$ as $n\to\infty$, hence $\sup_n\|\mc P_nx\|_c<\infty$. By the uniform boundedness principle, we have $\sup_n\|\mc P_n\|_{\mathrm{op},c}=C<\infty$.
\end{proof}

We close this section by noting that, more generally, if a sequence of convex functions $\{f_n\colon\vct V_n\to\RR\cup\{\infty\}\}$ is intersection-compatible, then it extends to $f_{\infty}\colon \vct V_{\infty}\to\RR\cup\{\infty\}$. By taking the closure of its epigraph it can further be extended to $\overline{f_{\infty}}\colon \overline{\vct V_{\infty}}\to\RR\cup\{\infty\}$, in which case $\overline{\gamma_{C_{\infty}}}=\gamma_{\overline{C_{\infty}}}$ and $\overline{h_{C_{\infty}}}=h_{\overline{C_{\infty}}}$. Thus, the conic descriptions we obtain in Theorem~\ref{thm:descriptions_of_lims} below for continuous limits of convex sets yield conic descriptions for limits of functions as well.

\section{Structural Results on \new{Dimension-}Free Descriptions}\label{sec:free_and_compatible}

With the relevant background from representation stability in hand, we now proceed to prove the structural results discussed in Section~\ref{sec:intro_structural} pertaining to freely-described convex sets.
We leverage these results to provide a range of examples of \new{dimension-}free descriptions arising in various applications in Section~\ref{sec:examples}.
We begin with the following instructive example, which shows that freely-described sequences of convex sets need not satisfy either of our compatibility conditions. It also shows that a fixed convex set can extend to different freely-described sequences, depending on its conic description.
\begin{example}[(In)compatible sequence of hypercubes]\label{ex:bad_cube}
In this example, we construct two freely-described sequences of hypercubes with respect to the same description spaces, one of which is both intersection and projection compatible and the other is neither.
Let $\mscr V = \{\RR^n\}$ with embeddings by zero-padding and the action of $\msf G_n=\msf B_n$ as in Example~\ref{ex:basic_consist_seq}. Let $s=\diag(-1,1,\ldots,1)\in\msf B_n$, which together with $\msf S_n$ generates $\msf B_n$. 

We define two consistent sequences that will be used to construct our description spaces. Let $\mscr W^{(1)}=\{\RR^n\}$ with embeddings by zero-padding but on which $s$ acts trivially and $\msf S_n$ acts as usual. Let $\mscr W^{(2)} = \{\RR^{n}\oplus\RR^n\}$ with embeddings by zero-padding each of the two summands, on which $\msf B_n$ acts as follows. We set $s\cdot(x,y)=([y_1,x_2,\ldots,x_n]^\top,[x_1,y_2,\ldots,y_n]^\top)$ and $\sigma\cdot(x,y)=(\sigma x,\sigma y)$ for $\sigma\in\msf S_n$.
We set our description spaces to
\begin{equation*}
    \mscr U = \mscr W^{(1)}\oplus \mscr W^{(2)}\oplus\RR = \{(\RR^n)^{\oplus 3}\oplus \RR\},\quad \mscr K = \{0\oplus (\RR^n_+)^{\oplus 2}\oplus\RR_+\}, \quad \mscr W=\mscr W^{(1)}\oplus\RR=\{\RR^n\oplus\RR\}.
\end{equation*}
Consider the following two sequences of convex sets, both freely-described with respect to these description spaces. 
The first is the sequence of unit hypercubes $\{\cvx C_n^{(1)} = [-1,1]^n\}$ given by~\eqref{eq:preim_seq} with
\begin{equation}\label{eq:normal_cube}
    A_n^{(1)}x = (0, x, -x, 0),\qquad B_n^{(1)}(y,\beta) = (y, y, y, 0),\qquad u_n^{(1)} = (-\mathbbm{1}_n, 0, 0, 0),
\end{equation}
where $x,y\in\RR^n$ and $\beta\in\RR$. 
The second is the sequence of scaled hypercubes $\{\cvx C_n^{(2)}=\frac{3}{2n-1}[-1,1]^n\}$ given by~\eqref{eq:preim_seq} with
\begin{equation}\label{eq:scaled_cube}
    A_n^{(2)} = A_n^{(1)},\qquad B_n^{(2)}(y,\beta) = (y-\mathbbm{1}_n\mathbbm{1}_n^\top y + \beta\mathbbm{1}_n, y, y, -\mathbbm{1}_n^\top y - \beta),\qquad u_n = (0,0,0,3).
\end{equation}
It is straightforward to check that both sequences are freely-described. 
In particular, all the relevant spaces of invariants stabilize starting at $n=2$, so that the extension of either description of $[-1,1]^2$ to a freely-described sequence is unique.  However, the two sequences of sets have notably different properties.  The first sequence is both intersection and projection compatible while the second is neither. Furthermore, note that $\cvx C_2^{(1)} = \cvx C_2^{(2)} = [-1,1]^2$ but $\cvx C_n^{(1)}\neq \cvx C_n^{(2)}$ for all $n \neq 2$. Thus, the extension of the unit square $[-1,1]^2$ to a freely-described sequence is not unique, but depends on its conic description.
\end{example}
Motivated by this example, we give conditions on \new{dimension-}free descriptions ensuring that the corresponding sequence of sets satisfies compatibility conditions.  In turn, these yield conditions on a description of a fixed set ensuring that it extends to a freely-described and compatible sequence of sets as well as conditions on the existence of limiting descriptions.


\subsection{\new{Dimension-}Free Descriptions Certifying Compatibility}
Example~\ref{ex:bad_cube} shows that a freely-described sequence of convex sets need not satisfy either compatibility condition in Definition~\ref{def:compatibility_conds}. 
In this section, we therefore give conditions on \new{dimension-}free descriptions ensuring that our compatibility conditions are satisfied.
\begin{proposition}\label{prop:compatible_descriptions}
Let $\{\vct V_n\},\{\vct W_n\},\{\vct U_n\}$ be consistent sequences of $\{\msf G_n\}$-representations, let $\{\cvx K_n\subseteq \vct U_n\}$ be an intersection- and projection-compatible sequence of convex cones, and let $\mscr C=\{\cvx C_n\subseteq \vct V_n\}$ be described by linear maps $\{A_n\colon \vct V_n\to \vct U_n\},\{B_n\colon \vct W_n\to \vct U_n\}$ and elements $\{u_n\in \vct U_n^{\msf G_n}\}$ as in~\eqref{eq:preim_seq}.
    \begin{enumerate}[(a)]
        \item If $\{A_n\},\{B_n\},\{B_n^*\}$ are morphisms, $\{u_n\}$ is freely-described, and $u_{n+1}-u_n\in \cvx K_{n+1}$ for all $n$, then $\mscr C$ is intersection-compatible. If, in addition, $\{A_n^*\}$ is a morphism, then $\mscr C$ is also projection-compatible.

        \item If $\{A_n\},\{A_n^*\}\{B_n\},\{B_n^*\}$ are morphisms, $\{u_n\}$ is freely-described, and $u_{n+1}-u_n\in \cvx K_{n+1} + A_{n+1}(\vct V_n^{\perp}) + B_{n+1}(\vct W_n)$, then $\mscr C$ is projection-compatible.
    \end{enumerate}
\end{proposition}
\begin{proof}
    \begin{enumerate}[(a)]
        \item We show $\cvx C_n\subseteq \cvx C_{n+1}$. If $x\in \cvx C_n$ then there is $y\in \vct W_n$ satisfying $A_nx + B_ny + u_n\in \cvx K_n$. 
        Then
        \begin{equation*}
            A_{n+1}x + B_{n+1}y + u_{n+1} = A_n x + B_n y + u_n + (u_{n+1}-u_n) \in \cvx K_{n+1},
        \end{equation*}
        where we used the facts that $\{A_n\},\{B_n\}$ are morphisms, that $\cvx K_n\subseteq \cvx K_{n+1}$ by intersection-compatiblity, and that $u_{n+1}-u_n\in \cvx K_{n+1}$.
        Thus, $x\in \cvx C_{n+1}$.
        Next, we show $\cvx C_{n+1}\cap \vct V_n\subseteq  \cvx C_n$. If $x\in  \cvx C_{n+1}\cap \vct V_n$, there exists $y\in \vct W_{n+1}$ satisfying $A_{n+1}x + B_{n+1}y + u_{n+1}\in \cvx K_{n+1}$. 
        Because $\{A_n\}$ is a morphism, we have $A_{n+1}x=A_nx$ and hence $\mc P_{\vct U_n}A_{n+1}x = A_nx$. 
        Because $\{B_n^*\}$ is a morphism, we have $\mc P_{\vct U_n}B_{n+1} = B_n\mc P_{\vct W_n}$, hence $\mc P_{\vct U_n}B_{n+1}y=B_{n+1}(\mc P_{\vct W_n}y)$. 
        Finally, we have $\mc P_{\vct U_n}u_{n+1}=u_n$ because $\{u_n\}$ is freely-described, and $\mc P_{\vct U_n}\cvx K_{n+1}\subseteq \cvx K_n$ because $\{\cvx K_n\}$ is projection-compatible. Thus, applying $\mc P_{\vct U_n}$ we obtain $A_nx + B_n(\mc P_{\vct W_n}y) + u_n\in \cvx K_n$, thus yielding that $x\in \cvx C_n$. We conclude that $\mscr C$ is intersection-compatible. 

        Because $\mscr C$ is intersection-compatible, we have $\mc P_{\vct V_n} \cvx C_{n+1}\supseteq  \cvx C_n$. Conversely, if $x\in \cvx C_{n+1}$ then $A_{n+1}x+B_{n+1}y+u_{n+1}\in \cvx K_{n+1}$ for some $y\in \vct W_{n+1}$. If $\{A_n^*\}$ is a morphism, then $\mc P_{\vct U_n}A_{n+1}=A_n\mc P_{\vct V_n}$. Applying $\mc P_{\vct U_n}$ to both sides we obtain $A_n\mc P_{\vct V_n}x + B_n\mc P_{\vct W_n}y + u_n\in \cvx K_n$ and hence $\mc P_{\vct V_n}x\in \cvx C_n$, which implies that $\mscr C$ is projection-compatible.

        \item First, we show $\mc P_{\vct V_n}\cvx C_{n+1}\subseteq \cvx C_n$. If $x\in \cvx C_{n+1}$ then there is $y\in \vct W_{n+1}$ satisfying $A_{n+1}x + B_{n+1}y + u_{n+1}\in \cvx K_{n+1}$. Applying $\mc P_{\vct U_n}$ to both sides and using the facts that $\{A_n^*\},\{B_n^*\}$ are morphisms, that $\{u_n\}$ is freely-described, and that $\{\cvx K_n\}$ is projection-compatible, we obtain $A_n(\mc P_{\vct V_n}x) + B_n(\mc P_{\vct W_n}y) + u_n\in \cvx K_n$, showing that $\mc P_{\vct V_n}x\in \cvx C_n$.
        Second, we show $\cvx C_n\subseteq \mc P_{\vct V_n}\cvx C_{n+1}$. Suppose $A_nx + B_ny+u_n\in \cvx K_n$ for $x\in \vct V_n$. Let $x_{\perp}\in \vct V_n^{\perp}$ and $y'\in \vct W_n$ satisfy $u_{n+1}-u_n + A_{n+1}x_{\perp} + B_{n+1}y' \in \cvx K_{n+1}$. As $\{A_n\},\{B_n\}$ are morphisms, 
        \begin{equation*}
            A_{n+1}(x + x_{\perp}) + B_{n+1}(y + y') + u_{n+1} = A_nx + B_ny + u_n + (u_{n+1}-u_n + A_{n+1}x_{\perp} + B_{n+1}y') \in \cvx K_{n+1},
        \end{equation*}
        hence $x+x_{\perp}\in \cvx C_{n+1}$ and $\mc P_{\vct V_n}(x+x_{\perp})=x$. This shows $\mscr C$ is projection-compatible. 
        \qedhere
    \end{enumerate}
\end{proof}
%
We say that a sequence of conic descriptions \emph{certifies compatibility} when it satisfies the hypotheses of Proposition~\ref{prop:compatible_descriptions}.
\begin{remark}\label{rmk:compatible_descrip}
We make a number of remarks about the conditions in Proposition~\ref{prop:compatible_descriptions}.    
\begin{itemize}

    \item Standard sequences of cones such as nonnegative orthants and PSD cones satisfy both intersection and projection compatibility.

    \item The set of linear maps $\{A_n\}$ and $\{B_n\}$ satisfying the hypotheses of Proposition~\ref{prop:compatible_descriptions}(a) form \emph{linear subspaces} of the corresponding spaces of freely-described elements. The set of sequences $\{u_n\}$ satisfying those hypotheses form a \emph{convex cone}. 
    Similarly, we can parametrize descriptions satisfying the hypotheses of Proposition~\ref{prop:compatible_descriptions}(b) by a convex cone, by considering $\{u_n\}$ of the form $u_n=A_n(v_n) + B_n(w_n) + z_n$ where $\{v_n\in \vct V_n^{\msf G_n}\}$, $\{w_n\in \vct W_n^{\msf G_n}\}$, and $\{z_n\in \cvx K_n^{\msf G_n}\}$ are freely-described.
    
    \item Freely-described sets satisfying the hypotheses of Proposition~\ref{prop:compatible_descriptions}(b) need not be intersection-compatible, as the elliptope in~\eqref{eq:elliptope} studied below demonstrates.
    
    \item \new{Dimension-}free descriptions certifying compatibility often extend to descriptions of infinite-dimensional limits, see Theorem~\ref{thm:descriptions_of_lims} below.
\end{itemize}
\end{remark}

Returning to Example~\ref{ex:bad_cube}, we note that the description~\eqref{eq:normal_cube} satisfies all the hypotheses of Proposition~\ref{prop:compatible_descriptions}(a) and hence certifies the intersection and projection compatibility of the sequence of hypercubes it describes. On the other hand, the description~\eqref{eq:scaled_cube} does not satisfy the hypotheses of either part of the above theorem because neither $\{B_n^{(2)}\}$ nor $\{(B_n^{(2)})^*\}$ are morphisms.

\subsection{Extending a Convex Set to a Freely-Described Sequence}
Let $\mscr V = \{\vct V_n\}$ be a consistent sequence of $\{\msf G_n\}$-representations. 
In this section, we fix $n_0$ and seek to extend a convex subset $\cvx C_{n_0}\subseteq \vct V_{n_0}$ to a freely-described and compatible sequence.
As we saw in Example~\ref{ex:bad_cube}, different extensions may be obtained from different fixed-dimensional conic descriptions of $\cvx C_{n_0}$. 
We therefore further fix consistent sequences $\mscr U=\{\vct U_n\},\mscr W=\{\vct W_n\}$ and cones $\{\cvx K_n\subseteq \vct U_n\}$ satisfying both intersection and projection compatibility, and assume $\cvx C_{n_0}$ is described by~\eqref{eq:preim_seq}. 
We now ask when the description of $\cvx C_{n_0}$ can be extended to a \new{dimension-}free description that yields a compatible sequence of convex sets. 

If $n_0$ exceeds the presentation degrees of $\mscr W,\mscr U,\mscr V\otimes\mscr U$ and $\mscr W\otimes\mscr U$, then Proposition~\ref{prop:presentation_implies_isom_of_canon} shows that we can uniquely extend all the elements in the description of $\cvx C_{n_0}$ to freely-described elements, and hence extend $\cvx C_{n_0}$ to a freely-described sequence. 
To ensure that this unique extension satisfies our compatibility conditions, we give conditions on these invariants guaranteeing that their extensions to freely-described elements satisfy the conditions of Proposition~\ref{prop:compatible_descriptions}.

We begin by characterizing when a fixed equivariant map $A_{n_0}\in\mc L(\vct V_{n_0},\vct U_{n_0})^{\msf G_{n_0}}$ extends to a morphism of sequences $\{A_n\in\mc L(\vct V_n,\vct U_n)^{\msf G_n}\}$. 
\begin{theorem}\label{thm:extending_maps_to_lims}
Let $\mscr V_0$ be a consistent sequence of $\{\msf G_n\}$-representations and let $\mscr V=\{\vct V_n\},\mscr U=\{\vct U_n\}$ be $\mscr V_0$-modules. Assume $\mscr V$ is generated in degree $d$ and presented in degree $k$, and fix $n_0\geq k$. Then $A_{n_0}$ extends to a morphism of sequences if and only if $A_{n_0}\in \mc L(\vct V_{n_0},\vct U_{n_0})^{\msf G_{n_0}}$ satisfies $A_{n_0}(\vct V_j)\subseteq \vct U_j$ for $j\leq d$.
\end{theorem}
\begin{proof}
If $A_{n_0}$ extends to a morphism then $A_{n_0}(\vct V_j) = A_j(\vct V_j)\subseteq \vct U_j$ for all $j\leq n_0$. Conversely, assume $A_{n_0}(\vct V_j)\subseteq \vct U_j$ for $j\leq d$, and let $\{\msf H_{n,d}\}$ be the centralizing subgroups of $\mscr V_0$.
    Suppose first that $\mscr V = \mscr F = \bigoplus_j\mathrm{Ind}_{\msf G_{d_j}}\vct W_{d_j}$ is free. Note that it is generated in degree $\max_jd_j\leq d$.
    
    Let $A_{d_j} = A_{n_0}|_{\vct W_{d_j}}$ and fix $n\geq d_j$.
    Because $\mscr U$ is a $\mscr V_0$-module, we have $\vct U_{d_j}\subseteq \vct U_n^{\msf H_{n,d_j}}$, so we can view $\vct U_{d_j}$ as a representation of $\msf G_{d_j}\msf H_{n,d_j}$ on which $\msf H_{n,d_j}$ acts trivially. 
    As $A_{d_j}(\vct W_{d_j})\subseteq \vct U_{d_j}$ and is $\msf G_{d_j}\msf H_{n,d_j}$-equivariant, the following composition defines an equivariant map
    \begin{equation*}
        A_{n,j}\colon \mathrm{Ind}_{\msf G_{d_j}\msf H_{n,d_j}}^{\msf G_n}(\vct W_{d_j}) \xrightarrow{\mathrm{Ind}(A_{d_j})} \mathrm{Ind}_{\msf G_{d_j}\msf H_{n,d_j}}^{\msf G_n}\vct U_{d_j} \xrightarrow{g\otimes u\mapsto g\cdot u} \vct U_n,
    \end{equation*}
    where the induced map $\mathrm{Ind}(A_{d_j})$ was defined in Section~\ref{sec:intro_notation}.
    Note that $A_{n_0,j}=A_{n_0}|_{\mathrm{Ind}_{\msf G_{d_j}}(\vct W_{d_j})_{n_0}}$, since $A_{n_0,j}(g\otimes w) = g\cdot A_{n_0}w$ for all $g\in \msf G_n$ and $w\in \vct W_{d_j}$. Also, $\{A_{n,j}\}$ defines a morphism $\mathrm{Ind}_{\msf G_{d_j}}(\vct W_{d_j})\to \mscr U$. Therefore, the desired extension of $A_{n_0}$ to a morphism $\{A_n\}$ is given by $A_n = \bigoplus_jA_{n,j}\colon \vct V_n\to \vct U_n$.

    Now suppose $\mscr F=\{\vct F_n\}$ is an algebraically free $\mscr V$-module as above with a surjection $\mscr F\to \mscr V$ whose kernel $\mscr K=\{\vct K_n\}$ is generated in degree $k$. Define the composition 
    \begin{equation*}
        \widetilde A_{n_0}\colon \vct F_{n_0}\to \vct V_{n_0} \xrightarrow{A_{n_0}} \vct U_{n_0},
    \end{equation*}
    which satisfies $\widetilde A_{n_0}(\vct F_j)\subseteq \vct U_j$ for all $j\leq d$ by assumption and $\widetilde A_{n_0}(\vct K_{n_0})=0$ by its definition. By the previous paragraph, it extends to a morphism $\{\widetilde A_n\colon \vct F_n\to \vct U_n\}$. Because $\mscr K$ is generated in degree $k$ and $n_0\geq k$, we have $\vct K_n = \RR[\msf G_n]\vct K_{n_0}$. Because $\widetilde A_n$ is equivariant, we have $\widetilde A_n(\vct K_n)=0$. Therefore, $\widetilde A_n$ can be factored as $\vct F_n\to \vct F_n/\vct K_n = \vct V_n\xrightarrow{A_n} \vct U_n$, where the maps $A_n$ in this factorization give the desired extension of $A_{n_0}$ to a morphism $\mscr V\to \mscr U$.
\end{proof}
To satisfy the conditions in Proposition~\ref{prop:compatible_descriptions}, we also use Theorem~\ref{thm:extending_maps_to_lims} to ensure $\{A_n^*\}$ defines a morphism. To that end, note that $A_{n_0}^*(\vct U_j)\subseteq \vct V_j$ if and only if $A_{n_0}(\vct V_j^{\perp})\subseteq \vct U_j^{\perp}$, where orthogonal complements are taken inside $\vct V_{n_0}$ and $\vct U_{n_0}$.
We can now give conditions guaranteeing extendability of a convex set to a freely-described and compatible sequence. 
\begin{theorem}[Parametrizing freely-described and compatible sequences]\label{thm:extending_compatible_descriptions}
    Let $\mscr V_0$ be a consistent sequence of $\{\msf G_n\}$ representations and let $\mscr V=\{\vct V_n\}$, $\mscr W=\{\vct W_n\}$, and $\mscr U=\{\vct U_n\}$ be $\mscr V_0$-modules generated in degrees $d_V,d_U,d_W$, respectively, and presented in degree $k$. Let $\{\cvx K_n\subseteq \vct U_n\}$ be an intersection and projection-compatible sequence of convex cones. Fix $n_0\geq k$. 

    Let $u_{n_0}\in \vct U_{n_0}^{\msf G_{n_0}}$ and let $A_{n_0}\in\mc L(\vct V_{n_0},\vct U_{n_0})^{\msf G_{n_0}}$, $B_{n_0}\in\mc L(\vct W_{n_0},\vct U_{n_0})^{\msf G_{n_0}}$ such that
    \begin{equation}\label{eq:ext_cond_on_maps}
        A_{n_0}(\vct V_j)\subseteq \vct U_j \textrm{ for } j\leq d_V,\quad B_{n_0}(\vct W_j)\subseteq \vct U_j \textrm{ for } j\leq d_W,\quad B_{n_0}(\vct W_j^{\perp})\subseteq \vct U_j^{\perp} \textrm{ for } j\leq d_U. 
    \end{equation}
    Then there are unique extensions of $A_{n_0}$ and $B_{n_0}$ to morphisms $\{A_n\in\mc L(\vct V_n,\vct U_n)^{\msf G_n}\}$ and $\{B_n\in\mc L(\vct W_n,\vct U_n)^{\msf G_n}\}$, where $\{B_n^*\}$ is a morphism as well. Furthermore, there is a unique extension of $u_{n_0}$ to a freely-described element $\{u_n\in\vct U_n^{\msf G_n}\}$.  
    Let $\mscr C = \{\cvx C_n\}$ be the freely-described sequence of convex sets given by~\eqref{eq:preim_seq}.

    \begin{enumerate}[(a)]
        \item If $u_{n+1}-u_n\in \cvx K_{n+1}$ for all $n$, then $\mscr C$ is intersection-compatible.
        If, in addition, we have $A_{n_0}(\vct V_j^{\perp})\subseteq \vct U_j^{\perp}$ for $j\leq d_U$, then $\mscr C$ is also projection-compatible.
        
        \item If $u_{n+1}-u_n\in \cvx K_{n+1} + A_{n+1}(\vct V_n^{\perp}) + B_{n+1}(\vct U_{n+1})$ for all $n$, then $\mscr C$ is projection-compatible. 
    \end{enumerate}
\end{theorem}
\begin{proof}
    Theorem~\ref{thm:extending_maps_to_lims} shows that $A_{n_0}$ and $B_{n_0}$ uniquely extend to morphisms $\{A_n\},\{B_n\}$ such that $\{B_n^*\}$ is also a morphism. Proposition~\ref{prop:presentation_implies_isom_of_canon} shows that $u_{n_0}$ extends to a freely-described element $\{u_n\}$. Under the stated conditions on these extensions, Proposition~\ref{prop:compatible_descriptions} yields the claimed compatibility conditions for the sequence of convex sets $\{\cvx C_n\}$ given by~\eqref{eq:preim_seq}.
\end{proof}
In Section~\ref{sec:numerical_examples}, we use Theorem~\ref{thm:extending_compatible_descriptions} to computationally parametrize and search over \new{dimension-}free descriptions certifying compatibility.  

\subsection{Limits of Freely-Described Convex Sets}\label{sec:lim_of_sets}
Our last structural result gives conditions under which \new{dimension-}free descriptions extend to descriptions of continuous limits of convex sets. 
The certificates of compatibility in Proposition~\ref{prop:compatible_descriptions} play a major role once again in the existence of these limiting descriptions.
If $\mscr C = \{\cvx C_n\subseteq \vct V_n\}$ is an intersection-compatible sequence of convex subsets of a consistent sequence $\{\vct V_n\}$, define $\cvx C_{\infty} = \bigcup_n\cvx C_n$, which is a convex subset of $\vct V_{\infty}$.
\begin{theorem}[Descriptions of limits]\label{thm:descriptions_of_lims}
    Suppose $\mscr C=\{\cvx C_n\subseteq \vct V_n\}$ is given by~\eqref{eq:preim_seq}. If all the hypotheses of Proposition~\ref{prop:compatible_descriptions}(a) are satisfied (so $\mscr C$ is intersection and projection compatible) and if $\{A_n\},\{B_n\},\{u_n\}$ extend continuously to the limit, then
    \begin{equation}\label{eq:lim_preim}
        \left\{x\in \overline{\vct V_{\infty}}: \exists y\in\overline{\vct W_{\infty}} \textrm{ s.t. } \overline{A_{\infty}}x + \overline{B_{\infty}}y + u_{\infty} \in \overline{\cvx K_{\infty}}\right\},
    \end{equation}
    contains $\cvx C_{\infty}$ and is dense in its closure.
    If $B_{\infty}=0$, then~\eqref{eq:lim_preim} equals $\overline{\cvx C_{\infty}}$.
\end{theorem}
\begin{proof}
        To prove that $\cvx C_{\infty}$ is contained in~\eqref{eq:lim_preim}, observe that $u_{\infty}-u_n = \lim_{N\to\infty}(u_N-u_n)\in \overline{\cvx K_{\infty}}$ for all $n$ by Lemma~\ref{lem:projections_converge}. Therefore, if $x\in \cvx C_n$ and $y\in \vct W_n$ satisfy $A_nx+B_ny+u_n\in \cvx K_n$, then $\overline{A_{\infty}}x + \overline{B_{\infty}}y + u_{\infty} = A_nx + B_ny + u_n + (u_{\infty}-u_n)\in \overline{\cvx K_{\infty}}$, proving that $x$ is in~\eqref{eq:lim_preim}. 
        
        To prove that~\eqref{eq:lim_preim} is contained in $\overline{\cvx C_{\infty}}$, suppose $x\in \overline{\vct V_{\infty}}$ and $y\in\overline{\vct W_{\infty}}$ satisfy $\overline{A_{\infty}}x + \overline{B_{\infty}}y + u_{\infty}\in\overline{\cvx K_{\infty}}$. Because $\{A_n^*\},\{B_n^*\}$ are morphisms and $\{\cvx K_n\}$ is projection-compatible, applying $\mc P_{\vct U_n}$ we obtain $A_n(\mc P_{\vct V_n}x) + B_n(\mc P_{\vct W_n}y) + u_n\in \cvx K_n$, hence $\mc P_{\vct V_n}x\in \cvx C_n$ for all $n$ and $x = \lim_n\mc P_{\vct V_n}x\in\overline{\cvx C_{\infty}}$ by Lemma~\ref{lem:projections_converge}.

        If $B_{\infty}=0$ then~\eqref{eq:lim_preim} is the preimage under the continuous map $x\mapsto \overline{A_{\infty}}x + u_{\infty}$ of the closed cone $\overline{\cvx K_{\infty}}$, hence~\eqref{eq:lim_preim} is closed and must equal $\overline{\cvx C_{\infty}}$.
\end{proof}
If the set~\eqref{eq:lim_preim} is dense in $\overline{\cvx C_{\infty}}$, then optimizing a continuous function over~\eqref{eq:lim_preim} and over $\overline{\cvx C_{\infty}}$ are equivalent, yielding an finitely-parametrized conic program in $\overline{\vct V_{\infty}}$. 

\section{Examples}\label{sec:examples}
In this section, we present examples of freely-described and compatible sequences of sets and functions arising in a variety of applications. For several of these, we derive finitely-parametrized families of freely-described sets, and we also give limiting descriptions and some related consequences.
\subsection{Simplices and Norms}\label{sec:simplex}
Let $\mscr V=\{\vct V_n=\RR^n\}$ with embedding by zero-padding and the action of $\msf G_n=\msf S_n$ as in Example~\ref{ex:basic_consist_seq}.
\paragraph{Simplex:}  The sequence of simplices $\Delta^{n-1}=\{x\in\RR^n:x\geq0,\ \mathbbm{1}_n^\top x=1\}$ is freely-described, as the associated description is of the form~\eqref{eq:preim_seq} with description spaces
   $\mscr U=\mscr V\oplus\RR = \{\vct U_n=\RR^{n+1}\}$, $\mscr W = \{\vct W_n=0\}$ and cones $\mscr K = \{\cvx K_n=\RR^n_+\oplus\{0\}\}$,
and with $A_n = [I_n, -\mathbbm{1}_n]^\top$, $B_n=0$, and $u_n=[0,1]^\top$.
Moreover, the simplices $\{\Delta^{n-1}\}$ are intersection-compatible (but not projection-compatible); the above description satisfies the hypotheses of Proposition~\ref{prop:compatible_descriptions}(a) and hence certifies this compatibility.

\paragraph{Solid simplex:} In contrast, consider the sequence of solid simplices $\Delta^n_s=\{x\in\RR^n:x\geq0, \mathbbm{1}_n^\top x\leq 1\}$. This sequence of conic descriptions is the same as above except that here $\cvx K_n=\RR^{n+1}_+$, and in particular is \new{dimension-}free and certifies intersection compatibility of $\{\Delta^n_s\}$. However, the sequence $\{\Delta^n_s\}$ is also projection-compatible, but the above \new{dimension-}free description of it does \emph{not} certify this compatibility. That is because $\{A_n^*\}$ is not a morphism, hence the hypotheses of the second part of Proposition~\ref{prop:compatible_descriptions}(a) are not satisfied. 
Indeed, the above description of the solid simplex does not make its projection compatibility apparent, since $\mathbbm{1}_n^\top\mc P_{\vct V_n}x$ can be arbitrary compared to $\mathbbm{1}_{n+1}^\top x$ for general $x\in\RR^{n+1}$, but must be smaller if $x\geq 0$. 

Instead, the following is a \new{dimension-}free \emph{semidefinite} description of the solid simplices that certifies both intersection and projection compatibilities:
\begin{equation}\label{eq:sdp_of_simplex}
    \Delta^n_s = \left\{x\in\RR^n:\begin{bmatrix} 1 & x^\top \\ x & \mathrm{diag}(x)\end{bmatrix}=A_nx + u_n\succeq0\right\},
\end{equation}
which is of the form~\eqref{eq:preim_seq} with $\mscr U = \mathrm{Sym}^2(\mathrm{Sym}^{\leq 1}\mscr V)=\{\mbb S^{n+1}\}$, $\mscr K = \{\mbb S^{n+1}_+\}$, and $\mscr W=0$. Here both $\{A_n\}$ and $\{A_n^*\}$ are morphisms.
It would be interesting to find a \new{dimension-}free linear programming description of the solid simplices which certifies both compatibilities, or to show that none exists.

\paragraph{Norm balls:} Several related \new{dimension-}free descriptions are derived from the above. The sequence of $\ell_1$ unit balls $\cvx B_{\ell_1}^n=\{x\in\RR^n:\|x\|_1\leq 1\}$ is intersection and projection compatible. It can be written as $\cvx B_{\ell_1}^n=\{x\in\RR^n:\exists y\in\Delta^n_s \textrm{ s.t. } -y\leq x\leq y\}$, which combined with~\eqref{eq:sdp_of_simplex} yields a \new{dimension-}free description certifying both compatibilities. 
Similarly, the sequence of $\ell_2$ unit balls $\cvx B_{\ell_2}^n=\{x\in\RR^n:\|x\|_2\leq1\}$ is intersection and projection compatible, and its standard SDP description $\cvx B_{\ell_2}^n=\left\{x\in\RR^n:\begin{bmatrix} 1 & x^\top \\ x & I_n\end{bmatrix} =A_nx+u_n\succeq0\right\}$ is \new{dimension-}free and certifies these compatibilities similarly to~\eqref{eq:sdp_of_simplex}.

Notice that both sequences of norm balls are invariant under the larger group $\msf G_n=\msf B_n$. All of the above descriptions are in fact equivariant under this larger group, when $\mscr U,\mscr W$ are endowed with the corresponding group action.

\subsection{Regular Polygons}\label{sec:regular_polygons}
The following example illustrates a natural sequence of convex sets that is freely-described but satisfies neither intersection nor projection compatibility.
Let $\mscr V=\{(\vct V_n,\varphi_n)\}$ be the consistent sequence $\vct V_n=\RR^2$ with $\varphi_n=\mathrm{id}_{\RR^2}$ and the standard action of the dihedral group $\msf G_n = \mathrm{Dih}_{2^n}$. Consider the sequence of regular $2^n$-gons $\mscr C = \{\cvx C_n\subseteq \RR^2\}$ defined by
\begin{equation}\label{eq:regular_polys}
    \cvx C_n = \mathrm{conv}\left\{\begin{bmatrix} \cos\theta_i\\ \sin\theta_i\end{bmatrix}\right\},\quad \theta_i = \frac{2\pi i}{2^n},\quad i\in \{0,\ldots,2^n-1\}.
\end{equation}
Because $\vct V_n=\vct V_{n+1}$ while $\cvx C_n\neq \cvx C_{n+1}$, the sequence $\mscr C$ satisfies neither intersection nor projection compatibility. Nevertheless, it admits the \new{dimension-}free description
\begin{equation*}
    \cvx C_n = \left\{x\in\RR^2: \exists y\in\RR^{2^n} \textrm{ s.t. } \begin{bmatrix} -I \\ 0 \\ 0 \end{bmatrix}x + \begin{bmatrix}\begin{bmatrix}\multirow{2}{*}{$\cdots$} & \cos(2\pi i/2^n) & \multirow{2}{*}{$\cdots$} \\ & \sin(2\pi i/2^n) & \end{bmatrix}\\ \multicolumn{1}{c}{I_{2^n}} \\ \multicolumn{1}{c}{\mathbbm{1}_{2^n}^\top}\end{bmatrix}y + \begin{bmatrix} 0 \\ 0 \\ -1\end{bmatrix} \in 0\oplus\RR^{2^n}_+\oplus 0\right\},
\end{equation*}
where $\mscr W=\{\RR^{2^n}\}_n$ with embeddings $y\mapsto y\otimes [1,0]^\top$, and $\mscr U = \mscr V\oplus\mscr W\oplus \RR$. We put the standard inner products on $\RR^{2^n}$. The group permutes the $2^n$ vertices of $\cvx C_n$, defining a permutation action on $[2^n]$, and it acts on $\RR^{2^n}$ by applying these permutations to coordinates. 

If $\theta_i=\pi(2i+1)/2^n$ in~\eqref{eq:regular_polys} instead, we mention that the semidefinite description of $\cvx C_n$ given in~\cite{regular_polygons} is also \new{dimension-}free when $\mscr U,\mscr W,\mscr K$ are chosen appropriately.

\subsection{Permutahedra, Schur--Horn Orbitopes, and Limits}\label{sec:DS_permuta}

\paragraph{Permutahedra:}
Consider the sequence of standard permutahedra
\begin{equation}\label{eq:permuta}
    \mathrm{Perm}_n = \mathrm{conv}\left\{g\cdot [1,2,\ldots,n]^\top\right\} = \left\{ M[1,2,\ldots,n]^\top: M\in\RR^{n\times n}_+,\ M\mathbbm{1}_n=M^\top\mathbbm{1}_n = \mathbbm{1}_n\right\},
\end{equation}
where the second equality follows by the Birkhoff--von Neumann theorem. 
The sequence $\{\mathrm{Perm}_n\}$, viewed as subsets of the consistent sequence in Example~\ref{ex:basic_consist_seq} with $\msf G_n=\msf S_n$, is neither intersection- nor projection-compatible. Furthermore, their description~\eqref{eq:permuta} is not \new{dimension-}free because the map $M \mapsto M[1,2,\ldots,n]^\top$ is not $\msf G_n$-equivariant. The smaller descriptions of these permutahedra in~\cite{goemans2015smallest,sorting_network_wright} are also not \new{dimension-}free because they are not equivariant.
However, there is a sequence of permutahedra arising naturally from a limiting perspective that is both intersection- and projection-compatible and whose description certifies this compatibility. 

Fix $q,m\in\NN$ and a vector $\lambda\in\RR^q$ with distinct entries, and define $\widetilde \lambda=(\lambda_1,\ldots,\lambda_1,\ldots,\lambda_q,\ldots,\lambda_q)\allowbreak \in\RR^m$ in which $\lambda_i$ appears $m_i$ times (so $\sum_im_i=m$). 
Let $\msf G_n=\msf S_{m2^n}$ embedded in $\msf G_{n+1}$ by sending a $m2^n\times m2^n$ matrix $g\in \msf G_n$ to $g\otimes I_2$. 
Let $\mscr V=\{\RR^{m2^n}\}_n$ with embeddings $x\mapsto x\otimes\mathbbm{1}_2$, the normalized inner product $\langle x,y\rangle = (m2^n)^{-1}x^\top y$, and the standard action of $\msf G_n$. 
We consider convex hulls of all the vectors in $\mscr V$ containing $\lambda_i$ in a fraction $m_i/m$ of its entries, which are given using the Birkhoff--von Neumann theorem by
\begin{equation}\label{eq:permuta_good}
\begin{aligned}
    \mathrm{Perm}(\lambda)_n &= \mathrm{conv}\{g\cdot(\widetilde \lambda\otimes\mathbbm{1}_{m2^n})\}_{g\in \msf S_{m2^n}}\\ &= \left\{M\lambda: M\in \RR^{m2^n\times q}_+,\ M\mathbbm{1}_q = \mathbbm{1}_{m2^n},\ M^\top\mathbbm{1}_{m2^n} = 2^n[m_1,\ldots,m_q]^\top\right\}.
\end{aligned}
\end{equation}
This is an intersection- and projection-compatible sequence of subsets of $\mscr V$. Moreover, the description~\eqref{eq:permuta_good} is \new{dimension-}free and certifies intersection and projection compatibility.
Indeed, let $\mscr W = \mscr V^{\oplus q} = \{\RR^{m2^n\times q}\}_n$ and $\mscr U = \mscr W\oplus \mscr V^{\oplus 2}\oplus\RR^q$ containing cones $\{\RR^{m2^n\times q}_+\oplus 0\oplus0\}$. 
Then~\eqref{eq:permuta_good} is of the form~\eqref{eq:preim_seq} with
\begin{align*}
    &A_nx = (0,-x,0,0),\quad B_nM = (M,\ M\lambda,\ M\mathbbm{1}_q,\ (m2^n)^{-1}M^\top\mathbbm{1}_{m2^n}),
\end{align*}
and $u_n = (0, 0, -\mathbbm{1}_{m2^n}, -\left[\frac{m_1}{m},\ldots,\frac{m_q}{m}\right]^\top)$.
Since $\{A_n\},\{A_n^*\},\{B_n\},\{B_n^*\}$ are morphisms and $u_n=u_{n+1}$ under our embedding, Proposition~\ref{prop:compatible_descriptions}(a) applies (note that the inner product here is nonstandard, so $A_n^*\neq A_n^\top$ and similarly for $B_n$).

The insight behind this construction comes from the limiting perspective of Section~\ref{sec:lim_of_sets}. 
Note that $\vct V_{\infty}$ can be viewed as a space of piecewise-constant functions on $[0,1)$. 
Indeed, define intervals $I_i^{(n)}=[(i-1)/m2^n, i/m2^n)$ and associate to each $v\in \vct V_n$ the piecewise-constant function $f_v(x)=\sum_{i=1}^{m2^n}v_i\mathbbm{1}_{I_i^{(n)}}(x)$ where $\mathbbm{1}_S(x)=1$ if $x\in S$ and 0 otherwise. 
Note that $v\in \vct V_n$ and $v\otimes\mathbbm{1}_2\in \vct V_{n+1}$ define the same function in this way, and that $\langle v,w\rangle = \langle f_v,f_w\rangle_{L^2([0,1))}$. 
Also, we have for $f\in \vct V_{\infty}$
\begin{equation}\label{eq:local_avg}
    (\mc P_nf)(x) = m2^n\int_{I_i^{(n)}}f\quad \textrm{if } x\in I_i^{(n)},
\end{equation}
hence Jensen's inequality implies $\|\mc P_nf\|_{L^p}\leq \|f\|_{L^p}$ for all $p\in[1,\infty]$.
Thus, we can identify $\vct V_{\infty}$ with the space of such step functions, which is contained in $L^p([0,1))$ for all $p\in[1,\infty]$, and we can consider the completion of $\vct V_{\infty}$ with respect to any of the $L^p$ norms. 
If $p<\infty$ we get $\overline{\vct V_{\infty}}=L^p([0,1))$ while if $p=\infty$ then $\overline{\vct V_{\infty}}$ is the space of functions continuous on $[0,1)\setminus\{m/2^n: m,n\in\NN\}$ and having left and right limits everywhere~\cite[\S7.6]{foundationsAnalysis}. 
Then $\overline{\mathrm{Perm}(\lambda)_{\infty}}$ is the closed convex hull of functions in $\overline{\vct V_{\infty}}$ that take values $\lambda_i$ on a subset of $[0,1]$ of measure $m_i/m$. 
Furthermore, Theorem~\ref{thm:descriptions_of_lims} applies to the description in~\eqref{eq:permuta_good} and yields the following dense subset of $\overline{\mathrm{Perm}(\lambda)_{\infty}}$.
\begin{proposition}\label{prop:limts_of_permuta}
    Endowing $\vct V_{\infty}$ with the $L^p([0,1))$ norm as above for any $p\in[1,\infty]$, we obtain
    \begin{equation}\begin{aligned}
        \overline{\mathrm{Perm}(\lambda)_{\infty}} &= \overline{\mathrm{conv}}\left\{f\in\overline{\vct V_{\infty}}: f([0,1))=\{\lambda_1,\ldots,\lambda_q\},\ |f^{-1}(\lambda_i)|=\frac{m_i}{m}\right\}\\ &= \overline{\left\{\sum_{i=1}^q\lambda_if_i: f_i\in\overline{\vct V_{\infty}},\ f_i\geq 0 \textrm{ a.e. }, \sum_{i=1}^qf_i = 1,\ \int_{[0,1)} f_i = \frac{m_i}{m}\right\}}.
    \end{aligned}\end{equation}
\end{proposition}
\begin{proof}
    For the first equality, the inclusion $\subseteq$ is clear. For the reverse inclusion, if $f\in\overline{\vct V_{\infty}}$ takes values $\lambda_i$ on sets $\Omega_i=f^{-1}(\lambda_i)$ of measure $m_i/m$ partitioning $[0,1)$, we can write $f=\sum_{i=1}^q\lambda_i\mathbbm{1}_{\Omega_i}$. Then $\mc P_nf = \sum_{i=1}^q\lambda_i(\mc P_n\mathbbm{1}_{\Omega_i})$, and under the above identification of $\RR^{m2^n}$ with piecewise-constant functions, we can view $\mc P_n\mathbbm{1}_{\Omega_i}\in \RR^{m2^n}_+$ as nonnegative vectors. 
    The matrix $M = [\mc P_n\mathbbm{1}_{\Omega_1},\ldots,\mc P_n\mathbbm{1}_{\Omega_q}]\in \RR_+^{m2^n\times q}$ then satisfies the conditions in~\eqref{eq:permuta_good}, hence $\mc P_nf\in \mathrm{Perm}(\lambda)_n$ and $f=\lim_n\mc P_nf\in \overline{\mathrm{Perm}(\lambda)_{\infty}}$ by Lemma~\ref{lem:projections_converge}, giving the reverse inclusion.

    The second equality follows from Theorem~\ref{thm:descriptions_of_lims}. Indeed, endow $\vct W_{\infty}=\vct V_{\infty}^q$ with the norm $\|[f_1,\ldots,f_q]\|=\max_i\|f_i\|_p$ and $\vct U_{\infty}$ with the norm $([f_1,\ldots,f_q],g_1,g_2,\mu) = \max\{\|f_i\|_p,\|g_j\|_p,\|\mu\|_{\infty}\}$. Then $\overline{\vct W_{\infty}}=\overline{\vct V_{\infty}}^q$ and $\overline{\vct U_{\infty}}=\overline{\vct W_{\infty}}\oplus \overline{\vct V_{\infty}}^2\oplus \RR^q$, and we have $\|A_n\|_{\mathrm{op}}=1$, $\|B_n\|_{\mathrm{op}}\leq \max\{\sum_i|\lambda_i|,q\}$, and $u_n=u_{n+1}$ for all $n$.
\end{proof}

\paragraph{Schur--Horn orbitopes:}
We consider the matrix analogs of the above permutahedra, which are convex hulls of all matrices with a given spectrum. Let $\msf G_n = \msf O_{m2^n}$ embed in $\msf G_{n+1}$ by sending a $m2^n\times m2^n$ matrix $g$ to $g\otimes I_2$, let $\vct V_n=\mbb S^{m2^n}$ embed in $\vct V_{n+1}$ by $X\mapsto X\otimes I_2$ with normalized inner product $\langle X,Y\rangle = (m2^n)^{-1}\mathrm{Tr}(X^\top Y)$, and finally let $\msf G_n$ act by conjugation on $\vct V_n$. 
Consider the sequence of Schur--Horn orbitopes~\cite[Eq.~(19)]{finding_planted}
\begin{equation}\begin{aligned}\label{eq:schur_horn}
    \mathrm{SH}(\lambda)_n &= \mathrm{conv}\{g\cdot\mathrm{diag}(\lambda\otimes \mathbbm{1}_{2^n})\}_{g\in \msf O_{m2^n}}\\
    &= \left\{\sum_{i=1}^q\lambda_iY_i: Y_1,\ldots,Y_q\in \vct V_n \textrm{ s.t. } \sum_{i=1}^qY_i=I,\ Y_i\succeq 0,\ \mathrm{Tr}(Y_i) = m_i2^n \textrm{ for } i=1,\ldots,q\right\},
\end{aligned}\end{equation}
which is the matrix analog of~\eqref{eq:permuta_good}.
This is again a \new{dimension-}free description certifying both intersection- and projection-compatibility. 
Indeed, let $\mscr W = \mscr V^{\oplus q}$ and $\mscr U = \mscr W\oplus\mscr V^{\oplus 2}\oplus\RR^q$ containing the cones $\{\cvx K_n=(\mbb S^{m2^n}_+)^{\oplus q}\oplus 0\oplus 0\}$. Then~\eqref{eq:schur_horn} is of the form~\eqref{eq:preim_seq} with
\begin{align*}
    &A_nX = (0,-X,0,0),\quad B_n[Y_i]_{i=1}^q = \left([Y_i]_{i=1}^q,\ \sum\nolimits_i\lambda_iY_i,\ \sum\nolimits_iY_i,\ (m2^n)^{-1}[\mathrm{Tr}(Y_1),\ldots, \mathrm{Tr}(Y_q)]\right),
\end{align*}
and $u_n = (0,0,-I_{m2^n},-[\tfrac{m_1}{m},\ldots,\tfrac{m_q}{m}]^\top)$.
Here $\{A_n\},\{A_n^*\},\{B_n\},\{B_n^*\}$ are all morphisms and $u_n=u_{n+1}$ under the above embedding, hence Proposition~\ref{prop:compatible_descriptions}(a) applies.

Again, we can use Theorem~\ref{thm:descriptions_of_lims} to describe an infinite-dimensional limit $\overline{\mathrm{SH}(\lambda)_{\infty}}$ of these Schur--Horn orbitopes. 
We would like to interpret the elements of $\mathrm{SH}(\lambda)_{\infty}$ as operators with spectral measure $\sum_{i=1}^q\frac{m_i}{m}\delta_{\lambda_i}$, then describe the convex hull of all such operators in $\overline{\vct V_{\infty}}$. 
To that end, we complete $\vct V_{\infty}$ with respect to the operator norm and consider spectral measures with respect to the normalized trace $\tau$ on $\vct V_{\infty}$, which also extends to the limit. 
Such spectral measures and associated orbitopes exist in more general algebras. In fact, our framework yields descriptions for Schur--Horn orbitopes, as well as a generalization of the Schur--Horn theorem, to self-adjoint elements in so-called approximately finite-dimensional (AF) algebras. 
These orbitopes generalize both the permutahedra and Schur--Horn orbitopes in this section.
We construct these algebras and apply our framework to the resulting orbitopes in Appendix~\ref{apdx:SH_in_AF_algebras}. 
In particular, we obtain the following description for the limit in our present setting, proved in Proposition~\ref{prop:infinite_SchurHorn_descrip} of that appendix.
\begin{proposition}\label{prop:lims_of_SH}
    Endow $\vct V_{\infty}$ with the operator norm and let $\tau\colon \vct V_{\infty}\to\RR$ be the normalized trace given by $\tau(X)=(m2^n)^{-1}\mathrm{Tr}(X)$ for $X\in \vct V_n$. Then $\tau$ extends to $\overline{\vct V_{\infty}}$, and each $X\in \overline{\vct V_{\infty}}$ has a spectral measure $\mu_X^{\tau}$ with respect to $\tau$. Furthermore,
    \begin{equation*}
        \overline{\mathrm{SH}(\lambda)_{\infty}} = \overline{\mathrm{conv}}\left\{X\in\overline{\vct V_{\infty}}: \mu_X^{\tau} = \sum\nolimits_{i=1}^q\frac{m_i}{m}\delta_{\lambda_i}\right\} = \overline{\left\{\sum\nolimits_{i=1}^q\lambda Y_i: Y_i\succeq0,\ \sum\nolimits_iY_i=I,\ \tau(Y_i)=\frac{m_i}{m}\right\}}.
    \end{equation*}
\end{proposition}
Furthermore, the diagonal map $\diag_n\colon \mbb S^{m2^n}\to \RR^{m2^n}$ satisfies $\diag_n(\mathrm{SH}(\lambda)_n)=\mathrm{Perm}(\lambda)_n$ by the Schur--Horn theorem. Our limiting descriptions yield the following limiting version of this theorem, proved in greater generality in Proposition~\ref{prop:infinite_SchurHorn_theorem} of the above appendix.
\begin{proposition}\label{prop:SH_theorem_special}
    In the setting of Proposition~\ref{prop:lims_of_SH}, we have $\diag(\overline{\mathrm{SH}(\lambda)_{\infty}}) = \overline{\mathrm{Perm}(\lambda)_{\infty}}$ where we complete the permutahedra as in Proposition~\ref{prop:limts_of_permuta} with respect to the $L^{\infty}$ norm.
\end{proposition}

Schur--Horn orbitopes are special cases of so-called \emph{spectral polyhedra} studied in~\cite{sanyal2020spectral}. It would be interesting to identify further examples of freely-described sequences of spectral polyhedra arising in applications and to consider their limits.

\begin{remark}\label{rmk:poset_gen_moved}
    We remark that the above constructions can be extended to handle vectors and matrices of all sizes, not just of sizes $(m2^n)_{n\in\NN}$. This can be done by indexing our consistent sequences by posets.  In particular, our results generalize to the following, more complicated, sequences of group representations. If $\mc N$ is a strict poset, an $\mc N$-indexed consistent sequence of $\{\msf{G}_n\}_{n\in\mc N}$-representations is a sequence $\{(\vct \vct V_n,\varphi_{N,n})\}_{n<N\in\mc N}$ of $\msf{G}_n$-representations and embeddings $\varphi_{N,n}\colon \vct \vct V_n\hookrightarrow \vct \vct V_N$ for each $n<N$ such that $\varphi_{N,n}$ is $\msf{G}_n$-equivariant, and $\varphi_{M,N}\circ \varphi_{N,n}=\varphi_{M,n}$ whenever $n<N<M$. All our results apply in this setting after replacing all occurrences of $n+1$ by $N>n$.  To handle permutahedra and Schur--Horn orbitopes of any sizes we let $\mc N=\NN$ with the divisibility partial order, whereby $n\leq m$ iff $m=nk$ for some $k\in\NN$, with embeddings $\varphi_{kn,n}\colon\RR^n\hookrightarrow\RR^{nk}$ sending $x\mapsto x\otimes\mathbbm{1}_k$ or $\varphi_{kn,n}\colon \mbb S^n\hookrightarrow\mbb S^{nk}$ sending $X\mapsto X\otimes I_k$.
%
\end{remark}

\subsection{Free Spectrahedra}\label{sec:free_spectr}
The family of free spectrahedra in Example~\ref{ex:free_descr_intro_param_reorg}(b) may be obtained as a special case of our framework with appropriate selection of description spaces and by imposing compatibility.

Let $\mscr V_0 = \{\mbb S^n\}$ with embeddings by zero-padding and the action of $\msf G_n=\msf O_n$ by conjugation. Fix $d,k\in\NN$, and let $\mscr V=\mscr V_0^{\oplus d}$, $\mscr U=\mbb S^k\otimes \mscr V_0$, and $\mscr W = \{\vct W_n=0\}$.

As the only morphisms $\mscr V_0\to\mscr V_0$ are multiples of the identity, and elements of $\mbb S^k\otimes\mbb S^n$ can be viewed as $k\times k$ symmetric block matrices with symmetric $n\times n$ blocks of $n\times n$ symmetric matrices, we conclude that the morphisms $\mscr V\to\mscr U$ are precisely maps of the form $(X_1,\ldots,X_d)\mapsto \sum_iL_i\otimes X_i$ for some $L_1,\ldots,L_d\in\mbb S^k$. Note that sequences of adjoints of such maps are also morphisms in this case.
As the only $\msf G_n$-invariants in $\mbb S^n$ are multiples of $I_n$, the space of freely-described elements in $\mscr U$ is $\{\{L_0\otimes I_n\}_n: L_0\in\mbb S^k\}$, which satisfy the condition $L_0\otimes(I_{n+1}-I_n)\succeq0$ from Proposition~\ref{prop:compatible_descriptions}(a) if and only if $L_0\succeq0$.
Thus, the parametric family of \new{dimension-}free descriptions certifying compatibility as in Proposition~\ref{prop:compatible_descriptions}(a) is
\begin{equation*}
    (\mc D_{\mc L})_n = \left\{(X_1,\ldots,X_d)\in(\mbb S^n)^d: L_0\otimes I_n + \sum_{i=1}^dL_i\otimes X_i\succeq0 \right\},\quad L_0\succeq0.
\end{equation*}
which are free spectrahedra parametrized by $\mc L = (L_0,\ldots,L_d)$.
It is common to take either $L_0 = I_k$ (the \emph{monic} case) or $L_0=0$ (the \emph{homogeneous} case)~\cite{kriel2019introduction}.
As discussed in Section~\ref{sec:related_work}, free spectrahedra are fundamental objects in noncommutative free convex and algebraic geometry, see~\cite{free_AG_chap,kriel2019introduction} for an introduction. 
In particular, they satisfy both intersection and projection compatibility, and more generally closure under so-called \emph{matrix-convex combinations}.  

We remark that orthogonal invariance and compatibility alone do not yield closure under matrix-convex combinations. For example, the nuclear norm balls $\{X\in\mbb S^n:\|X\|_{\ast}\leq 1\}$ are $\msf O_n$-invariant and form a compatible sequence which is not matrix-convex. However, a compatible sequence of orthogonally-invariant convex \emph{cones} is indeed matrix convex, as the following result shows.
\begin{proposition}\label{prop:compatible_cones_mtx_cvx}
    Let $\{\vct V_n=(\mbb S^n)^d\}$ be the consistent sequence of $\{\msf O_n\}$-representations above, and suppose $\{\cvx C_n\subseteq \vct V_n\}$ is an intersection- and projection-compatible sequence of convex cones such that $\cvx C_n$ is $\msf O_n$-invariant for all $n$. Then $\{\cvx C_n\}$ is matrix-convex.
\end{proposition}
\begin{proof}
    By~\cite[\S2.3]{helton2016matrix}, it suffices to show that if $(X_1,\ldots,X_d)\in \cvx C_n$ and $(Y_1,\ldots,Y_d)\in\cvx C_m$ then $(X_1\oplus Y_1,\ldots,X_d\oplus Y_d)\in\cvx C_{n+m}$, and that $(V^\top X_1V,\ldots,V^\top X_dV)\in \cvx C_k$ for any isometry $V\in \RR^{n\times k}$. For the first, intersection compatibility shows that $(X_1\oplus 0_m,\ldots,X_d\oplus 0_m)\in \cvx C_{n+m}$ and $(Y_1\oplus 0_n,\ldots,Y_d\oplus 0_n)\in \cvx C_{n+m}$. Conjugating the latter tuple by appropriate permutation matrices, we conclude that $(0_n\oplus Y_1,\ldots,0_n\oplus Y_d)\in \cvx C_{n+m}$. Finally, since $\cvx C_{n+m}$ is a convex cone we get $(X_1\oplus Y_1,\ldots,X_d\oplus Y_d)\in\cvx C_{n+m}$, so $\{\cvx C_n\}$ is closed under direct sums. For the second, any isometry $V\in\RR^{n\times k}$ can be written as $V=U\iota_{n,k}$ where $U\in\msf O_n$ and $\iota_{n,k}\colon \RR^k\to \RR^n$ is the zero-padding embedding. Observe that $(\iota_{n,k}^\top X_1\iota_{n,k},\ldots,\iota_{n,k}^\top X_d\iota_{n,k})=\mc P_k(X_1,\ldots,X_d)$ is the orthogonal projection of $(X_1,\ldots,X_d)$ onto $\vct V_k$ when embedded in $\vct V_n$. Thus, $(V^\top X_1V,\ldots,V^\top X_dV) = \mc P_k(U^\top X_1U,\ldots,U^\top X_dU)\in\cvx C_k$ since $\cvx C_n$ is $\msf O_n$-invariant and since $\{\cvx C_n\}$ is projection-compatible. 
\end{proof}

To recap, orthogonal invariance and compatibility of a sequence of convex sets do not yield matrix convexity in general, but they do so for convex cones.  With additional choices of description spaces, our framework yields a particular parametric family of matrix-convex sets, namely free spectrahedra.


\subsection{Spectral Functions, (Quantum) Entropy, and Variants}\label{sec:spectral_funcs}
Let $\mscr V = \{\mbb S^n\}$ with embeddings by zero-padding and with the action of $\msf O_n$ by conjugation, and let $\mscr V' = \{\RR^n\}$ with embeddings by zero-padding and the standard action of $\msf S_n$.
Recall (e.g.,~\cite{spectral_funcs}) that a convex function $F_n\colon \mbb S^n\to\RR$ is $\msf O_n$-invariant if and only if there exists an $\msf S_n$-invariant convex function $f_n\colon \RR^n\to\RR$ satisfying $F_n(X)=f_n(\lambda(X))$ where $\lambda(X)\in\RR^n$ is the vector of eigenvalues of $X\in\mbb S^n$. Also, the sequence $\{F_n\colon \mbb S^n\to\RR\}$ is intersection-compatible if and only if the sequence $\{f_n\colon \RR^n\to\RR\}$ is.

Examples of such sequences of functions $\mfk F=\{F_n\}$ and $\mfk f=\{f_n\}$ arise in (quantum) information theory, where $\mfk F$ is the quantum analog of classical information-theoretic parameters $\mfk f$. These are often intersection-compatible as distributions on $n$ states can be viewed as distributions on $n+1$ states with zero probability on the last state. For example, the negative entropy and relative entropy and their quantum variants are given by
\begin{equation}\label{eq:entropies}
\begin{aligned}
    &h_n(x) = \sum_ix_i\log x_i, && \msf H_n(X) = h_n(\lambda(X)) = \mathrm{Tr}(X\log X),\\
    &D_n(x, y) = \sum_ix_i\log\frac{x_i}{y_i}, && S_n(X, Y) = D_n(\lambda(X),\lambda(Y))=\mathrm{Tr}(X(\log X - \log Y)).
\end{aligned}\end{equation}
Here $\mathrm{dom}(h_n) = \Delta^{n-1}$ and $\mathrm{dom}(D_n) = (\RR^n_+)^2$, while $\mathrm{dom}(\msf H_n)=\mc D^{n-1}$ is the spectraplex from Example~\ref{ex:free_descr_intro}(b) and $\mathrm{dom}(S_n) = (\mbb S^n_+)^2$. We use the standard convention that $0\log\frac{0}{y}=0$ even if $y=0$, and $x\log\frac{x}{0}=\infty$ when $x\neq0$~\cite[\S2.3]{cover1999elements}. 
These sequences of functions are intersection-compatible but not projection-compatible (e.g., their domains are not projection-compatible). 

\paragraph{Semidefinite approximations:} The functions~\eqref{eq:entropies} are not semidefinite-representable (i.e., cannot be evaluated using semidefinite programming), though semidefinite approximations of them have been proposed in the literature~\cite{fawzi2019semidefinite}. We show that these approximations are freely-described, but that these descriptions do not certify intersection compatibility. 
The family of approximations of~\cite{fawzi2019semidefinite} to the negative quantum entropy is parametrized by $m,k\in\NN$ via their epigraphs 
\begin{align*}
    E_n^{(m,k)} &= \Bigg\{(X,t)\in \mbb S^n\oplus \RR \Bigg| \exists T_0,\ldots,T_m,Z_0,\ldots,Z_k\in \mbb S^{n} \textrm{ s.t. } Z_0=I_{n},\ \sum_{j=1}^m w_jT_j = -2^{-k}T_0,\\ &\quad \begin{bmatrix} Z_i & Z_{i+1}\\ Z_{i+1} & X\end{bmatrix}\succeq 0, \textrm{ for } i=0,\ldots,k-1,\ \begin{bmatrix} Z_k - X - T_j & -\sqrt{s_j}T_j\\ -\sqrt{s_j}T_j & X - s_jT_j\end{bmatrix}\succeq0,\\ &\quad \textrm{for } j=1,\ldots,m,\  \mathrm{Tr}(T_0)\leq t
    \Bigg\},
\end{align*}
where $s,w\in \RR^m$ are the nodes and weights for Gauss-Legendre quadrature. 

This is a \new{dimension-}free description of the form~\eqref{eq:preim_seq} which almost, but not quite, satisfies the conditions of Proposition~\ref{prop:compatible_descriptions}. 
Indeed, let $\mscr W = \mscr V^{\oplus (m+k+1)}$ and $\mscr U = \mscr V\oplus(\mbb S^2\otimes\mscr V)^{\oplus(m+k)}\oplus\RR$ containing the cones $\{\cvx K_n = \{0\}\oplus (\mbb S^2\otimes\mbb S^n)_+^{\oplus(m+k)}\oplus \RR_+\}$. 
Define
\begin{align*}
    &A_n(X,t) = \left(0, \left(\begin{bmatrix} 0 & 0\\ 0 & 1\end{bmatrix}\otimes X\right)^{\oplus k}, \left(\begin{bmatrix} -1 & 0 \\ 0 & 1\end{bmatrix}\otimes X\right)^{\oplus m}, t\right),\\
    &B_n(T_0,\ldots,T_m,Z_1,\ldots,Z_k) = \Bigg(2^{-k}T_0 + \sum_{j=1}^mw_jT_j, \begin{bmatrix}0 & 1\\ 1 & 0\end{bmatrix}\otimes Z_1, \bigoplus_{i=1}^{k-1}\left(\begin{bmatrix} 1 & 0\\ 0 & 0\end{bmatrix}\otimes Z_i +  \begin{bmatrix}0 & 1\\ 1 & 0\end{bmatrix}\otimes Z_{i+1}\right),\\&\qquad \bigoplus_{j=1}^m\left(\begin{bmatrix} 1 & 0\\ 0 & 0\end{bmatrix}\otimes Z_k - \begin{bmatrix} 1 & \sqrt{s_j}\\ \sqrt{s_j} & s_j\end{bmatrix}\otimes T_j\right), -\mathrm{Tr}(T_0)\Bigg),\\
    &u_n = \left(0, \begin{bmatrix} 1 & 0\\ 0 & 0\end{bmatrix}\otimes I_n, 0,\ldots,0\right).
\end{align*}
Note that $\{A_n\}$ and $\{B_n\}$ are morphisms, that $\{u_n\}$ is a freely-described element of $\mscr U$ satisfying $u_{n+1}-u_n\in \cvx K_{n+1}$, but that $\{B_n^*\}$ is \emph{not} a morphism because of the $\mathrm{Tr}(T_0)$ term in $B_n$. 
By the proof of Proposition~\ref{prop:compatible_descriptions}, this description certifies that $E_n^{(m,k)}\subseteq E_{n+1}^{(m,k)}$ but not $E_{n+1}^{(m,k)}\cap (\mbb S^n\oplus\RR)\subseteq E_n^{(m,k)}$. 
Analogously to the na\"ive description of the solid simplex in Section~\ref{sec:simplex}, this \new{dimension-}free description does not make the intersection compatibility of $\{E_n^{(m,k)}\}$ obvious.


\paragraph{Parametric families:} 
We can use the description spaces of~\cite{fawzi2019semidefinite} above to derive parametric families of freely-described sets.  When compatibility is not required, the resulting family includes the approximation $\{E_n^{(m,k)}\}_n$ of~\cite{fawzi2019semidefinite}; when compatibility is imposed, we obtain a smaller family excluding $\{E_n^{(m,k)}\}_n$.

Note that $\mbb S^2\otimes\mbb S^n\cong (\mbb S^n)^3$ as $\msf O_n$-representations, and these isomorphisms commute with zero-padding, so that $\mscr U\cong \mscr V^{1+3(m+k)}\oplus\RR$ as consistent sequences.
As $\dim(\mbb S^n)^{\msf O_n}=1$ and $\dim\mc L(\mbb S^n)^{\msf O_n}=2$, the dimension of invariants parametrizing \new{dimension-}free descriptions are
\begin{align*}
    &\dim\mc L(\vct V_n,\vct U_n)^{\msf O_n} = 3[2(m+k)+1],\quad \dim\mc L(\vct W_n,\vct U_n)^{\msf O_n} = 3(m+k+1)[2(m+k)+1],\\ &\dim \vct U_n^{\msf O_n} = 2+3(m+k).
\end{align*}
When $m=k=3$ (the default values in the implementation of~\cite{fawzi2019semidefinite}), we get 
\begin{equation*}
    \dim\mc L(\vct V_n,\vct U_n)^{\msf O_n}=39,\quad \dim\mc L(\vct W_n,\vct U_n)^{\msf O_n}=273,\quad \dim \vct U_n^{\msf O_n}=20.
\end{equation*}
As the only morphisms of sequences $\mscr V\to\mscr V$ are multiples of the identity, and the only morphisms $\mscr V\to\RR$ are multiples of the trace, the dimensions of $\{A_n\},\{B_n\}$ satisfying the conditions of Proposition~\ref{prop:compatible_descriptions}(a) are
\begin{align*}
    &\dim \Big\{\{A_n\colon \vct V_n\to \vct U_n\} \textrm{ morphism}\Big\} = 3(m+k)+2=20,\\ &\dim \Big\{\{B_n\colon \vct W_n\to \vct U_n\}: \textrm{both } \{B_n\} \textrm{ and } \{B_n^*\} \textrm{ morphisms}\Big\} = (m+k+1)[3(m+k)+1] = 133.
\end{align*}

\subsection{Graph Parameters}\label{sec:graph parameters}
Let $\mscr V=\{\mbb S^n\}$ with embeddings by zero-padding and the action of $\msf G_n=\msf S_n$ by conjugation. As discussed in Section~\ref{sec:intro_structural} and Example~\ref{ex:graph_params}, graph parameters are sequences $\{f_n\colon\mbb S^n\to\RR\}$ of $\msf G_n$-invariant functions, and many standard graph parameters are convex (or concave) and satisfy either intersection or projection compatibility.
Furthermore, it is desirable to find relaxations for standard graph parameters satisfying the same compatibility conditions.
We consider here the examples of max-cut and inverse stability number, and we show that their usual relaxations are freely-described with descriptions certifying their compatibility. We then derive parametric families of \new{dimension-}free descriptions from our framework that may be used to fit graph parameters to data.

\paragraph{Max-cut:}
Computing the max-cut value of a weighted undirected graph amounts to evaluating the support function of the \emph{cut polytope} $\mathrm{CUT}_n = \mathrm{conv}\{xx^\top: x\in\{\pm 1\}^n\}$.
The sequence of cut polytopes $\{\mathrm{CUT}_n\}$ viewed as subsets of $\mscr V$ is projection-compatible and compact, hence the sequence of their support functions and the max-cut value itself are intersection-compatible (see Section~\ref{sec:intro_notation}). 
Approximation of the max-cut value reduces to approximation of the cut polytopes.
A standard outer approximation of the sequence of cut polytopes is the sequence of elliptopes 
\begin{equation}\label{eq:elliptope}
    \mc E_n=\{X\in\mbb S^n: X\succeq0, \mathrm{diag}(X)=\mathbbm{1}_n\}.
\end{equation}
The sequence $\{\mc E_n\}$ also satisfies projection compatibility, and the above is a \new{dimension-}free description certifying this compatibility. 
Indeed, let $\mscr W=\{0\}$ and $\mscr U = \mscr V\oplus \{\RR^n\}$ where the latter sequence is the usual one from Example~\ref{ex:basic_consist_seq}, with cones $\mscr K = \{\mbb S^n_+\oplus\{0\}\}$. Then~\eqref{eq:elliptope} is of the form~\eqref{eq:preim_seq} with $A_nX = (X, \diag(X))$, $B_n=0$, and $u_n=(0,-\mathbbm{1}_n)$. Note that $\{A_n\},\{A_n^*\}$ are morphisms and $u_{n+1}-u_n = (0,-e_{n+1}) = (e_{n+1}e_{n+1}^\top,0)-A_{n+1}e_{n+1}e_{n+1}^\top$, hence Proposition~\ref{prop:compatible_descriptions}(b) applies. 
Neither the cut polytopes nor the elliptopes is intersection-compatible, as zero-padding a matrix with all-1's diagonal does not yield such a matrix. The sequences $\{\mathrm{CUT}_n-I_n\}$ and $\{\mc E_n-I_n\}$ are, however, both intersection- and projection-compatible, and the shifted elliptopes admit \new{dimension-}free descriptions certifying their compatibility.

\paragraph{Inverse stability number:} Computing the inverse stability number reduces to evaluating the support functions of $\mc D_n = \mathrm{conv}\{xx^\top: x\in\Delta^{n-1}\}$, see~\cite{motzkin_straus_1965}. A natural SDP relaxation for this problem is evaluating the support function of
\begin{align*}
     \widetilde{\mc D}_n = \{X\in\mbb S^n: X\succeq0, X\geq0, \mathbbm{1}_n^\top X\mathbbm{1}_n=1\},
\end{align*}
where $X\geq 0$ denotes an entrywise nonnegative matrix.
Both $\{\mc D_n\}$ and $\{\widetilde{\mc D}_n\}$ are intersection-compatible. Moreover, the above description of $\{\widetilde{\mc D}_n\}$ is \new{dimension-}free and certifies this compatibility as in Proposition~\ref{prop:compatible_descriptions}(b). Indeed, let $\mscr W=\{0\}$ and $\mscr U = \mscr V\oplus\RR$ with cones $\cvx K_n=(\mbb S^n_+\cap \RR^{n\times n}_+)\oplus 0$. Then the above description of $\widetilde{\mc D}_n$ is of the form~\eqref{eq:preim_seq} with $A_nX = (X, \mathbbm{1}_n^\top X\mathbbm{1}_n)$, $B_n=0$, and $u_n=(0,-1)$. Note that $\{A_n\}$ is a morphism and $u_n=u_{n+1}$, hence Proposition~\ref{prop:compatible_descriptions}(a) applies.
Neither $\mc D_n$ nor $\widetilde{\mc D}_n$ is projection-compatible since their support functions, which are the inverse stability number and its semidefinite approximation above, are not intersection compatible (see Section~\ref{sec:intro_notation}). Indeed, appending isolated vertices to a graph increases its stability number.

\paragraph{Parametric families:} Beyond the two particular preceding examples, we derive here an expressive parametric family of convex graph parameters; in addition to above examples, our family also includes several other examples from~\cite{cvx_grph_invar}.
To that end, let $\mscr W = \mscr U = \mathrm{Sym}^2(\mathrm{Sym}^{\leq 2}\{\RR^n\})=\{\mbb S^{\binom{n+2}{2}}\}$. 
The dimensions of invariants parametrizing \new{dimension-}free descriptions in this case is too large for us to write explicit bases for them as functions of $n$.
Instead, we compute these dimensions using the algorithm in Section~\ref{sec:implementation} below, see Example~\ref{ex:dim_counts}(b):
\begin{equation*}
    \dim\mc L(\vct V_n,\vct U_n)^{\msf G_n}=93,\quad \dim\mc L(\vct W_n,\vct U_n)^{\msf G_n}=1068,\quad \dim \vct W_n^{\msf G_n} = 17,\quad \textrm{for all } n\geq 8.
\end{equation*}
Using the same algorithm, the dimensions of sequences $\{A_n\},\{B_n\}$ certifying intersection compatibility as in Proposition~\ref{prop:compatible_descriptions}(a) are
\begin{equation}\label{eq:dim_graph_morphisms}\begin{aligned}
    &\dim \Big\{\{A_n\colon \vct V_n\to \vct U_n\} \textrm{ morphism}\Big\} = 19,\\ &\dim \Big\{\{B_n\colon \vct W_n\to \vct U_n\}: \textrm{both } \{B_n\} \textrm{ and } \{B_n^*\} \textrm{ morphisms}\Big\} = 104.
\end{aligned}\end{equation}
The algorithm we present in Section~\ref{sec:numerical_examples} further allows us to compute bases for these invariants in a fixed dimension and extend them to any other; in turn, these are useful for fitting graph parameters defined for graphs of all sizes given data.

\subsection{Graphon Parameters}\label{sec:graphons}
A different embedding between graphs arises in the theory of graphons~\cite{lovasz2012large}, where a weighted graph $X\in \mbb S^{2^n}$ is viewed as a step function $W_X\colon[0,1]^2\to\RR$ defined by $W_X(x,y)=X_{i,j}$ if $(x,y)\in[(i-1)/2^n,i/2^n)\times[(j-1)/2^n,j/2^n)$; see Figure~\ref{fig:graphon}. 
Note that $X$ and $X\otimes\mathbbm{1}_2\mathbbm{1}_2^\top\in \mbb S^{2^{n+1}}$ correspond to the same step function, and that the inner product of two such step functions $W_X, W_Y$ in $L^2([0,1]^2)$ equals the normalized Frobenius inner product $\langle X,Y\rangle = 2^{-2n}\mathrm{Tr}(X^\top Y)$. 
We therefore define the graphon consistent sequence $\mscr V=\{\vct V_n = \mbb S^{2^n}\}$ with embeddings $\varphi_n(X) = X\otimes\mathbbm{1}_{2\times 2}$, the above normalized inner products, and the action of $\msf G_n=\msf S_{2^n}$ by conjugation.
Here $\msf G_n$ is embedded into $\msf G_{n+1}$ by sending a permutation matrix $g$ to $g\otimes I_2$. 
We can extend our consistent sequence to include symmetric matrices of any size by having our consistent sequence be indexed by the poset $\NN$ with the divisibility partial order, see Remark~\ref{rmk:poset_gen_moved}. 
\begin{figure}
    \centering
    \includegraphics[width=\linewidth]{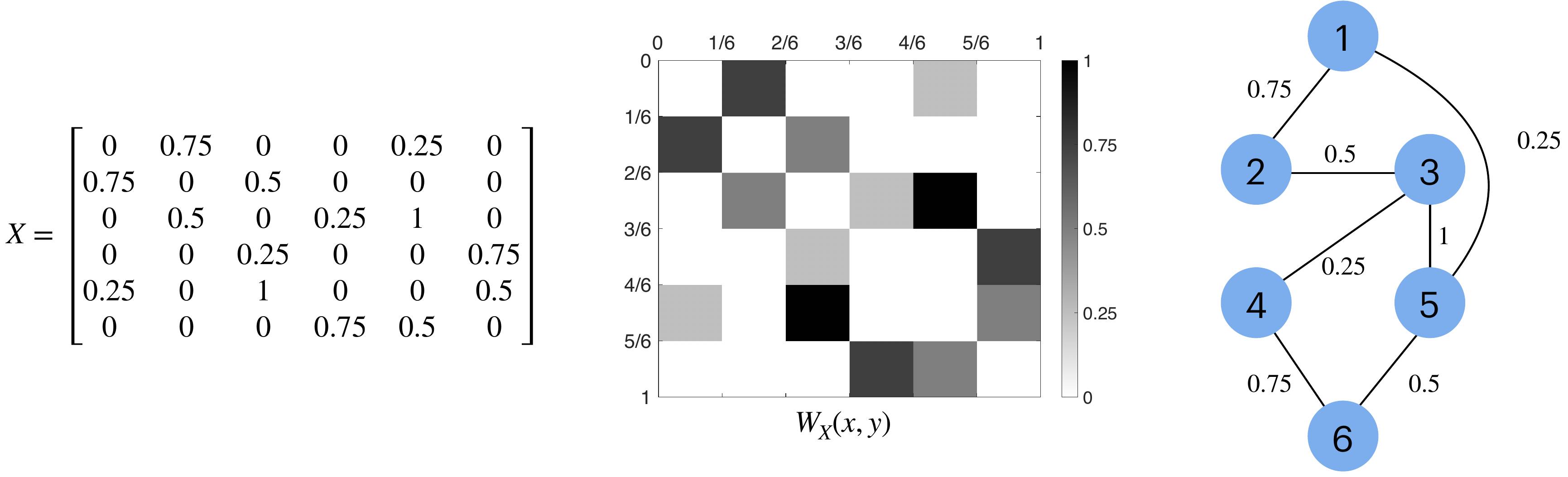}
    \caption{Weighted undirected graph represented as a graph, an adjacency matrix $X$, and a symmetric function (graphon) $W_X$ on $[0,1]^2$.}
    \label{fig:graphon}
\end{figure}
The graphon sequence is finitely-generated, as the following computer-assisted proof shows.
\begin{proposition}\label{prop:graphon_gen_deg}
    The graphon sequence $\{\vct V_n=\mbb S^{2^n}\}$ is generated in degree $2$.
\end{proposition}
\begin{proof}
    Define $E_1^{(n)} = e_1^{(2^n)}(e_1^{(2^n)})^\top$ and $E_2^{(n)} = e_1^{(2^n)}(e_2^{(2^n)})^\top + e_2^{(2^n)}(e_1^{(2^n)})^\top$, whose $\msf G_n$-orbits span $\vct V_n$. 
    We verify computationally that $\dim \sum_{i=1}^2\RR[\msf S_{2^3}]\varphi_2(E_i^{(2)}) = \dim \vct V_3$, see the \href{https://github.com/eitangl/anyDimCvxSets/blob/main/check_graphon_gen_deg.m}{GitHub repository}.
    Therefore, 
    \begin{equation}\label{eq:graphon_gen_deg}
        E_i^{(3)} = \sum\nolimits_{j=1}^2r_{i,j}\varphi_2(E_j^{(2)}),\quad \textrm{ for } i\in[2],\ r_{i,j}\in \RR[S_{2^3}].
    \end{equation}
    Let $\psi_n(X) = X\oplus 0$ be an embedding of $\vct V_n$ into $\vct V_{n+1}$ by zero-padding, and note that $E_i^{(n+1)}=\psi_n(E_i^{(n)})$ and that $\psi_n$ commutes with a different embedding of $\msf S_{2^n}$ into $\msf S_{2^{n+1}}$, namely, one that sends $g\mapsto g\oplus I_{2^n}$. Applying $\psi_n$ to~\eqref{eq:graphon_gen_deg}, we conclude that $E_i^{(n+1)}$ can be written as $\RR[\msf S_{2^{n+1}}]$-linear combinations of $\varphi_n(E_j^{(n)})$ for all $n\geq 2$, hence that $\RR[\msf S_{2^{n+1}}]\varphi_n(\vct V_n)=\vct V_{n+1}$ for all $n\geq 2$.
\end{proof}

\paragraph{Graphon parameters:} A permutation-invariant and intersection-compatible sequence of functions $\mfk f=\{f_n\colon \vct V_n\to \RR\}$ is called a \emph{graphon parameter}, since these are precisely the functions of graphs that do not depend on their inputs' vertex labels, and that only depend on their input graphs via their associated graphons.
A family of graphon parameters that plays a central role in the theory of graphons and in extremal combinatorics are graph homomorphism densities~\cite{lovasz2012large}. Their convexity is related to weakly-norming graphs and Sidorenko's conjecture, a major open problem in extremal combinatorics~\cite{sidorenko1991inequalities,lee2021convex}.

We can obtain parametric families of graphon parameters by taking the gauge functions of parametric families of intersection-compatible and freely-described convex sets.
For example, let $\mscr U = \mbb S^k\otimes \mscr V$ with cones $\mscr K = \{\cvx K_n=(\mbb S^k\otimes \mbb S^{2^n})_+\}$ and $\mscr W = \{\vct W_n=0\}$.
Using Proposition~\ref{prop:compatible_descriptions}(a), we get the following parametric family of freely-described and intersection-compatible sets, parametrized by $L_1,\ldots,L_7\in\mbb S^k$:
\begin{equation}\label{eq:graphon_sets}\begin{aligned}
    \cvx C_n = &\Bigg\{X\in\mbb S^{2^n}:\frac{\mathbbm{1}^\top X\mathbbm{1}}{2^{2n}} L_1\otimes\mathbbm{1}\mathbbm{1}^\top + \frac{\mathrm{Tr}(X)}{2^n}L_2\otimes \mathbbm{1}\mathbbm{1}^\top + L_3\otimes\frac{1}{2^n}\left(X\mathbbm{1}\mathbbm{1}^\top + \mathbbm{1}\mathbbm{1}^\top X\right)\\ &\ + L_4\otimes \left(\diag(X)\mathbbm{1}^\top + \mathbbm{1}\diag(X)^\top\right) + L_5\otimes X + L_6\otimes \mathbbm{1}\mathbbm{1}^\top + L_7\otimes (2^nI_{2^n}) \succeq0\Bigg\},\quad L_7\succeq0.
\end{aligned}\end{equation} 
\new{Note that all the functions of $X$ appearing in the above description can be interpreted as natural functions of the associated step graphon $W_X$. For example, $\frac{\mathbbm{1}^\top X\mathbbm{1}}{2^{2n}}=\int_{[0,1]^2}W_X(t,s)\, \mathrm{d}t\, \mathrm{d}s$ and $\frac{\mathrm{Tr}(X)}{2^n}=\int_0^1W_X(t,t)\, \mathrm{d}t$. Similarly, $\frac{1}{2^n}(X\mathbbm{1})_i = \int_{(i-1)/2^n}^{i/2^n}W_X((i-1/2)/2^n,y)\, dy$ gives the values of the degree function of the graphon $W_X$, and $\mathrm{diag}(X)$ gives the values of the diagonal function $W_X(x,x)$ as $x$ ranges over $[0,1]$, see~\cite[Chap.~7]{lovasz2012large}.}

Expressing the above \new{dimension-}free description in terms of graphons yields a description of a limiting convex set by Theorem~\ref{thm:descriptions_of_lims}. Its gauge function is, in turn, a convex graphon parameter extending continuously to the corresponding limit. Endow $\vct V_{\infty}$ with the $L^{\infty}([0,1]^2)$-norm, noting that $\|\mc P_nW\|_{\infty}\leq \|W\|_{\infty}$ for all $W\in L^{\infty}([0,1]^2)$ by Jensen's inequality. 
We view elements $[W_{i,j}(x,y)]_{i,j=1}^k\in \mbb S^k\otimes L^{\infty}([0,1]^2)$ as symmetric matrices with entries in $L^{\infty}([0,1]^2)$ (so $W_{j,i}=W_{i,j}$), and denote 
\begin{equation}\label{eq:psd_kernel}
    [W_{i,j}(x,y)]_{i,j=1}^k\succeq0\quad \textrm{ if }\quad \int_{[0,1]^2}\sum_{i,j=1}^kW_{i,j}(x,y)f_i(x)f_j(y)\, dx\, dy\geq0\quad \textrm{ for all }\quad f_1,\ldots,f_k\in L^2([0,1]).
\end{equation}
Remarkably, convex graphon parameters that extend continuously to this limit are also projection-compatible by~\cite[Thm.~3.17]{dolevzal2018cut}.
\begin{proposition}\label{prop:graphon_descrip_prop}
     If $L_7=0$ in~\eqref{eq:graphon_sets} then
    \begin{equation*}\begin{aligned}
        \overline{C_{\infty}} = &\Bigg\{W\in \overline{\vct V_{\infty}}: \overline{A_{\infty}}(W) + u_{\infty} \coloneqq \Big[(L_1)_{i,j}\int_{[0,1]^2}W(s,t)\, ds\, dt + (L_2)_{i,j}\int_{[0,1]}W(t,t)\, dt\\ &+ (L_3)_{i,j}\int_{[0,1]}[W(x,t) + W(t,y)]\, dt + (L_4)_{i,j}[W(x,x) + W(y,y)] + (L_5)_{i,j}W(x,y)\Big]_{i,j=1}^k + L_6 \succeq0\Bigg\},
    \end{aligned}\end{equation*}
    where $\overline{\vct V_{\infty}}\subseteq L^{\infty}([0,1]^2)$ is a subspace of symmetric bounded measurable functions on $[0,1]^2$.
\end{proposition}
\begin{proof}
    Endow $\vct U_{\infty} = \mbb S^k\otimes \vct V_{\infty}$ with the norm $\|[W_{i,j}]_{i,j=1}^k\| = \max_{i,j}\|W_{i,j}\|_{\infty}$, so $\overline{\vct U_{\infty}}=\mbb S^k\otimes \overline{\vct V_{\infty}}$. Then $\overline{\cvx K_{\infty}}$ is the cone of PSD matrices $[W_{i,j}(x,y)]_{i,j=1}^k\succeq 0$ in $\overline{\vct U_{\infty}}$. 
    Furthermore, for any $W\in \vct V_{\infty}$ we have
    \begin{equation*}
        \|A_{\infty}\|\leq \max_{i,j}\Big(|(L_1)_{i,j}| + |(L_2)_{i,j}| + 2|(L_3)_{i,j}| + 2|(L_4)_{i,j}| + |(L_5)_{i,j}|\Big)\|W\|_{\infty},
    \end{equation*}
    so $A_{\infty}$ extends continuously to the limit, and $u_{\infty}=L_6\in U_1\subseteq \overline{\vct U_{\infty}}$ satisfies $\mc P_nu_{\infty}=L_6$ for all $n$. 
    Thus, Theorem~\ref{thm:descriptions_of_lims} yields the claim.
\end{proof}
We require $L_7=0$ since there is no $u_{\infty}\in\overline{\vct V_{\infty}}$ satisfying $\mc P_nu_{\infty}=2^nI_{2^n}$ for all $n$.

\subsection{Compatibility in Inverse Problems}\label{sec:inverse_probs}
Our compatibility conditions naturally arise in the context of inverse problems, where we demonstrate that it is desirable to use regularizers which are both intersection- and projection-compatible.

Consider a consistent sequence $\mscr V=\{\vct V_n\}$ of $\{\msf G_n\}$-representations. A popular approach to recover $x\in \vct V_n$ from $m\in\NN$ linear observations takes as input a forward map $A\colon \vct V_n\to \RR^m$ and data $y\in \RR^m$ and outputs 
\begin{equation}\label{eq:inv_prob_recovery}
    F_{m,n}(A,y) = \argmin{x\in \vct V_n}\ f_n(x) + \lambda\|Ax - y\|_2^2,
\end{equation}
where $f_n\colon \vct V_n\to\RR$ is a convex regularizer promoting desired structure in the solution. 
It is desirable for the maps $F_{m,n}$ defined in~\eqref{eq:inv_prob_recovery}---which can be instantiated for any $(A,y)\in\mc L(\vct V_n,\RR^m)\oplus\RR^m$ and for any $n,m\in\NN$---to satisfy
\begin{equation}\label{eq:argmin_compatibility}
    F_{m,n+1}(A\mc P_n, y) = F_{m,n}(A,y),    
\end{equation}
whenever the corresponding minimizers are unique. Indeed, condition~\eqref{eq:argmin_compatibility} requires the recovered solution to lie in $\vct V_n$ if the data only depends on the component of $x\in \vct V_{n+1}$ in $\vct V_n$, as this property avoids overfitting. 
Condition~\eqref{eq:argmin_compatibility} holds when the sequence of regularizers is both intersection and projection-compatible. Indeed, if the sequence of regularizers $\mfk f = \{f_n\}$ is \emph{projection-compatible}, then 
\begin{equation*}
        \min_{\widetilde x\in \vct V_{n+1}} f_{n+1}(\widetilde x) + \lambda\|A\mc P_n \widetilde x - b\|_2^2 = \min_{x\in \vct V_n}\min_{\substack{\widetilde x\in \vct V_{n+1}\\ \mc P_n\widetilde x = x}} f_{n+1}(\widetilde x) + \lambda \|Ax - y\|_2^2 = \min_{x\in \vct V_n}f_n(x) + \lambda\|Ax - y\|_2^2
\end{equation*}
Moreover, if $x_{\ast}=F_{m,n}(A,y)$ minimizes $f_n(x)+\lambda\|Ax-y\|_2^2$ and $\mfk f$ is \emph{intersection-compatible}, then $f_n(x_{\ast})+\lambda\|Ax_{\ast}-y\|_2^2 = f_{n+1}(x_{\ast}) + \lambda\|A\mc P_nx_{\ast}-y\|_2^2$ and hence $x_{\ast} = F_{m,n+1}(A\mc P_n,y)$, showing~\eqref{eq:argmin_compatibility}.

\section{Constant-Sized Invariant Conic Programs}\label{sec:const_sized_progs}
In the previous sections, we studied freely-described sequences of convex sets $\{\cvx C_n\}$ contained in a consistent sequence $\mscr V$. These convex sets are group-invariant whenever the cones $\cvx K_n$ in their descriptions~\eqref{eq:preim_seq} are group-invariant, which is the case for all the standard sequences of cones we consider. 
In this section, we further consider optimizing sequences of invariant linear functionals over such sequences of sets.
Each program in the sequence can be simplified by restricting its domain to invariant vectors~\cite[\S3]{GATERMANN200495}.
As we have seen in Proposition~\ref{prop:gen_imp_surj}, when $\mscr V$ is finitely-generated the dimensions of its spaces of invariants stabilize, so the size of the variables in such programs stabilizes as well. 
However, the size of the constraints may not stabilize, because the invariant sections of the cones $\{\cvx K_n^{\msf G_n}\}$ may grow in complexity.
For example, if $\cvx K_n$ is the cone of $n$-variate degree $k$ polynomials that are nonnegative over all of $\RR^n$ and $\msf G_n=\msf S_n$, the best-known description of $\cvx K_n^{\msf G_n}$ has complexity which is a (nonconstant) polynomial in $n$~\cite{timofte2003positivity,riener2016symmetric,ACEVEDO20162936}.
We therefore seek conditions for the existence of \emph{constant-sized} descriptions for $\{\cvx K_n^{\msf G_n}\}$, and bounds on the value of $n$ after which the size stabilizes in the sense of Definition~\ref{def:const_sized_descrp}.
Constant-sized descriptions for symmetric PSD and relative entropy cones have been obtained on a case-by-case basis in the literature~\cite{riener2013exploiting,raymond2018symmetric,debus2020reflection,moustrou2021symmetry}.
In this section, we explain how these results can be generalized and derived systematically from an interplay between representation stability and the structure of the cones in question.

\subsection{The PSD Cone and Variants}\label{sec:psd_cone_const_size}
We begin by giving constant-sized descriptions for certain sequences of PSD cones. We do so by deriving constant-sized bases for spaces of invariants in terms of which membership in the cones is simply expressed. The following is a precise statement of Theorem~\ref{thm:informal} stated informally in Section~\ref{sec:intro}.
\begin{theorem}\label{thm:sym2_const_size}
Let $\mscr V_0=\{\RR^n\}$ with $\msf G_n=\msf B_n,\msf D_n$, or $\msf S_n$ as in Example~\ref{ex:basic_consist_seq}, and let $\mscr V=\{\vct V_n\}$ be a $\mscr V_0$-module generated in degree $d$ and presented in degree $k$. Then the cones $\{\mathrm{Sym}^2_+(\vct V_n)^{\msf G_n}\}$ admit constant-sized descriptions for $n\geq k+d$.
\end{theorem}
\begin{proof}
By Theorem~\ref{thm:rep_stability_FIW}, there exists a finite set $\Lambda$ satisfying 
\begin{equation*}
    \vct V_n = \bigoplus_{\lambda\in\Lambda} \underbrace{\vct W_{\lambda[n]}^{m_{\lambda}}}_{=:\vct V_{\lambda[n]}}    
\end{equation*}
where $\vct W_{\lambda[n]}$ is a $\msf G_n$-irreducible and $\vct V_{\lambda[n]}$ is the corresponding isotypic component.
Invariant elements of $\vct U_n=\mathrm{Sym}^2(\vct V_n)$ are equivariant and self-adjoint endomorphisms of $\vct V_n$. 
If $X\in \vct U_n^{\msf G_n}$ is such an endomorphism and $\lambda\neq\mu\in \Lambda$ index distinct irreducibles, then $\mc P_{\vct V_{\mu[n]}}X|_{\vct V_{\lambda[n]}}=0$ by Schur's Lemma~\cite[\S1.2]{fulton2013representation}. 
Because the irreducibles of $\msf G_n=\msf B_n,\msf D_n,$ and $\msf S_n$ are of real type~\cite{WANG1988299} (meaning they remain irreducible when complexified), Schur's Lemma also implies that $\mc P_{\vct W_{\lambda[n]}}X|_{\vct W_{\lambda[n]}}$ is a multiple of the identity for each $\lambda\in\Lambda$, hence $\mc P_{\vct V_{\lambda[n]}}X|_{\vct V_{\lambda[n]}}=X_{\lambda}\otimes I_{\dim \vct W_{\lambda[n]}}$ for some $X_{\lambda}\in \mbb S^{m_{\lambda}}$.
We conclude that there exists an orthogonal matrix $Q_n$ depending on the irreducible decomposition of $\vct V_n$ satisfying
\begin{equation}\label{eq:blkdiag_psd_cone}
    \vct U_n^{\msf G_n} = \left\{Q_n\left(\bigoplus_{\lambda\in\Lambda}(X_{\lambda}\otimes I_{\dim \vct W_{\lambda[n]}})\right)Q_n^*:X_{\lambda}\in\mbb S^{m_{\lambda}}\right\},
\end{equation}
hence
\begin{equation}\label{eq:psd_cones_const_size}
    \mathrm{Sym}^2_+(\vct V_n)^{\msf G_n} = \left\{X\in \vct U_n^{\msf G_n}:X\succeq0\right\} = \left\{Q_n\left(\bigoplus_{\lambda\in\Lambda_n}(X_{\lambda}\otimes I_{\dim \vct W_{\lambda}})\right)Q_n^*:X_{\lambda}\in\mbb S^{m_{\lambda}}_+\right\}.
\end{equation}
Thus, we obtain constant-sized descriptions by defining $\vct U=\bigoplus_{\lambda\in\Lambda}\mbb S^{m_{\lambda}}$ and $T_n\colon \vct U\to \vct U_n^{\msf G_n}$ sending $(X_{\lambda})_{\lambda\in\Lambda}$ to $Q_n\bigoplus_{\lambda\in\Lambda}(X_{\lambda}\otimes I_{\dim \vct W_{\lambda}})Q_n^*$, which maps $\cvx K=\bigoplus_{\lambda\in\Lambda}\mbb S^{m_{\lambda}}_+$ onto $\mathrm{Sym}^2_+(\vct V_n)^{\msf G_n}$. 
\end{proof}
We now instantiate $\mscr V$ to obtain more concrete corollaries.
\begin{corollary}\label{cor:const_size_moment_mats}
    If $\msf G_n=\msf S_n,\msf D_n$ or $\msf B_n$ acts on $\RR^n$ as in Example~\ref{ex:basic_consist_seq}, then the cones $\mathrm{Sym}^2_+(\mathrm{Sym}^{\leq d}\RR^n)^{\msf G_n}\cong \left(\mbb S^{\binom{n+k}{k}}_+\right)^{\msf G_n}$ admit constant-sized descriptions by~\eqref{eq:psd_cones_const_size} for $n\geq 2d$ if $\msf G_n=\msf S_n, \msf B_n$ and $n\geq 2d+1$ if $\msf G_n=\msf D_n$.
\end{corollary}
\begin{proof}
    The sequence $\mscr V = \mathrm{Sym}^{\leq d}\mscr V_0$ is generated in degree $d$ and presented in degree $d$ if $\msf G_n=\msf S_n,\msf B_n$ or in degree $d+1$ if $\msf G_n=\msf D_n$, by Theorem~\ref{thm:calc_for_pres_degs_FIW} and Example~\ref{ex:basic_consist_seq_degs}.
\end{proof}

To obtain constant-sized descriptions for cones of invariant sums-of-squares, we consider equivariant images of the PSD cones above.
\begin{proposition}\label{prop:proj_of_const_cones}
Let $\mscr U = \{\vct U_n\}$ and $\mscr W=\{\vct W_n\}$ be sequences of $\{\msf G_n\}$-representations (not necessarily consistent). Let $\{\cvx K_n\subseteq \vct U_n\}$ be a sequence of convex cones such that $\cvx K_n$ is $\msf G_n$-invariant for each $n$. If $\{\cvx K_n^{\msf G_n}\}$ admits constant-sized descriptions for $n\geq t$, then so does $\{\pi_n(\cvx K_n\cap \vct L_n)^{\msf G_n}\subseteq \vct W_n\}$ for any sequence $\{\pi_n\in\mc L(\vct U_n,\vct W_n)^{\msf G_n}\}$ and any sequence of $\msf G_n$-invariant subspaces $\vct L_n\subseteq \vct U_n$.
\end{proposition}
\begin{proof}
Suppose $\cvx K_n^{\msf G_n}=T_n(\cvx K\cap \vct L_n')$ for $n\geq t$ where $T_n\colon \vct U\to \vct U_n^{\msf G_n}$ and $\vct L_n'\subseteq \vct U$ are subspaces as in Definition~\ref{def:const_sized_descrp}. 
Because $\pi_n$ is $\msf G_n$-equivariant,
\begin{equation*}
    \pi_n(\cvx K_n\cap \vct L_n)^{\msf G_n}=\pi_n(\cvx K_n^{\msf G_n}\cap \vct L_n^{\msf G_n}) = (\pi_n\circ T_n)(\cvx K\cap \vct L_n'\cap T_n^{-1}(\vct L_n^{\msf G_n})).
\end{equation*}
Noting that $\vct L_n'\cap T_n^{-1}(\vct L_n^{\msf G_n})$ is a subspace of $\vct U$, we get the desired constant-sized descriptions.
\end{proof}
We now prove Theorem~\ref{thm:sym_sos_const_size} giving constant-sized descriptions for invariant sums-of-squares.
s the vector whose coordinates are all the monomials in 
\begin{proof}[\textbf{Proof (Theorem~\ref{thm:sym_sos_const_size})}]
    If $v(x)$ is the vector whose coordinates are all the monomials \new{of degree at most $d$ in the variables $x_{i_1,\ldots,i_k}$ with $1\leq i_1\leq\ldots\leq i_k\leq n$}, then~\cite[Thm.~3.39]{SOS_chapter} yields
    \begin{equation*}
        \mathrm{SOS}_{\vct U_n} = \pi_n\left(\mathrm{Sym}_+^2\left(\mathrm{Sym}^{\leq d}\left(\new{\mathrm{Sym}^k}\RR^n\right)\right)\right),\quad \pi_n(M) = v(x)^\top Mv(x) + \mc I_n.
    \end{equation*}
    The map $\pi_n\colon\mathrm{Sym}^2\left(\mathrm{Sym}^{\leq d}\left(\new{\mathrm{Sym}^k}\RR^n\right)\right)\to \vct U_n$ is equivariant by definition of the action of $\msf G_n$ and the invariance of $\mc I_n$. 
    If $\mscr V_0=\{\RR^n\}$ as in Example~\ref{ex:basic_consist_seq}, then $\mscr V = \mathrm{Sym}^{\leq d}\left(\new{\mathrm{Sym}^k}\mscr V_0\right)$ is a $\mscr V_0$-module generated in degree $kd$, and presented in degree $kd$ if $\msf G_n=\msf S_n,\msf B_n$ and in degree $kd+1$ if $\msf G_n=\msf D_n$ by Theorem~\ref{thm:calc_for_pres_degs_FIW}. 
    Thus, the result follows from Theorem~\ref{thm:sym2_const_size} and Proposition~\ref{prop:proj_of_const_cones}.
\end{proof}
Note that Theorem~\ref{thm:sym_sos_const_size} applies to \emph{any} sequence of invariant ideals, not necessarily related to each other across dimensions, so that $\{\vct U_n\}$ in this result is not necessarily a consistent sequence. Nevertheless, it is a sequence of equivariant images of a consistent sequence, a fact crucial to the proof.

\subsection{The Relative Entropy Cone and Variants}\label{sec:rel_ent_cones}
For a finite set $\mc A$, define the associated relative entropy cone
\begin{equation}\label{eq:rel_ent_cone}
    \mathrm{RE}_{\mc A} = \{(\nu, c, t)\in \RR^{\mc A}\oplus\RR^{\mc A}\oplus\RR: \nu,c\geq 0,\ D(\nu, c)\leq t\},
\end{equation}
where $D(\nu,c)=\sum_{\alpha\in\mc A}\nu_{\alpha}\log\left(\frac{\nu_{\alpha}}{c_{\alpha}}\right)$ is the relative entropy. 
\begin{proposition}\label{prop:rel_ent_const_size}
    Let $\mscr V_0$ be a consistent sequence of $\{\msf G_n\}$ representations and $\mscr V=\{\RR^{\mc A_n}\}$ be a permutation $\mscr V_0$-module for finite $\mc A_n$ satisfying $\mc A_n \subseteq \mc A_{n+1}$ (Definition~\ref{def:perm_rep}).  
    If $\dim \vct V_n^{\msf G_n}$ is constant for all $n\geq d$, then the invariant section $\{\mathrm{RE}_{\mc A_n}^{\msf G_n}\}$ of the cones~\eqref{eq:rel_ent_cone} admit constant-sized descriptions for $n\geq d$. 
\end{proposition}
\begin{proof}
Let $k=\dim (\RR^{\mc A_d})^{\msf G_d}$ and fix $n\geq d$. Let $\{\alpha_j\}_{j\in[k]}\subseteq \mc A_n$ be a set of $\msf G_n$-orbit representatives, and let $\mathbbm{1}_{j,n}=\sum_{g\in \msf G_n/\msf{Stab}_{\msf G_n}(\alpha_j)}e_{g\alpha_j}$ for each $j\in[k]$, so that $\{\mathbbm{1}_{j,n}\}_{j\in[k]}$ is a basis for $(\RR^{\mc A_n})^{\msf G_n}\cong\RR^k$.
Let $\mscr U = \mscr V^{\oplus 2}\oplus \RR=\{\vct U_n\}$, which contains the relevant relative entropy cones $\{\mathrm{RE}_{\mc A_n}\}$.
Then a basis for $\vct U_n^{\msf G_n}$ consists of $\{(\mathbbm{1}_{j,n},0,0),(0,\mathbbm{1}_{j,n},0)\}_{j\in[k]}\cup\{(0,0,1)\}$. 

If $(\nu,c,t)\in \vct U_n^{\msf G_n}$ for $n\geq d$ is expanded as $\nu=\sum_{j=1}^k\hat \nu_j\mathbbm{1}_{j,n}$ and similarly for $c$, then
\begin{equation*}
    \mathrm{RE}_{\mc A_n}^{\msf G_n} = \left\{(\nu, c, t)\in \vct U_n^{\msf G_n}: \hat \nu,\hat c\geq 0,\ \sum_{j=1}^k|\msf G_n/\msf{Stab}_{\msf G_n}(\alpha_j)| \hat\nu_j\log\left(\frac{\hat \nu_j}{\hat c_j}\right)\leq t\right\}.
\end{equation*}
Let $\vct U = \RR^{L}\oplus\RR^{L}\oplus\RR$, define
\begin{equation*}
    \cvx K = \left\{(\hat \nu, \hat c, t)\in \RR^{L}\oplus\RR^{L}\oplus\RR: \hat \nu,\hat c\geq 0,\ \sum_{j=1}^k\hat\nu_j\log\left(\frac{\hat \nu_j}{\hat c_j}\right)\leq t\right\},
\end{equation*}
and define $T_n\colon \vct U\to \vct U_n^{\msf G_n}$ sending $(e_j,0,0)\mapsto |\msf G_n/\msf{Stab}_{\msf G_n}(\alpha_j)|^{-1}(\mathbbm{1}_{j,n},0,0)$, sending $(0,e_j,0)\mapsto$ \linebreak $|\msf G_n/\msf{Stab}_{\msf G_n}(\alpha_j)|^{-1}(0,\mathbbm{1}_{j,n},0)$, and sending $(0,0,1)\mapsto (0,0,1)$. Then $\mathrm{RE}_{\mc A_n}^{\msf G_n} = T_n(\cvx K)$ for all $n\geq d$, giving the desired constant-sized descriptions.
\end{proof}

Now suppose $\mscr W=\{\RR^{\mc B_n}\}$ is another permutation $\mscr V_0$-module. Let $\widetilde{\mscr U} = \mscr W\otimes\mscr U = \{\widetilde{\vct U_n}=\mc L(\RR^{\mc B_n},\vct U_n)\}$ and consider the cones of maps
\begin{equation}\label{eq:rel_ent_maps_cone}
    \mathrm{REM}_{\mc A_n,\mc B_n} = \left\{M\in \widetilde{\vct U_n}:M(\RR^{\mc B_n}_+)\subseteq \mathrm{RE}_{\mc A_n}\right\} = \left\{M\in \widetilde{\vct U_n}: Me_{\beta}\in \mathrm{RE}_{\mc A_n} \textrm{ for all } \beta\in\mc B_n\right\}.
\end{equation}
We obtain constant-sized descriptions for these cones for a specific $\mscr V_0$.
\begin{proposition}\label{prop:rel_ent_map_cone_const_size}
    Suppose $\mscr V_0=\{\RR^n\}$ with $\msf G_n=\msf B_n,\msf D_n,$ or $\msf S_n$ as in Example~\ref{ex:basic_consist_seq} and that $\{\RR^{\mc A_n}\},\ \{\RR^{\mc B_n}\}$ are permutation $\mscr V_0$-modules generated in degrees $d_V,d_W$, respectively. Then the sequence of cones $\{\mathrm{REM}_{\mc A_n,\mc B_n}^{\msf G_n}\}$ in~\eqref{eq:rel_ent_maps_cone} admits constant-sized descriptions for $n\geq d_V+d_W$ if $\msf G_n=\msf S_n,\msf B_n$ and $n\geq d_V+d_W+1$ if $\msf G_n=\msf D_n$.
\end{proposition}
\begin{proof} 
Set $d_0=0$ if $\msf G_n=\msf S_n,\msf B_n$ and $d_0=1$ if $\msf G_n=\msf D_n$. Let $\hat{\mc B}\subseteq\mc B_{d_W}$ be a set of minimal-degree $\msf G_{d_W+d_0}$-orbit representatives in $\mc B_{d_W+d_0}$, which are also orbit representatives for $\mc B_n$ for all $n\geq d_W+d_0$ by Proposition~\ref{prop:pres_deg_perm_mod}(c).
Any $M\in \widetilde{\vct U_n}^{\msf G_n} = \mc L(\RR^{\mc B_n},\vct U_n)^{\msf G_n}$ for $n\geq d_W+d_0$ is fully determined by the images $Me_{\beta}\in \vct U_n^{\msf{Stab}_{\msf G_n}(\beta)}$ of the basis elements $e_{\beta}$ for $\beta\in\hat{\mc B}$ and conversely, for any collection $\left\{u_{\beta}\in \vct U_n^{\msf{Stab}_{\msf G_n}(\beta)}\right\}_{\beta\in\hat{\mc B}}$ there is a unique $M\in \widetilde{\vct U_n}^{\msf G_n}$ satisfying $Me_{\beta}=u_{\beta}$. Moreover, $M\in \widetilde{\cvx K_n}$ if and only if $Me_{\beta}\in \cvx K_n^{\msf{Stab}_{\msf G_n}(\beta)}$ for all $\beta\in\hat{\mc B}$. Thus, we have
\begin{equation*}
    \widetilde{\vct U_n}^{\msf G_n} = \bigoplus_{\beta\in\hat{\mc B}}\vct U_n^{\msf{Stab}_{\msf G_n}(\beta)} \supseteq \bigoplus_{\beta\in\hat{\mc B}}\mathrm{RE}_{\mc A_n}^{\msf{Stab}_{\msf G_n}(\beta)} = \mathrm{REM}_{\mc A_n,\mc B_n}^{\msf G_n}.
\end{equation*}
Applying Corollary~\ref{cor:stab_deg_for_stabs} to the free module with which $\{\RR^{\mc A_n}\}$ agrees for $n\geq d_V+d_0$ by Proposition~\ref{prop:pres_deg_perm_mod}(c), and which is presented in the same degree, we conclude that the projections $\mc P_{\vct V_n}\colon (\RR^{\mc A_{n+1}})^{\msf{Stab}_{\msf G_{n+1}(\beta)}} \to (\RR^{\mc A_n})^{\msf{Stab}_{\msf G_n}(\beta)}$ are isomorphisms for all $n\geq d_V+d_W+d_0$.
Proposition~\ref{prop:rel_ent_const_size} then gives constant-sized descriptions for $\left\{\mathrm{RE}_{\mc A_n}^{\msf{Stab}_{\msf G_n}(\beta)}\right\}_n$ for each $\beta\in\hat{\mc B}$.
\end{proof}

As an application of Proposition~\ref{prop:rel_ent_map_cone_const_size}, we obtain constant-sized descriptions for SAGE cones of signomials. Indeed, if $\mc A_n,\mc B_n\subseteq\RR^n$ as in the above proposition, define the sequence $\mscr F=\{\vct F_n\}$ of functions on $\{\RR^n\}$
\begin{equation*}
    \vct F_n = \left\{f(x) = \sum_{\alpha\in\mc A_n}c_{\alpha}e^{\langle \alpha, x\rangle} + \sum_{\beta\in\mc B_n}t_{\beta}e^{\langle \beta,x\rangle}:c_{\alpha},t_{\beta}\in\RR\right\} \cong \RR^{\mc A_n}\oplus \RR^{\mc B_n},
\end{equation*}
with $\msf G_n$ acting by $g\cdot f = f\circ g^{-1}$. Note that $\mscr F = \mscr V\oplus\mscr W$ as consistent sequences, where $\mscr V = \{\RR^{\mc A_n}\}$ and $\mscr W=\{\RR^{\mc B_n}\}$ are as above.
Sums of exponentials as in $\vct F_n$ are called signomials, and optimization problems involving such functions arise in a number of applications~\cite{murray2021signomial}. As usual, minimizing a signomial $f$ over $\RR^n$ can be recast as maximizing $\gamma\in\RR$ such that $f-\gamma\geq 0$ on $\RR^n$, so that optimizing signomials can be reduced to certifying their nonnegativity. This is NP-hard in general, but it can be done efficiently if only a single coefficient of $f$ is nonnegative or if $f$ is a sum of such signomials~\cite{sage}.  Formally, define the cones of (Sums of) nonnegative AM/GM Exponential functions, called AGE (resp., SAGE) functions, by
\begin{align*}
    &\mathrm{AGE}_{\mc A_n,\beta} = \left\{f(x)=\sum_{\alpha\in\mc A_n}c_{\alpha}e^{\langle \alpha, x\rangle} + te^{\langle \beta,x\rangle}:f\geq 0 \textrm{ on } \RR^n \textrm{ and } c_{\alpha}\geq 0 \textrm{ for all } \alpha\in\mc A_n\right\},\\
    &\mathrm{SAGE}_{\mc A_n,\mc B_n} = \sum_{\beta\in\mc B_n}\mathrm{AGE}_{\mc A_n,\beta}.
\end{align*}
\begin{theorem}\label{thm:sage_cones_const_size}
    Suppose $\mc A_n,\mc B_n\subseteq \RR^n$ where $\RR^n$ is embedded in $\RR^{n+1}$ by zero-padding and with the standard action of $\msf G_n=\msf S_n,\msf D_n$ or $\msf B_n$. If $\mc A_n=\bigcup_{g\in \msf G_n}g\mc A_{d_A}$ for all $n\geq d_A$ and $\mc B_n=\bigcup_{g\in \msf G_n}g\mc A_{d_B}$ for all $n\geq d_B$, then the invariant SAGE cones $\{\mathrm{SAGE}_{\mc A_n,\mc B_n}^{\msf G_n}\}_n$ admit constant-sized descriptions for $n\geq d_A+d_B$ if $\msf G_n=\msf S_n,\msf B_n$ and $n\geq d_A+d_B+1$ if $\msf G_n=\msf D_n$.
\end{theorem}
\begin{proof}
Identify $M\in\widetilde{\vct U_n}$ with tuples $(Me_{\beta})_{\beta\in\mc B_n} = (\nu^{(\beta)},c^{(\beta)},t_{\beta})_{\beta\in\mc B_n}\in \RR^{\mc A_n}\oplus\RR^{\mc A_n}\oplus\RR$ for each $\beta\in \mc B_n$. The authors of~\cite{sage} show that, in our notation,
\begin{align*}
    \mathrm{SAGE}_{\mc A_n,\mc B_n} &= \Big\{(c,t)\in\RR^{\mc A_n}\oplus\RR^{\mc B_n}: \exists M=(\nu^{(\beta)},c^{(\beta)},t_{\beta})_{\beta\in\mc B_n}\in \mathrm{REM}_{\mc A_n,\mc B_n} \textrm{ s.t. } \sum\nolimits_{\beta\in\mc B_n}c^{(\beta)} = c,\\ &\qquad \sum\nolimits_{\alpha\in\mc A_n}\nu^{(\beta)}_{\alpha}(\alpha-\beta)=0 \textrm{ for all } \beta\in\mc B_n\Big\}\\
    &= \pi_n(\mathrm{REM}_{\mc A_n,\mc B_n}\cap \vct L_n),
\end{align*}
where 
\begin{align*}
    \vct L_n &= \left\{M=(\nu^{(\beta)},c^{(\beta)},t_{\beta})_{\beta\in\mc B_n}\in \widetilde{\vct U_n}: \sum\nolimits_{\alpha\in\mc A_n}\nu^{(\beta)}_{\alpha}(\alpha-\beta)=0 \textrm{ for all } \beta\in\mc B_n\right\},\\
    \pi_n(M) &= \left(\sum\nolimits_{\beta\in\mc B_n}c^{(\beta)},(t_{\beta})_{\beta\in\mc B_n}\right)\in\RR^{\mc A_n}\oplus\RR^{\mc B_n}.
\end{align*}
Note that $\pi_n$ is equivariant, since 
\begin{equation*}
    (g\cdot M)e_{\beta}=gMg^{-1}e_{\beta}=gMe_{g^{-1}\beta} = g\cdot (\nu^{(g^{-1}\beta)}, c^{(g^{-1}\beta)},t_{g^{-1}\beta})=(g\cdot \nu^{(g^{-1}\beta)},g\cdot c^{(g^{-1}\beta)}, t_{g^{-1}\beta}),
\end{equation*}
hence
\begin{equation*}
    \pi_n(g\cdot M) = \left(\sum_{\beta\in\mc B_n}g\cdot c^{(g^{-1}\beta)}, (t_{g^{-1}\beta})_{\beta\in\mc B_n}\right) = \left(g\cdot \sum_{\beta\in\mc B_n}c^{(\beta)}, g\cdot(t_{\beta})_{\beta\in\mc B_n}\right) = g\cdot \pi_n(M).
\end{equation*}
Similarly, if $\sum_{\alpha\in\mc A_n}\nu^{(\beta)}_{\alpha}(\alpha-\beta)=0$ for all $\beta\in\mc B_n$ then
\begin{equation*}
    \sum_{\alpha\in\mc A_n}(g\cdot \nu^{(g^{-1}\beta)})_{\alpha}(\alpha-\beta) = g\sum_{\alpha\in\mc A_n}\nu^{(g^{-1}\beta)}_{g^{-1}\alpha}(g^{-1}\alpha-g^{-1}\beta) = g\sum_{\alpha\in\mc A_n}\nu^{(g^{-1}\beta)}_{\alpha}(\alpha-g^{-1}\beta) = 0,
\end{equation*}
hence $\vct L_n$ is $\msf G_n$-invariant.
Thus, the result follows from Proposition~\ref{prop:rel_ent_map_cone_const_size} and Proposition~\ref{prop:proj_of_const_cones}.
\end{proof}

Theorem~\ref{thm:sage_cones_const_size} generalizes~\cite[Thm.~5.3]{moustrou2021symmetry} beyond $\msf S_n$ to the other classical Weyl groups $\msf D_n$ and $\msf B_n$. It would be interesting to further generalize to signomials defined on more general consistent sequences than $\{\RR^n\}$, which would require generalizing the description of the AGE cone from~\cite{sage}.

\section{\new{Dimension-}Free Convex Regression}\label{sec:numerical_examples}
In this section, we use our framework to describe a solution to the \new{dimension-}free convex regression problem.  Recall that in this problem we are given evaluation data in different dimensions and we seek a sequence of convex functions that best fit the data and can be instantiated in any desired dimension (including those not represented in the data).  To that end, we use the framework we developed in Section~\ref{sec:free_and_compatible} to obtain parametric families of freely-described convex sets in a fully algorithmic manner. We then present a numerical procedure to fit elements of these families to the given data.  Our implementation of the resulting algorithms can be found at \url{https://github.com/eitangl/anyDimCvxSets}.



\subsection{Computationally Parametrizing Descriptions}\label{sec:implementation}
Suppose we seek a parametric family of convex subsets $\{\cvx C_n\subseteq \vct V_n\}$ of a consistent sequence $\mscr V = \{\vct V_n\}$ of $\mscr G=\{\msf G_n\}$-representations. Both $\mscr V$ and $\mscr G$ are usually dictated by the application at hand and the symmetries it exhibits, as in Example~\ref{ex:graph_params}.
We then select description spaces $\mscr U=\{\vct U_n\},\mscr W=\{\vct W_n\}$, usually constructed from $\mscr V$ as described in Section~\ref{sec:consist_seq_construct}. Their choice is dictated by the desired richness of the family of the freely-described sets and, as we shall see in Remark~\ref{rmk:identif} below, by the dimensions of the available data.  
Once all of these are chosen, parametrizing the associated family of freely-described convex sets amounts to identifying bases for the spaces of freely-described vectors and linear maps appearing in~\eqref{eq:preim_seq}.
Up to this point we have obtained such bases analytically, as the relevant spaces of invariants have been fairly low-dimensional, but this approach becomes impractical for richer description spaces.  To address this challenge, we present a computational method for obtaining the relevant bases, which yields a fully algorithmic approach for deriving parametrized families of freely-described convex sets.


Computing a basis for the space of freely-described elements in a finitely-generated consistent sequence proceeds in two steps. First, we compute a basis for the space of invariants in a fixed, sufficiently large dimension. Second, we extend the elements of this basis to any other dimension, which is done by solving a linear system.  
Our procedure is summarized in Algorithm~\ref{algo:free_descr}, which we proceed to describe in more detail.

\begin{algorithm}
	\caption{Computationally parametrize a freely-described (and possibly compatible) sequence of convex sets.}
	\label{algo:free_descr}
	\begin{algorithmic}[1]
		\State \textbf{Input:} \new{Finitely-presented} consistent sequences $\mscr V,\mscr W,\mscr U$.
        \State \textbf{Output:} Bases $\{\mscr A^{(i)}=\{A_n^{(i)}\}_{n\in\NN}\}_{i=1}^{d_A}, \{\mscr B^{(j)}=\{B_n^{(j)}\}_{n\in\NN}\}_{j=1}^{d_B}$ for freely-described equivariant maps or morphisms.


        \If{compatibility not required}
        \State Fix $n_0\geq$ presentation degrees of $\mscr V\otimes \mscr U$ and $\mscr W\otimes\mscr U$.
        \State \label{step:equivariant_basis} Find bases $\{A_{n_0}^{(i)}\}_{i=1}^{d_A},\{B_{n_0}^{(j)}\}_{j=1}^{d_B}$ for $\mc L(\vct V_{n_0},\vct U_{n_0})^{\msf G_{n_0}}$ and $\mc L(\vct W_{n_0},\vct U_{n_0})^{\msf G_{n_0}}$. 
        \Else 
        \State Fix $n_0\geq$ presentation degrees of $\mscr V,\mscr W,\mscr U$.
        \State \label{step:extendable_basis} Find bases $\{A_{n_0}^{(i)}\}_{i=1}^{d_A},\{B_{n_0}^{(j)}\}_{j=1}^{d_B}$ for subspaces of $\mc L(\vct V_{n_0},\vct U_{n_0})^{\msf G_{n_0}}$ and $\mc L(\vct W_{n_0},\vct U_{n_0})^{\msf G_{n_0}}$ satisfying the hypotheses of Theorem~\ref{thm:extending_compatible_descriptions}. 
        
        \EndIf

        \State For any $n>n_0$ and each $i,j$, find unique equivariant $A_n^{(i)},B_n^{(j)}$ projecting onto $A_{n_0}^{(i)},B_{n_0}^{(j)}$. For each $n<n_0$, project $A_{n_0}^{(i)},B_{n_0}^{(j)}$ to dimension $n$ to obtain $A_n^{(i)},B_n^{(j)}$.\label{step:extend}
	\end{algorithmic}
\end{algorithm}
Fix $n_0\in\NN$ as in Algorithm~\ref{algo:free_descr}\new{, computed for instance using Theorem~\ref{thm:calc_for_pres_degs_FIW}}. 
We now explain how to find bases for the desired spaces of equivariant linear maps $A_{n_0}$ and $B_{n_0}$ in a fixed dimension, and how to extend them to bases of freely-described elements of $\mscr V\otimes\mscr U$ and $\mscr W\otimes\mscr U$. 
These elaborate on steps~\ref{step:equivariant_basis}, \ref{step:extendable_basis} and~\ref{step:extend} in Algorithm~\ref{algo:free_descr}. 
\paragraph{Step~\ref{step:equivariant_basis}: Computing basis for equivariant maps.}
We explain how to compute a basis for invariants in a fixed vector space, which we then instantiate in the context of Algorithm~\ref{algo:free_descr} to perform step~\ref{step:equivariant_basis}.
If $\vct V$ is a representation of a group $\msf G$, a vector $v\in \vct V$ is $\msf G$-invariant iff $g\cdot v = v$ for all $g\in \msf G$, which can be rewritten as $v\in\ker(g-I)$ for all $g\in \msf G$. 
Thus, finding a basis for invariants in a fixed vector space reduces to finding a basis for the kernel of a matrix, though this matrix may be very large or even infinite.
We can dramatically reduce the size of this matrix by only considering discrete and continuous generators of $\msf G$, as proposed in~\cite{pmlr-v139-finzi21a}.
\begin{theorem}[{\cite[Thm.~1]{pmlr-v139-finzi21a}}]\label{thm:practical_invariance}
    Let $\msf G$ be a real Lie group with finitely-many connected components acting on a vector space $\vct V$ via the homomorphism $\rho\colon \msf G\to \mathrm{GL}(\vct V)$. Let $\{H_i\}$ be a basis for the Lie algebra $\mfk g$ of $\msf G$ and $\{h_j\}\subseteq\msf G$ be a finite collection \new{such that $\msf G$ is generated by $\mathrm{exp}(\mfk g)$ and $\{h_j\}$\footnote{Such $\{h_j\}$ exist by~\cite[Appdx.~H]{pmlr-v139-finzi21a}.}}. Then 
    \begin{equation*}
        v\in \vct V^{\msf G} \iff \D\rho(H_i)v=0 \textrm{  and  } (\rho(h_j)-\mathrm{id}_{\vct V
        })\cdot v=0\quad \textrm{for all } i,j.
    \end{equation*}
\end{theorem}
Here $\D\rho\colon\mfk g\to \mc L(\vct V)$ is the differential of $\rho$.
Sets of Lie algebra bases and discrete generators for various standard groups are given in~\cite{pmlr-v139-finzi21a}. For example, $\msf G=\msf S_n$ is generated by two elements, namely, the transposition $(1,2)$ and the $n$-cycle $(1,\ldots,n)$, reducing the number of group elements that must be considered from the na\"ive $n!$ to two. For $\msf G=\msf O_n$, a basis for the Lie algebra $\mfk g = \mathrm{Skew}(n)$ is given by $E_{i,j}=e_ie_j^\top - e_je_i^\top$ for $i<j$, and only one discrete generator, e.g., $h_1=\mathrm{diag}(-1,1,\ldots,1)$, is needed, for a total of $\binom{n}{2} + 1$ elements.
    
As equivariant linear maps are precisely the invariants in the space of linear maps, Theorem~\ref{thm:practical_invariance} allows us to obtain a basis for equivariant maps between fixed vector spaces. Explicitly, if $\rho_V\colon \msf G_{n_0}\to \mathrm{GL}(\vct V_{n_0})$ and $\rho_U\colon \msf G_{n_0}\to \mathrm{GL}(\vct U_{n_0})$ are the group homomorphisms defining the actions of $\msf G_{n_0}$ on $\vct V_{n_0},\vct U_{n_0}$, then $A_{n_0}\in\mc L(\vct V_{n_0},\vct U_{n_0})^{\msf G_{n_0}}$ if and only if
\begin{equation}\label{eq:ker_for_invars}
    \D\rho_U(H_i)A_{n_0} - A_{n_0}\D\rho_V(H_i) = 0,\quad \rho_U(h_j)A_{n_0} - A_{n_0}\rho_V(h_j) = 0,\quad \textrm{for all } i,j.
\end{equation}
The equations~\eqref{eq:ker_for_invars} express the space $\mc L(\vct V_{n_0},\vct U_{n_0})^{\msf G_{n_0}}$ as the kernel of a matrix, which is often very large and sparse. \new{Forming this matrix only requires specifying the matrices representing $\{H_i\}$ and $\{h_j\}$ for $\msf G_{n_0}$. Computing a basis for the kernel of this matrix can then by done} using its LU decomposition as in~\cite{computing_nullSpace,spspaces}, or using the algorithm of~\cite[\S5]{pmlr-v139-finzi21a}.
A basis for the space $\mc L(\vct W_{n_0},\vct U_{n_0})^{\msf G_{n_0}}$ is obtained analogously.     

\noindent\paragraph{Step~\ref{step:extendable_basis}: Computing a basis for extendable equivariant linear maps.}
We know from Example~\ref{ex:bad_cube} that extending arbitrary invariants $A_{n_0},B_{n_0},u_{n_0}$ to freely-described elements does not yield a compatible sequence of convex sets.
However, Theorem~\ref{thm:extending_compatible_descriptions} identifies subspaces of invariant linear maps $A_{n_0},B_{n_0}$ whose extensions do yield a compatible sequence of sets, provided we fix a freely-described element $\{u_n\}$ satisfying $u_{n+1}-u_n\in\cvx K_n$. 
Specifically, we need to find a basis for equivariant linear maps satisfying $A_{n_0}(\vct V_j)\subseteq \vct U_j$ for $j\leq d_V$ where $d_V$ is the generation degree of $\mscr V$. This is again a linear condition on $A_{n_0}$; defining $\varphi_{n_0,j} = \varphi_{n_0-1}\cdots\varphi_j\colon \vct V_j\hookrightarrow \vct V_{n_0}$ if $j<n_0$ and $\varphi_{n_0,n_0}=\id_{\vct V_{n_0}}$, and similarly for $\psi_{n_0,j}$, we have
\begin{equation}\label{eq:kernel_matrix}
    A_{n_0}(\vct V_j)\subseteq \vct U_j \iff (I-\mc P_{\vct U_j})A_{n_0}|_{\vct V_j} = 0 \iff (I-\psi_{n_0,j}\psi_{n_0,j}^*)A_{n_0}\varphi_{n_0,j} = 0.
\end{equation}
The subspace of $\mc L(\vct V_{n_0},\vct U_{n_0})^{\msf G_{n_0}}$ satisfying the hypotheses of Theorem~\ref{thm:extending_compatible_descriptions} is thus again the kernel of a matrix obtained by combining~\eqref{eq:ker_for_invars} and~\eqref{eq:kernel_matrix}.
To also impose $A_{n_0}(\vct V_j^{\perp})\subseteq \vct U_j^{\perp}$ for $j\leq d_U$ where $d_U$ is the generation degree of $\mscr U$, so that $A_{n_0}^*$ extends to a morphism, note that
\begin{equation}\label{eq:kernel_matrix_extra}
    A_{n_0}(\vct V_j^{\perp})\subseteq \vct U_j^{\perp} \iff \mc P_{\vct U_j}A_{n_0}|_{V_i^{\perp}}=0\iff \psi_{n_0,i}^*A_{n_0}(I-\varphi_{n_0,i}\varphi_{n_0,i}^*) = 0,
\end{equation}
hence the corresponding subspace of $\mc L(\vct V_{n_0},\vct U_{n_0})^{\msf G_{n_0}}$ is the kernel of the matrix obtained by combining~\eqref{eq:ker_for_invars}-\eqref{eq:kernel_matrix_extra}.
\new{Forming this matrix requires specifying the matrices representing $\{H_i\}$ and $\{h_j\}$ for $\msf G_{n_0}$, as well as the matrices representing the embeddings $\varphi_{n_0,i}$ and $\psi_{n_0,i}$.}
The subspace of $\mc L(\vct W_{n_0},\vct U_{n_0})^{\msf G_{n_0}}$ satisfying the hypotheses of Theorem~\ref{thm:extending_compatible_descriptions} is again the kernel of a matrix and its basis is computed similarly.

\paragraph{Step~\ref{step:extend}: Extending bases to higher dimensions.}

Given bases of equivariant $A_{n_0}^{(i)},B_{n_0}^{(j)}$, we wish to extend them to freely-described elements. 
We do so computationally by applying a linear map for $n<n_0$ and solving a linear system for each $n>n_0$ to which we wish to extend.
For each $n<n_0$, we set $A_{n}^{(i)} = \psi_{n_0,n}^*A_{n_0}^{(i)}\varphi_{n_0,n}$ and similarly for $B_n^{(j)}$. 
For each $n>n_0$, we set $A_n^{(i)}$ to be the unique solution to the linear system~\eqref{eq:ker_for_invars} (with $n_0$ replaced by $n$) and $\psi_{n,n_0}^*A_n^{(i)}\varphi_{n,n_0} = A_{n_0}^{(i)}$.
\new{Forming this linear system requires sepcifying the matrices representing $\{H_i\}$ and $\{h_j\}$ for $\msf G_n$, as well as matrices representing $\varphi_{n,n_0}$ and $\psi_{n,n_0}$.}
The above linear system is typically large and sparse, and we solve it using LSQR~\cite{paige1982lsqr}. The extension of $B_{n_0}^{(j)}$ is handled similarly, except that $n_0$ needs to exceed the presentation degrees of both $\mscr W$ and $\mscr U$ to guarantee that both $B_{n_0}^{(j)}$ and $(B_{n_0}^{(j)})^*$ extend to morphisms.

\begin{example}[Dimension counts]\label{ex:dim_counts}
We use the above algorithm to obtain dimension counts for parametric families of \new{dimension-}free descriptions. See the functions \href{https://github.com/eitangl/anyDimCvxSets/blob/main/compute_dims_a.m}{\texttt{compute\_dims\_a}}, \href{https://github.com/eitangl/anyDimCvxSets/blob/main/compute_dims_b.m}{\texttt{compute\_dims\_b}}, and \href{https://github.com/eitangl/anyDimCvxSets/blob/main/compute_dims_c.m}{\texttt{compute\_dims\_c}} on GitHub for the code computing these dimensions using Algorithm~\ref{algo:free_descr}.
\begin{enumerate}[(a)]
    \item Let $\mscr V = \{\RR^n\}$ with $\msf G_n=\msf S_n$ as in Example~\ref{ex:basic_consist_seq}, and let $\mscr W = \mscr U = \mathrm{Sym}^2(\mathrm{Sym}^{\leq 2}\mscr V)$. 
    Then $\mscr V,\mscr U,\mscr V\otimes\mscr U,\mscr W\otimes\mscr U$ are all $\mscr V$-modules and are presented in degrees $1,4,5,8$, respectively, by Theorem~\ref{thm:calc_for_pres_degs_FIW}. 
    
    The dimensions of invariants parametrizing \new{dimension-}free descriptions are $\dim\mc L(\vct V_n,\vct U_n)^{\msf G_n}=39$, $\dim\mc L(\vct W_n,\vct U_n)^{\msf G_n}\allowbreak =1068$, and $\dim \vct U_n^{\msf G_n}=17$ for $n\geq 8$. The dimensions of linear maps $\{A_n\}$ and $\{B_n\}$ satisfying Proposition~\ref{prop:compatible_descriptions}(a), yielding intersection-compatible convex sets, are
    \begin{equation*}\begin{aligned}
        &\dim \Big\{\{A_n\colon \vct V_n\to \vct U_n\} \textrm{ morphism}\Big\} = 6,\\ &\dim \Big\{\{B_n\colon \vct W_n\to \vct U_n\}: \textrm{both } \{B_n\} \textrm{ and } \{B_n^*\} \textrm{ morphisms}\Big\} = 104.
    \end{aligned}\end{equation*}
    If we further require $\{A_n^*\}$ to be a morphism to obtain projection compatibility, the dimension of $\{A_n\}$ decreases to 5.
    
    \item Let $\mscr V = \{\mbb S^n\}$ with $\msf G_n=\msf S_n$ used to obtain graph parameters in Example~\ref{ex:graph_params}, and let $\mscr W,\mscr U$ be as in (a). 
    Then $\mscr V,\mscr U,\mscr V\otimes\mscr U,\mscr W\otimes\mscr U$ are all $\mscr V$-modules and are presented in degrees $2,4,6,8$, respectively. 
    
    The dimension of invariant $\{A_n\}$ in this case is $\dim \mc L(\vct V_n,\vct U_n)^{\msf G_n}=93$, and the dimensions of $\dim \mc L(\vct W_n,\vct U_n)^{\msf G_n}$ and $\dim \vct W_n^{\msf G_n}$ are as in (a), giving the number of parameters describing freely-described convex sets with these description spaces. The dimensions of linear maps $\{A_n\}$ and $\{B_n\}$ satisfying Proposition~\ref{prop:compatible_descriptions}(a), and hence describing intersection-compatible sets, are 19 and 104, respectively, as given in~\eqref{eq:dim_graph_morphisms}. If we further require $\{A_n^*\}$ to be a morphism to get projection compatibility, the dimension of $\{A_n\}$ decreases to 12.

    \item Let $\mscr V=\{\RR^n\}$ with $\msf G_n=\msf B_n$ as in Example~\ref{ex:basic_consist_seq}, used below to learn regularizers defined for vectors of any length, let $\mscr V' = \{\RR^{2n+1}\}=\mscr W^{(2)}\oplus\RR$ as in Example~\ref{ex:bad_cube}, and let $\mscr W = \mscr U = \mathrm{Sym}^2(\mathrm{Sym}^{\leq 1}\mscr V')$. 
    Then $\mscr V,\mscr U,\mscr V\otimes\mscr U,\mscr W\otimes\mscr U$ are all $\mscr V$-modules and are presented in degrees $1,2,3,4$, respectively.  
    
    The dimensions of invariants parametrizing freely-described convex sets in this case are
    \begin{equation*}
        \dim \mc L(\vct V_n,\vct U_n)^{\msf G_n} = 4,\quad \dim\mc L(\vct W_n,\vct U_n)^{\msf G_n} = 108,\quad \dim \vct U_n^{\msf G_n} = 8,\quad \textrm{for all } n\geq 4.
    \end{equation*}
    The dimensions of linear maps $\{A_n\}$ and $\{B_n\}$ satisfying Proposition~\ref{prop:compatible_descriptions}(a) and parametrizing intersection-compatible sets are
    \begin{equation}\begin{aligned}
        &\dim \Big\{\{A_n\colon \vct V_n\to \vct U_n\} \textrm{ morphism}\Big\} = 3,\\ &\dim \Big\{\{B_n\colon \vct W_n\to \vct U_n\}: \textrm{both } \{B_n\} \textrm{ and } \{B_n^*\} \textrm{ morphisms}\Big\} = 37.
    \end{aligned}\end{equation}
    If we further require $\{A_n^*\}$ to be a morphism, the dimension does not decrease in this case, so all of these intersection-compatible sets are also projection-compatible. 
\end{enumerate}

\end{example}

\subsection{Fitting Freely-Described Convex Functions to Data}\label{sec:cvx_regr}


We now present an algorithm for the \new{dimension-}free convex regression problem from Section~\ref{sec:intro_cvx_regr}. Recall that in this problem we are given data $\{(x_i,\phi_i)\in \vct V_{n_i}\oplus\RR\}$ in finitely-many dimensions $n_i$ corresponding to a sequence of vector spaces $\{\vct V_n\}$, and our objective is to identify a sequence of convex functions $\{f_n : \vct V_n \rightarrow \RR\}$ such that $f_{n_i}(x_i)\approx \phi_i$ in the dimensions in which data is available. We tackle this problem by endowing $\{\vct V_n\}$ with the structure of a consistent sequence of $\{\msf G_n\}$-representations, and choosing description spaces $\{\vct W_n\},\{\vct U_n\}$ that yield a finitely-parametrized family of freely-described and compatible convex functions which we fit to the given data. We explain how to perform this fitting below in two key steps.  

The first step is to define a finitely-parametrized family of freely-described convex functions. Let $\{\{A_n^{(i)}\}\}_{i=1}^{d_A}$ and $\{\{B_n^{(j)}\}\}_{j=1}^{d_B}$ be the bases for freely-described maps computed by Algorithm~\ref{algo:free_descr} where we choose $n_0\geq\max_in_i$ so that we do not have to extend the basis elements computed there to access the data dimensions. Further, we select intersection and projection compatible cones $\mscr K=\{\cvx K_n\subseteq \vct U_n\}$, and freely-described $\{u_n\in \vct U_n^{\msf G_n}\}$ satisfying $u_{n+1}-u_n\in \cvx K_n$ and $u_n\in\mathrm{int}(\cvx K_n)$ for all $n$. We consider families of nonnegative and positively-homogeneous convex functions $f_n$ parametrized by $\alpha\in\RR^{d_A},\ \beta\in\RR^{d_B}$, and $\lambda\in\RR$ of the form
\begin{align}
    f_n(x;\alpha,\beta,\lambda) &= \inf_{\substack{t\geq 0\\y\in \vct W_n}} t + \lambda\|y\|\quad \textrm{ s.t. } \sum_{i=1}^{d_A}\alpha_i A_n^{(i)}x + \sum_{j=1}^{d_B}\beta_jB_n^{(j)}y + tu_n \in \cvx K_n \label{eq:fn_primal_form}\tag{P}\\ 
    &= \sup_{z\in \vct U_n} -\sum_{i=1}^{d_A}\alpha_i\langle z, A_n^{(i)}x\rangle \quad \textrm{ s.t. } \left\|\sum_{j=1}^{d_B}\beta_j(B_n^{(j)})^*z\right\| \leq \lambda,\ \langle z, u_n\rangle\leq 1,\ z\in \cvx K_n^*, \tag{D}\label{eq:fn_dual_form}
\end{align}
where the primal and dual programs above are equal since Slater's condition is satisfied by our choice of $u_n$.  Here $\|\cdot\|$ is the norm on $\vct W_{n}$ induced by its inner-product and the parameter $\lambda$ is chosen to be positive; the purpose of this term in \eqref{eq:fn_primal_form} is to prevent numerical issues arising in our subsequent regression procedure (see \eqref{eq:cvx_regr_cost}).  Note that this is indeed a finitely-parametrized (by $d_A+d_B+1$ parameters) infinite sequence of functions. Also note that if we require compatibility in Algorithm~\ref{algo:free_descr}, then the sequence $\{f_n\}$ is intersection (and if desired, projection) compatible for any value of the parameters, since it is the sequence of gauge functions of a correspondingly compatible sequence of sets (see Section~\ref{sec:intro_notation}).


The second step concerns optimizing over the parameters $\alpha,\beta,\lambda$ to fit $\{f_n\}$ to the available data. To this end, we consider the following optimization problem, with a user-specified $\lambda_{\min} > 0$:
\begin{alignat}{2}
    &\min_{\substack{\varepsilon\in\RR^D_+\\\alpha\in\RR^{d_A},\beta\in\RR^{d_B},\lambda\geq\lambda_{\min}\\ \{(t_i,y_i)\},\{z_i\}}} &&\|\varepsilon\|_{\ell_2} \qquad \mathrm{s.t.} \label{eq:cvx_regr_cost}\tag{Regress}\\
    &\ && (y_i,t_i) \textrm{ feasible for~\eqref{eq:fn_primal_form} with $n=n_i$ and cost } \leq \phi_i+\varepsilon_i, \label{eq:cvx_regr_primal_constr}\tag{PC}\\
    &\ && z_i \textrm{ feasible for~\eqref{eq:fn_dual_form} with $n=n_i$ and cost } \geq \phi_i - \varepsilon_i, \label{eq:cvx_regr_dual_constr}\tag{DC}.
\end{alignat}
The constraints~\eqref{eq:cvx_regr_primal_constr} and~\eqref{eq:cvx_regr_dual_constr} are required to hold for all $i\in[D]$ and they ensure that $\phi_i-\varepsilon_i\leq f_{n_i}(x_i;\alpha,\beta,\lambda)\leq \phi_i+\varepsilon_i$, so that minimizing $\|\varepsilon\|_{\ell_2}$ yields a good fit to the data.  We emphasize that this problem is finite-dimensional even though the it yields an infinite sequence of convex functions $\{f_n\}$.

As~\eqref{eq:cvx_regr_cost} involves bilinear constraints, we solve the problem via \emph{alternating minimization}, where we alternate between fixing $\alpha,\beta,\lambda$ and $\{(t_i,y_i)\},\{z_i\}$ while optimizing over the rest of the variables.
Note that Slater's condition holds in~\eqref{eq:cvx_regr_cost} for both steps of alternating minimization when $u_n\in\mathrm{int}(\cvx K_n)$.

\begin{remark}\label{rmk:identif}
The quantity $n_0$ in Algorithm~\ref{algo:free_descr} is governed by the choice of description spaces and leads to a tradeoff between the richness of the parametric family and the dimensions of the available data.  The data we are given only contains information about $\{f_n\}_{n\leq \max_in_i}$ which, in turn, only depends on $\{\{A_n^{(i)}\}_{n\leq\max_in_i}\}_{i=1}^{d_A}$ and $\{\{B_n^{(j)}\}_{n\leq\max_in_i}\}_{j=1}^{d_B}$. If $n_0$ in Algorithm~\ref{algo:free_descr} is strictly larger than the maximum dimension $\max_in_i$ in which data is available, then the number of distinct basis elements for $n\leq\max_in_i$ might be strictly smaller than $d_A$ and $d_B$, respectively, in which case our parameters $\alpha,\beta$ are not identifiable from such low-dimensional data.  Imposing compatibility decreases the bound on $n_0$, thus facilitating \new{dimension-}free convex regression from lower-dimensional data.  More broadly, even when $\max_in_i \geq n_0$, it is of interest to investigate the landscape of~\eqref{eq:cvx_regr_cost}.
\end{remark}

\subsection{Numerical Results}\label{sec:numerical_results}

In all experiments below, we use $\mscr U=\mathrm{Sym}^2(\mathrm{Sym}^{\leq k}\mscr V')$ for some $\mscr V$-module $\mscr V'$ and some $k$, with the corresponding PSD cones $\mscr K = \mathrm{Sym}^2_+(\mathrm{Sym}^{\leq k}\mscr V')$. We choose $\{u_n\}$ to be the sequence of identity matrices, which satisfy $u_n\in\mathrm{int}(\cvx K_n)$ and $u_{n+1}-u_n\in \cvx K_n$ as required.  
We apply our algorithm to learn semidefinite approximations of two non-SDP-representale functions, comparing the results obtained with and without imposing compatibility in Algorithm~\ref{algo:free_descr}.

The first function we approximate is the $\ell_{\pi}$ norm $\|x\|_{\pi}=\left(\sum_i|x_i|^{\pi}\right)^{1/\pi}$, which is not SDP-representable because $\pi$ is irrational. We view the $\ell_{\pi}$ norm as defined on the sequence $\mscr V = \{\RR^n\}$ with $\msf G_n=\msf B_n$ from Example~\ref{ex:basic_consist_seq}. It satisfies both intersection and projection compatibility. 
We choose description spaces $\mscr W=\mscr U = \{\mathrm{Sym}^2(\mathrm{Sym}^{\leq 1}\mbb R^{2n+1}) = \mbb S^{2n+2}\}$ as in Example~\ref{ex:dim_counts}(c), with the corresponding PSD cones $\{\cvx K_n=\mbb S^{2n+2}_+\}$. 
We used 50 data points in $\RR^n$ for $n_i\in\{1,2\}$. 
When we impose compatibility on our family of functions, we use $n_0=2=\max_in_i$ in Algorithm~\ref{algo:free_descr}.
If we do not impose compatibility and search over all freely-described (possibly incompatible) sequences, we take $n_0=4$ which is the presentation degree of $\mscr W\otimes\mscr U$, and which strictly exceeds $\max_in_i$.
The constraints in the fitting program~\eqref{eq:cvx_regr_cost} do not depend on some of the entries of $\alpha,\beta$ in this case, and we set such entries to zero.
This highlights the advantage of imposing compatibility mentioned in Remark~\ref{rmk:identif}--it allows us to uniquely identify a \new{dimension-}free description from lower-dimensional data.

The second sequence of functions we approximate is the nonnegative and positively-homogeneous variant of the quantum entropy given by~\eqref{eq:quant_ent_mod} given in Section~\ref{sec:intro_cvx_regr}, defined on the sequence $\{\mbb S^n\}$ with embeddings by zero-padding and the action of $\msf G_n=\msf O_n$ by conjugation.
Once again, the function~\eqref{eq:quant_ent_mod} cannot be evaluated using semidefinite programming, though a family of semidefinite approximations is analytically derived in~\cite{fawzi2019semidefinite}. Here we aim to learn a semidefinite approximation entirely from evaluation data.
To that end, we choose description spaces $\mscr W=\{\vct W_n=\mathrm{Sym}^2(\mathrm{Sym}^{\leq 1}\RR^n)=\mbb S^{n+1}\}$ and $\mscr U = \{\vct U_n=\mathrm{Sym}^2(\mathrm{Sym}^{\leq 2}\mbb R^n) = \mbb S^{\binom{n+2}{2}}\}$, with corresponding PSD cones $\{\cvx K_n=\mbb S^{\binom{n+2}{2}}_+\}$. 
Our data consists of 200 PSD matrices in $\mbb S^n$ for $n\leq n_0=4$.
Without a calculus for presentation degrees for $\msf G_n=\msf O_n$, our theory does not guarantee the existence of an $\msf O_n$-invariant extension of our learned description. Our theory does however guarantee a unique $\msf B_n$-invariant extension, and in practice we observe that the extension is, in fact, $\msf O_n$-invariant. 

To approximate the above functions, which we generically denote $\{f_n^{\mathrm{true}}\}$, we used~\eqref{eq:cvx_regr_cost} with 100 random initializations to fit the data in degree $n_0$. For the above two examples, not only is $f_{n_0}^{\mathrm{true}}$ positively-homogeneous and nonnegative, but also $f_{n_0}^{\mathrm{true}}(x)\neq 0$ for $x\neq 0$ in the domain. We therefore normalize the data $x_i$ by $x_i\mapsto x_i/f_{n_0}^{\mathrm{true}}(x_i)$, so that $f_{n_0}^{\mathrm{true}}(x_i)= 1$ for all $i$ and all points contribute equally to the objective of~\eqref{eq:cvx_regr_cost}.
We impose $\lambda\geq \lambda_{\min}=10^{-3}$ in~\eqref{eq:cvx_regr_cost}.
To evaluate the results, we extended our learned descriptions to $n=20$, sampled $10^3$ unit-norm points (also PSD for the quantum entropy example) and computed the average normalized errors $\frac{|f_n(x)-f_n^{\mathrm{true}}(x)|}{f_n^{\mathrm{true}}(x)}$ in each $n$ up to 20. 

The resulting errors are plotted in Figure~\ref{fig:regr_results}, shown and discussed in Section~\ref{sec:intro_cvx_regr}.
Since imposing compatibility conditions decreases the search space in~\eqref{eq:cvx_regr_cost}, we expect the optimal solution of~\eqref{eq:cvx_regr_cost} with compatibility conditions to exhibit larger errors in dimensions in which data is available compared to the optimal freely-described (but possibly incompatible) solution. That is not the case in Figure~\ref{fig:regr_results}(b), illustrating the nonconvexity of the fitting problem~\eqref{eq:cvx_regr_cost} and demonstrating another advantage of imposing compatibility---the resulting smaller parametric family allows our algorithm to better fit the data.

\section{Conclusions and Future Work}\label{sec:conclusions}
We developed a systematic framework to study convex sets that can be instantiated in different dimensions using representation stability, as well as a computational method to parametrize such sets and fit them to data.
We did so by formally defining \new{dimension-}free descriptions of convex sets and compatibility conditions relating sets in different dimensions. We then proved a number of structural results pertaining to \new{dimension-}free descriptions, namely, characterizing such descriptions certifying compatibility; giving conditions on fixed-dimensional descriptions ensuring their extendability to \new{dimension-}free ones; and studying infinite-dimensional limits of freely-described sequences of sets.
We further used representation stability to systematically derive constant-sized descriptions for sequences of invariant sections of PSD and relative entropy cones. 
Finally, we developed an algorithm to computationally parametrize and search over \new{dimension-}free descriptions to fit them to data.
Our work can be viewed as identifying and exploiting a new point of contact between representation stability and convex geometry through conic descriptions of convex sets. 

Our work suggests questions and directions for future research in several areas.
\begin{enumerate}[align=left, font=\textbf]
    \item[(Computational algebra)] Is there an algorithm to compute the generation and presentation degrees of a given consistent sequence? 

    \item[(Lie groups)] Can we extend our calculus for presentation degrees in Theorem~\ref{thm:calc_for_pres_degs_FIW} to Lie groups such as $\msf G_n=\msf O_n$? 
    
    

    
    \item[(Constructing descriptions)] Given a sequence of convex sets instantiable in any relevant dimension, can we systematically construct freely-described, and possibly compatible, approximations for the sequence? When are approximations derived from sums-of-squares hierarchies such as~\cite{equivariant_lifts} \new{dimension-}free and certify compatibility?

    \item[(Complexity)] Is there a systematic framework to study the smallest possible size of a \new{dimension-}free description for a given sequence of sets, extending the slack operator-based approach for fixed convex sets~\cite{gouveia2013lifts}?

    \item[(\new{Dimension-}free separation)] Under what conditions can a point outside a compatible sequence of convex sets be separated by a freely-described sequence, generalizing the Effros-Winkler theorem~\cite{EFFROS1997117}?

    \item[(Statistical inference)] How much data do we need to learn a given sequence of sets or functions, and in what dimensions should this data lie? 
\end{enumerate}

\section*{Acknowledgements}
We thank Mateo Díaz for helpful discussions and suggestions for improving our implementation, Jan Draisma for suggesting relevant references, and Jorge Garza Vargas for helpful discussions regarding approximately finite-dimensional algebras. \new{We are grateful to reviewer X for carefully reading our paper and offering numerous detailed suggestions for its improvement.} The authors were supported in part by Air Force Office of Scientific Research grants FA9550-22-1-0225 and FA9550-23-1-0070, and by National Science Foundation grant DMS 2113724.
\bibliographystyle{unsrtnat}
\bibliography{free_cvx_refs}

\appendix
\section{Representations of Categories}\label{sec:cat_reps}
As Remark~\ref{rmk:V-mods_vs_subgps} shows, the set of embeddings from low to high dimensions in a consistent sequence $\mscr V$ of $\{\msf G_n\}$-representations, determined by $\{\msf G_n\}$ and the centralizing subgroups $\{\msf H_{n,d}\}$ from Definition~\ref{def:stab_subgps}, play a central role in our framework. 
These sets of embeddings are conveniently encoded in a category, whose representations are precisely the $\mscr V$-modules of Definition~\ref{def:V_mods}. 
Morphisms between such representations in the categorical sense coincide with morphisms of sequences. 
This categorical approach to representation stability was introduced in~\cite{FImods} for the case $\msf G_n=\msf S_n$ and the $\msf H_{n,d}$ from Example~\ref{ex:basic_consist_seq_degs}, and has since been extended to other groups in~\cite{WILSON2014269,GADISH2017450,sam2017grobner}.

\begin{definition}\label{def:rep_of_cats}
    A (real) \emph{representation} of a category $\mc C$, also called a \emph{$\mc C$-module}, is a functor $\mc C\to\mathrm{Inn}_{\RR}$ from $\mc C$ to the category \new{whose objects are real inner product spaces and whose morphisms are isometries.\footnote{Representations of categories are typically defined as functors to the category of all vector spaces over a given field. The results in this section continue to hold for functors to the category of all real vector spaces if we allow consistent sequences consisting of non-orthogonal representations and non-isometric linear maps between them.}}
\end{definition}
In other words, a $\mc C$-module is an assignment of a vector space $\vct V_n$ to each object $n\in\mc C$ and a linear map $\varphi_{N,n}\colon \vct V_n\to \vct V_N$ to each morphism in $\mathrm{Hom}_{\mc C}(n,N)$ such that compositions  are respected. 
Each $\vct V_n$ is then a representation of the group $\msf G_n=\mathrm{End}_{\mc C}(n)^{\times}$ of the automorphisms of $n$ in $\mc C$.
Every consistent sequence is a representation of a suitable category.
\begin{definition}\label{def:cat_from_seq}
    Given a consistent sequence $\mscr V=\{(\vct V_n,\varphi_n)\}$ of $\{\msf G_n\}$-representations, define a category $\mc C_{\mscr V}$ whose set of objects is $\NN$ and whose morphisms are $\mathrm{Hom}_{\mc C_V}(n,N) = \{g\varphi_{N-1}\cdots\varphi_n:g\in \msf G_N\}$ for $n\leq N$ and zero otherwise. 
    Note that $\mathrm{Hom}_{\mc C_V}(n,N)=\msf G_N/\msf H_{N,n}$ where $\msf H_{N,n}\subseteq \msf G_N$ is the subgroup of elements acting trivially on $\vct V_n$.

    \new{If $n\leq N\leq M$, we compose morphisms by setting $$(g_M\varphi_{M-1}\cdots\varphi_N)\circ(g_N\varphi_{N-1}\cdots\varphi_n)=(g_Mg_N)\varphi_{M-1}\cdots \varphi_n,$$
    which is well-defined and satisfies the necessary axioms since $\varphi_N,\ldots,\varphi_{M-1}$ are $\msf G_N$-equivariant when acting on the vector spaces in $\mscr V$. }
\end{definition}
This definition clearly extends to consistent sequences indexed by posets (Remark~\ref{rmk:poset_gen_moved}).
If $\mscr U=\{(\vct U_n,\psi_n)\}$ is a $\mscr V$-module (Definition~\ref{def:V_mods}), then $\mscr U$ is a $\mc C_{\mscr V}$-module, since sending $n\in\mbb N$ to $\vct U_n$ and $g\varphi_{N-1}\cdots\varphi_n$ to the map $g\psi_{N-1}\cdots\psi_n$ for each $g\in \msf G_N$ is a well-defined functor $\mc C_{\mscr V}\to\mathrm{Vect}_{\mathbb{R}}$. 
Conversely, if $\mscr U$ is a $\mc C_{\mscr V}$-module then it is a $\mscr V$-module since $\msf H_{n,d}$ acts trivially on the image of $\psi_{N-1}\cdots\psi_d$ by definition of a functor. 
Furthermore, if $\mscr W,\mscr U$ are $\mscr C_{\mscr V}$-modules, then a morphism of functors $\mscr W\to\mscr U$ (also called a natural transformation) 
coincides with a morphism of sequences in Definition~\ref{def:morphisms_of_seqs}.
Applying the constructions in Section~\ref{sec:consist_seq_construct} to $\mc C$-modules yields other $\mc C$-modules.

\begin{example}\label{ex:cats}
Here are some examples of the categories resulting from Definition~\ref{def:cat_from_seq}.
    \begin{enumerate}[(a)]
        \item The category corresponding to Example~\ref{ex:basic_consist_seq} with $\msf G_n=\msf S_n$ is (the skeleton of) $\mc C = \mathrm{FI}$, the category whose objects are finite sets and whose morphisms are injections. 

        \item The category corresponding to Example~\ref{ex:basic_consist_seq} with $\msf G_n=\msf B_n$ (resp., $\msf D_n$) is $\mc C = \mathrm{FI}_{BC}$ (resp., $\mc C = \mathrm{FI}|_D$) defined in~\cite{WILSON2014269}, whose objects are the sets $[\pm n]\coloneqq \{\pm 1,\ldots,\pm n\}$ for $n\in\NN$ and whose morphisms are injections $f\colon [\pm n]\hookrightarrow [\pm N]$ satisfying $f(-i)=-f(i)$ (and reverse evenly-many signs if $\msf G_n=\msf D_n$). 

        
        \item The category corresponding to the graphon sequence in Section~\ref{sec:graphons} is the opposite category $\mc C = \mc P_2^{\mathrm{op}}$ of the category $\mc P_2$ with objects $[2^n]$ and morphisms which are $2^{N-n}$-to-one surjections $[2^N]\to[2^n]$, or equivalently, partitions of $[2^N]$ into $2^n$ equal parts. 
    \end{enumerate}
\end{example}
Following~\cite{WILSON2014269}, we say $\mc C=\mathrm{FI}|_{\mc W}$ if $\mc C=\mathrm{FI},\mathrm{FI}|_{BC}$ or $\mathrm{FI}|_D$.
(Algebraically) free $\mc C$-modules are defined exactly as in Definition~\ref{def:free_mods}, see~\cite[Def.~2.2.2]{FImods} and~\cite[Def.~1.8,3.1]{GADISH2017450} for example. 
The theory of~\cite{GADISH2017450} gives the following result for $\mc C=\mathrm{FI}|_{\mc W}$, which extends to categories of FI-type introduced in~\cite{GADISH2017450}.
\begin{theorem}[{\cite[Thm.~B(1)]{GADISH2017450}}]\label{thm:tensor_prods_free}
    Tensor products of free $\mathrm{FI}|_{\mc W}$-modules are free.
\end{theorem}

The following result illustrates two further properties of $\mathrm{FI}|_{\mc W}$-modules, the second of which is included in our Theorem~\ref{thm:calc_for_pres_degs_FIW} stated above.
\begin{theorem}[Noetherianity and tensor products]\label{thm:fin_gen_calc}
    Let $\mc C=\mathrm{FI}|_{\mc W}$.
    \begin{enumerate}[align=left, font=\emph]
        \item[(Noetherianity)] Any submodule of a finitely-generated $\mc C$-module is finitely-generated. 
        \item[(Tensor products)] If $\mscr V_1$ and $\mscr V_2$ are $\mc C$-modules generated in degrees $d_1$ and $d_2$, respectively, then $\mscr V_1\otimes \mscr V_2$ is generated in degree $d_1+d_2$.
    \end{enumerate}
\end{theorem}
\begin{proof}
    Noetherianity is shown for FI in~\cite[Thm.~1.13]{FImods} and for $\mathrm{FI}|_{BC},\mathrm{FI}|_D$ in~\cite[Thm.~4.21]{WILSON2014269}. The generation degree bound is shown in~\cite[Prop.~2.3.6]{FImods} for FI and in~\cite[Prop.~5.2]{WILSON2014269} for $\mathrm{FI}|_{BC},\mathrm{FI}|_D$.
\end{proof}
Noetherianity helps explain the ubiquity of representation stability, while the generation degree bound for tensor products allows us to bound the generation degrees of complicated sequences from degrees of simple ones. 
The two properties in Theorem~\ref{thm:fin_gen_calc} hold over more general categories than $\mathrm{FI}|_{\mc W}$, including for categories of FI-type and certain quasi-Gr\"obner categories introduced in~\cite{sam2017grobner}. 
We remark that these two types of categories only include representations of finite groups, and do not include the graphon category in Example~\ref{ex:cats}(c), whose properties would be interesting to study in future work. 
\begin{definition}[Property (TFG)]
    We say that a category $\mc C$ satisfies \emph{property (TFG)} if\new{:
    \begin{enumerate}[(a)]
        \item tensor products of free $\mc C$-modules are free;
        \item if $\mscr V_1$ and $\mscr V_2$ are $\mc C$-modules generated in degrees $d_1$ and $d_2$, respectively, then $\mscr V_1\otimes \mscr V_2$ is generated in degree $d_1+d_2$.
    \end{enumerate}}
\end{definition}
An example of a category not satisfying (TFG) is given in~\cite[Rmk.~7.4.3]{sam2017grobner}.
We can use property (TFG) to obtain a calculus for presentation degrees from which Theorem~\ref{thm:calc_for_pres_degs_FIW} may be deduced.
\begin{proposition}\label{prop:pres_calc_for_tensors}
    Suppose $\mc C$ is a category satisfying (TFG). If $\mscr V,\mscr U$ are $\mc C$-modules which are generated in degrees $d_V,d_U$ and presented in degrees $k_V,k_U$, respectively, then $\mscr V\otimes\mscr U$ is presented in degree $\max\{d_V+k_U, d_U+k_V\}$.
\end{proposition}
\begin{proof}
    Suppose $\mscr F_V,\mscr F_U$ are free $\mc C$-modules generated in degrees $d_V,d_U$, respectively, and $\mscr F_V\to \mscr V$ and $\mscr F_U\to\mscr U$ are surjective morphisms whose kernels $\mscr K_V$ and $\mscr K_U$ are generated in degrees $r_V,r_U$, respectively. Then $\mscr F_V\otimes\mscr F_U$ is a free $\mc C$-module generated in degree $d_V+d_U$ by (TFG), and the morphism $\mscr F_V\otimes\mscr F_U\to \mscr V\otimes\mscr U$ is surjective with kernel $\mscr K_V\otimes\mscr F_U + \mscr F_V\otimes\mscr K_U$. Since $\mscr K_V\otimes\mscr F_U$ is generated in degree $r_V+d_U$ and similarly for $\mscr F_U\otimes\mscr K_V$, their sum is generated in degree $\max\{r_V+d_U, d_V + r_U\}$. 
\end{proof}
Our next goal is to understand presentation degrees for images, and in particular, for Schur functors. We begin with a number of elementary lemmas.
\begin{lemma}
    Let $\mscr V$ and $\mscr U$ be $\mc C$-modules. If $\mscr V,\mscr U$ are generated in degrees $d_V,d_U$ and presented in degrees $k_V,k_U$, respectively, then $\mscr V\oplus\mscr U$ is generated in degree $\max\{d_V,d_U\}$ and presented in degree $\max\{k_V,k_U\}$.
\end{lemma}
\begin{proof}
    The claim about the generation degree is immediate from its definition.
    Suppose that $\mscr F_{V}\to\mscr V$ and $\mscr F_U\to \mscr U$ are surjective morphisms with $\mscr F_V,\mscr F_U$ being free $\mc C$-modules generated in degrees $d_V,d_U$ with kernels $\mscr K_V,\mscr K_U$ generated in degrees $k_V,k_U$, respectively. Then $\mscr F_V\oplus\mscr F_U$ is free (by definition) and generated in degree $\max\{d_V,d_U\}$, and surjects onto $\mscr V\oplus\mscr U$ with kernel $\mscr K_V\oplus\mscr K_U$ which is generated in degree $\max\{k_V,k_U\}$.
\end{proof}
\begin{lemma}\label{lem:gen_deg_of_preim}
Let $\mscr V=\{\vct V_n\}$ and $\mscr U=\{\vct U_n\}$ be two $\mc C$-modules, let $\mscr A=\{A_n\}\colon\mscr V\to \mscr U$ be a surjective morphism, and let $\mscr W=\{\vct W_n\subseteq \vct U_n\}$ be a $\mc C$-submodule of $\mscr U$. If $\mathrm{ker}\mscr A$ is generated in degree $d$ and $\mscr W$ is generated in degree $d_W$, then $\mscr A^{-1}(\mscr W)=\{A_n^{-1}(\vct W_n)\}$ is a $\mc C$-module generated in degree $\max\{d,d_W\}$.
\end{lemma}
\begin{proof}
    \new{Suppose $\{y_i\}\subseteq\vct W_{d_W}$ are generators for $\mscr W$. Since $A_{d_W}$ is surjective, we can find $\{x_i\}\subseteq \vct V_{d_W}$ such that $A_{d_W}x_i=y_i$ for all $i$. Also let $\{z_j\}\subseteq \ker A_d$ be generators for $\ker\mscr A$. We claim that $\{x_i\}\cup\{z_j\}\subseteq A_{\max\{d,d_W\}}^{-1}(\vct W_{\max\{d,d_W\}})$ generate $\mscr A^{-1}(\mscr W)$. Indeed, if $n\geq \max\{d,d_W\}$ and $x\in A_n^{-1}(\vct W_n)$, then we can write $A_nx=\sum_i\alpha_ig_iy_i$ for some $\alpha_i\in\RR$ and $g_i\in\msf G_n$, in which case $x-\sum_i\alpha_ig_ix_i\in\ker A_n$ and thus, we have $x=\sum_i\alpha_ig_ix_i + \sum_j\beta_jg_jz_j$ for some $\beta_j\in\RR$ and $g_j\in\msf G_n$. }
\end{proof}
\begin{lemma}\label{lem:subseq_gen_bound}
    Suppose $\mscr V = \{\vct V_n\}$ and $\{\vct U_n\}$ are two $\mc C$-modules, and $\mscr A=\{A_n\}\colon\mscr V\to\mscr U$ is a morphism. If $\mscr V$ is generated in degree $d$, then $\mathrm{im}\mscr A$ is generated in degree $d$. If, moreover, $\mscr A^*=\{A_n^*\}$ is a morphism, then $\ker\mscr A$ is also generated in degree $d$.
\end{lemma}
\begin{proof}
    The first claim follows from $A_n(\vct V_n)=A_n(\RR[\msf G_n]\vct V_d)=\RR[\msf G_n]A_n(\vct V_d)=\RR[\msf G_n]A_d(\vct V_d)$, where we used the equivariance of $A_n$ and the fact that $A_n|_{\vct V_d}=A_d$. For the second claim, note that if $\mscr A^*$ is a morphism, then $\{\mathrm{im}\, A_n^*=(\ker A_n)^{\perp}\}$ is a $\mc C$-submodule of $\mscr V$. Therefore, $\{\mc P_{\ker A_n}\}\colon\mscr V\to\mscr V$ is a morphism, and its image is precisely $\ker\mscr A$.
\end{proof}
\begin{proposition}\label{prop:pres_deg_of_im_morph}
    Suppose $\mscr V=\{\vct V_n\},\mscr U=\{\vct U_n\}$ are two $\mc C$-modules and both $\mscr A=\{A_n\colon \vct V_n\to \vct U_n\}$ and $\{A_n^*\colon \vct U_n\to \vct V_n\}$ are morphisms. If $\mscr V$ is generated in degree $d$ and presented in degree $k$, then $\mathrm{im}\mscr A=\{A_n(\vct V_n)\}$ is generated in degree $d$ and presented in degree $k$.
\end{proposition}
\begin{proof}
    Let $\mscr F = \{\vct F_n\}$ be a free $\mc C$-module generated in degree $d$ and let $\mscr B=\{B_n\}\colon\mscr F\to \mscr V$ be a surjective morphism whose kernel $\mscr K=\{\vct K_n\}$ is generated in degree $k$. The composition $\mscr F\xrightarrow{\mscr B} \mscr V \xrightarrow{\mscr A} \mathrm{im}\mscr A$ is a surjective morphism from the free $\mc C$-module $\mscr F$ whose kernel is $\mscr B^{-1}(\mathrm{ker}\mscr A)$ and is generated in degree $\max\{d,k\}=k$ by Lemmas~\ref{lem:gen_deg_of_preim} and~\ref{lem:subseq_gen_bound}. 
\end{proof}
\begin{corollary}
    Suppose $\mscr C$ satisfies property (TFG). If $\mscr V$ is a $\mc C$-module generated in degree $d$ and presented in degree $k$, and $\lambda$ is a partition, then $\mbb S^{\lambda}\mscr V$ is generated in degree $d|\lambda|$ and presented in degree $k + d(|\lambda|-1)$. 
\end{corollary}
Schur functors generalize symmetric and alternating algebras, see~\cite[\S6.1]{fulton2013representation}. Their generation degree for $\mc C=\mathrm{FI}$ was bounded using a similar approach in~\cite[Prop.~3.4.3]{FImods}.
\begin{proof}
    By Proposition~\ref{prop:pres_calc_for_tensors}, the $\mc C$-module $\mscr V^{\otimes|\lambda|}$ is generated in degree $d|\lambda|$ and presented in degree $d(|\lambda|-1) + k$.
    Let $\msf S_{|\lambda|}$ act on each $\vct V_n^{\otimes|\lambda|}$ by permuting its factors $\sigma\cdot(v_1\otimes\cdots\otimes v_{|\lambda|})=v_{\sigma(1)}\otimes\cdots\otimes v_{\sigma(|\lambda|)}$, which is an orthogonal action commuting with the embeddings $\vct V_n^{\otimes|\lambda|}\subseteq \vct V_{n+1}^{\otimes|\lambda|}$. 
    In this way, any element $c_{\lambda}\in\RR[\msf S_{|\lambda|}]$ defines a morphism $c_{\lambda}\colon \mscr V^{\otimes|\lambda|}\to\mscr V^{\otimes|\lambda|}$ such that $c_{\lambda}^*\in\RR[\msf S_{|\lambda|}]$ is also a morphism. 
    If $c_{\lambda}\in \RR[\msf S_{|\lambda|}]$ is the Young symmetrizer corresponding to partition $\lambda$, then $\mathrm{im}\, c_{\lambda}=\mbb S^{\lambda}\mscr V$. The result follows from Proposition~\ref{prop:pres_deg_of_im_morph}.
\end{proof}
We conclude this appendix by summarizing our calculus for generation and presentation degrees. Instantiating the following theorem with $\mc C=\mathrm{FI}|_{\mc W}$ yields Theorem~\ref{thm:calc_for_pres_degs_FIW}.
\begin{theorem}[Calculus for generation and presentation degrees]\label{thm:calculus_for_degrees}
    Let $\mscr V=\{\vct V_n\},\mscr U=\{\vct U_n\}$ be $\mc C$-modules generated in degrees $d_V,d_U$ and presented in degrees $k_V,k_U$, respectively.
    \begin{enumerate}[align=left, font=\emph]
        \item[(Sums)] $\mscr V\oplus \mscr U$ is generated in degree $\max\{d_V,d_U\}$ and presented in degree $\max\{k_V,k_U\}$. 
        \item[(Images and kernels)] If $\mscr A\colon\mscr V\to\mscr U$ and $\mscr A^*$ are morphisms, then $\mathrm{im}\mscr A$ and $\ker\mscr A$ are generated in degree $d_V$ and presented in degree $k_V$.
    \end{enumerate}
    Suppose $\mc C$ satisfies (TFG). Then
    \begin{enumerate}[align=left, font=\emph]
        \item[(Tensors)] $\mscr V\otimes\mscr U$ is generated in degree $d_V+d_U$ and presented in degree $\max\{k_V+d_U, k_U + d_V\}$.
        \item[(Schur functors)] $\mbb S^{\lambda}\mscr V$ is generated in degree $d_V|\lambda|$ and presented in degree $d_V(|\lambda|-1)+k_V$ for any partition $\lambda$. 
    \end{enumerate}
\end{theorem}

\section{Deriving the Presentation Degree from Extendability}\label{apdx:pres_deg}
In this appendix, we show how the definitions of algebraically free consistent sequences and the presentation degree naturally arise when trying to extend a fixed equivariant linear map to a morphism of sequences, necessary for our algorithm in Section~\ref{sec:numerical_examples}.

Let $\mscr V=\{\vct V_n\},\ \mscr U=\{\vct U_n\}$ be consistent sequences of $\{\msf G_n\}$-representations, fix $n_0\in\NN$ and consider a linear map $A_{n_0}\in\mc L(\vct V_{n_0},\vct U_{n_0})^{\msf G_{n_0}}$. When can we extend $A_{n_0}$ to a morphism of sequences $\{A_n\}$? 
We seek conditions on $A_{n_0}$ which are easy to enforce computationally, so that they can be used to parametrize and search over compatible sequences of convex sets computationally.
The following proposition gives an equivalent characterization for the existence of such an extension.
\begin{proposition}\label{prop:extend_using_rels}
    Let $\mscr V=\{\vct V_n\},\mscr U=\{\vct U_n\}$ be consistent sequences such that $\mscr V$ is generated in degree $d$, and fix $A_{n_0}\in \mc L(\vct V_{n_0},\vct U_{n_0})^{\msf G_{n_0}}$ for $n_0\geq d$. 
    \begin{enumerate}[(a)]
        \item There exists $\{A_n\in\mc L(\vct V_n,\vct U_n)^{\msf G_n}\}_{n<n_0}$ satisfying $A_{n+1}|_{\vct V_n}=A_n$ for all $n < n_0$ if and only if $A_{n_0}(V_j)\subseteq U_j$ for $j\leq d$. 

        \item There exists $\{A_n\in\mc L(\vct V_n,\vct U_n)^{\msf G_n}\}_{n>n_0}$ satisfying $A_{n+1}|_{\vct V_n}=A_n$ for all $n\geq n_0$ if and only if the following implication holds
        \begin{equation}\label{eq:relation_condition}
            \sum\nolimits_ig_ix_i=0\implies \sum\nolimits_ig_iA_{n_0}x_i = 0,\qquad \textrm{for all } g_i\in \msf G_n,\ x_i\in \vct V_d,\ n\in\NN. 
        \end{equation}
    \end{enumerate}
    If an extension $\{A_n\}$ of $A_{n_0}$ exists, then it is unique.
\end{proposition}
\begin{proof}
\begin{enumerate}[(a)]
    \item If such $\{A_n\}_{n<n_0}$ exists, then it is uniquely given in terms of $A_{n_0}$ by $A_n = A_{n_0}|_{\vct V_n}$. Therefore, we have $A_{n_0}(V_j)=A_j(V_j)\subseteq U_j$ for all $j\leq d$. Conversely, suppose $A_{n_0}(V_j)\subseteq U_j$ for $j\leq d$. We claim that $A_{n_0}(\vct V_n)\subseteq \vct U_n$ for all $d\leq n\leq n_0$ as well. Indeed, because $\mscr V$ is generated in degree $d$, we have $A_{n_0}(\vct V_n)=A_{n_0}(\RR[\msf G_n]\vct V_d)=\RR[\msf G_n]A_{n_0}(\vct V_d)\subseteq \RR[\msf G_n]\vct U_d\subseteq \vct U_n$ for such $n\geq d$. Defining $A_n=A_{n_0}|_{\vct V_n}$ for each $n<n_0$ yields the desired extension to lower dimensions.

    \item If such $\{A_n\}_{n>n_0}$ exists, it is unique and is explicitly given in terms of $A_{n_0}$ as follows. 
    For any $n>n_0$ and $x\in \vct V_n$, we can write $x=\sum_ig_ix_i$ for some $g_i\in \msf G_n$ and $x_i\in \vct V_d\subseteq \vct V_{n_0}$ by definition of the generation degree. 
    Because $A_n\colon \vct V_n\to \vct U_n$ is $\msf G_n$-equivariant and satisfies $A_n|_{\vct V_{n_0}}=A_{n_0}$, we have
    \begin{equation}\label{eq:nec_cond_for_ext}
        A_nx = A_n\left(\sum\nolimits_ig_ix_i\right) = \sum\nolimits_ig_i(A_{n_0}x_i),
    \end{equation}
    which expresses $A_n$ in terms of $A_{n_0}$. The expression~\eqref{eq:nec_cond_for_ext} shows that~\eqref{eq:relation_condition} is satisfied. 
    Conversely, suppose that~\eqref{eq:relation_condition} is satisfied. For each $n>n_0$ define $A_n\colon \vct V_n\to \vct U_n$ as follows. For any $x\in \vct V_n$, write $x=\sum_ig_ix_i$ for some $g_i\in \msf G_n$ and $x_i\in \vct V_d$, which is possible because $\mscr V$ is generated in degree $d$, and set $A_nx$ to the right-hand side of~\eqref{eq:nec_cond_for_ext}. This is well-defined because if $x=\sum_ig_ix_i=\sum_jg_j'x_j'$ for $g_i,g_j'\in \msf G_n$ and $x_i,x_j'\in \vct V_d$ then $\sum_ig_iA_{n_0}x_i = \sum_jg_j'A_{n_0}x_j'$ by~\eqref{eq:relation_condition}. Moreover, $A_n$ is linear, $\msf G_n$-equivariant, and extends $A_{n_0}$ by construction, so $\{A_n\}_{n>n_0}$ is the desired extension of $A_{n_0}$. \qedhere
\end{enumerate}
\end{proof}
The conditions on $A_{n_0}$ in Proposition~\ref{prop:extend_using_rels}(a) are easy to impose computationally, and we do so in Section~\ref{sec:implementation}. In contrast, while condition~\eqref{eq:relation_condition} in Proposition~\ref{prop:extend_using_rels}(b) fully characterizes extendability of $A_{n_0}$ to higher dimensions, it is unclear how to impose it computationally. We therefore proceed to study it further. 
In algebraic terms, elements $x_i\in \vct V_d$ are called the \emph{generators} of $\mscr V$, and expressions of the form $\sum_ig_ix_i=0$ with $g_i\in \msf G_n$ are called \emph{relations} between those generators over the group $\msf G_n$.
Proposition~\ref{prop:extend_using_rels}(b) shows that $A_{n_0}$ extends to a morphism iff any relation satisfied by the generators of $\mscr V$ is also satisfied by their images under $A_{n_0}$. 
We therefore need to understand the relations among the generators in $\vct V_d$.

We study these relations in two stages. First, we identify two simple types of relations that are satisfied by the images of the generators under $A_{n_0}$ for appropriate $\mscr V,\mscr U$. We then define algebraically free consistent sequences whose generators satisfy only these two types of relations. Second, we express an arbitrary consistent sequence as the quotient of an algebraically free one. The kernel of this quotient morphism is a consistent sequence that encodes all additional relations. To capture the degree starting from which both the generators and the relations between them stabilize, we define the presentation degree of a consistent sequence as the maximum of the generation degree of the sequence itself and that of the above kernel, see Definition~\ref{def:relations}. The presentation degree plays a prominent role in our structural results in Section~\ref{sec:free_and_compatible}. 

We begin carrying out the first stage of the above program and identify two simple types of relations. The first source for relations between generators in $\vct V_d$ arises from \new{the action of $\msf G_d$ on $\vct V_d$ itself, without involving any higher dimensions}. Indeed, if $\sum_ig_ix_i=0$ for $g_i\in \msf G_d$ and $x_i\in \vct V_d$ \new{is such a relation then $\sum_ig_iA_{n_0}x_i=0$ so the image $A_{n_0}x_i$ satisfies the same relation. Moreover, the above relations in $\vct V_d$ give rise to relations in $\vct V_n$ for $n>d$ since if $\sum_ig_ix_i=0$ as above then $\sum_igg_ix_i=0$ in $\vct V_n$ for any $g\in \msf G_n$. Once again, these higher-dimensional relations are always satisfied by $A_{n_0}x_i$.}
A second source for such relations arises from subgroups of $\msf G_n$ acting trivially on $\vct V_d$, which are precisely the centralizing subgroups of Definition~\ref{def:stab_subgps}.
Centralizing subgroups yield a second source for relations, namely, the relations $(h-\mathrm{id})x=0$ for all $x\in \vct V_d$ and $h\in \msf H_{n,d}$.
Thus, if $A_{n_0}$ extends to a morphism then $A_{n_0}(\vct V_d)\subseteq \bigcap_{n\geq d}\vct U_n^{\msf H_{n,d}}$.
Rather than attempt to enforce these constraints computationally, we make a standard simplifying assumption from the representation stability literature. Specifically, we assume that $\vct U_d$ is fixed by $\msf H_{n,d}$ (or a subgroup of it, see below) for all $n \geq d$, in which case there is a simple sufficient condition for the above constraints that we can impose computationally. This is precisely the assumption that $\mscr U$ is a $\mscr V$-module as in Definition~\ref{def:V_mods}.
This terminology comes from a categorical approach to representation stability, see Appendix~\ref{sec:cat_reps}. 
If $\mscr U$ is a $\mscr V$-module, then imposing $A_{n_0}(\vct V_d)\subseteq \vct U_d$ is sufficient to guarantee $A_{n_0}(\vct V_d)\subseteq \bigcap_{n\geq d}\vct U_n^{\msf H_{n,d}}$ since $\vct U_d$ is contained in the right-hand side of this inclusion. Imposing this sufficient condition can done computationally, as we do in Section~\ref{sec:implementation}.
This concludes the first stage.

To go beyond the above two simple types of relations satisfied by the generators, we define algebraically free consistent sequences (Definition~\ref{def:free_mods}), whose generators do not satisfy any additional types of relations. 
We then write any consistent sequence as the image under a morphism of sequences of an algebraically free one. The kernel of this morphism precisely captures the relations satisfied by the generators. 
To that end, note that any finitely-generated consistent sequence is the image under a morphism of sequences of an algebraically free sequence. 
\begin{proposition}\label{prop:ind_surj_vs_fg}
    Let $\mscr V$ be a consistent sequence of $\{\msf G_n\}$-representations and let $\mscr U=\{\vct U_n\}$ be a $\mscr V$-module. Then $\mscr U$ is generated in degree $d$ if and only if there exists an algebraically free $\mscr V$-module $\mscr F$ generated in degree $d$ and a surjective morphism of sequences $\mscr F\to \mscr U$.
\end{proposition}
\begin{proof}
    If $\mscr F = \{\vct F_n\}$ is a $\mscr V$-module and $\{A_n\colon \vct F_n\to \vct U_n\}$ is a surjective morphism, then for any $n\geq d$ we have $\vct U_n = A_n(\vct F_n) = A_n(\RR[\msf G_n]\vct F_d) = \RR[\msf G_n]A_n(\vct F_d) = \RR[\msf G_n]A_d(\vct F_d) = \RR[\msf G_n]\vct U_d,$
    where we used the fact that $\mscr F$ is generated in degree $d$; the equivariance of $A_n$; the fact that $A_n|_{\vct F_d}=A_d$ since $\{A_n\}$ is a morphism; and the surjectivity of $A_d$. 
    This shows $\mscr U$ is generated in degree $d$.

    Conversely, if $\mscr U$ is generated in degree $d$, define the algebraically free $\mscr V$-module $\mscr F = \bigoplus_{i=1}^d\mathrm{Ind}_{\msf G_i}(U_i)$ and consider the morphism $\mscr F\to\mscr U$ defined by $g\otimes u\mapsto g\cdot u$ for each $g\in \msf G_n$, $u\in U_i$, and $i\in[d]$ (see Section~\ref{sec:intro_notation}). The image of this morphism in $\vct U_n$ is precisely $\sum_{i=1}^{\min\{d,n\}}\RR[\msf G_n]U_i$, hence it is surjective for all $n$. Finally, $\mathrm{Ind}_{\msf G_i}(U_i)$ is generated in degree $i$, so $\mscr F$ is generated in degree $d$. 
\end{proof}
The kernel of the morphism in Proposition~\ref{prop:ind_surj_vs_fg} precisely encodes all the additional relations beyond the two simple types above satisfied by the generators of $\mscr V$.
The generation degree of this kernel then captures the point at which relations stabilize. We therefore define the presentation degree (Definition~\ref{def:relations}) as the maximum of the generation degree of $\mscr V$ and that of this kernel, which captures stabilization of the generators as well as of the relations between them.

The presentation degree allows us to ensure condition~\eqref{eq:relation_condition} is satisfied and hence to extend a fixed linear map to a morphism of sequences in Theorem~\ref{thm:extending_maps_to_lims}, thus answering the question posed in the beginning of this section. Indeed, comparing Theorem~\ref{thm:extending_maps_to_lims} with Proposition~\ref{prop:extend_using_rels}, we see that condition~\eqref{eq:relation_condition} is satisfied by any fixed equivariant map in dimension $n_0$ exceeding the presentation degree.
Thus, the presentation degree appears in our extendability result for convex sets (Theorem~\ref{thm:extending_compatible_descriptions}), and in our algorithm for computationally parametrizing such sets (Algorithm~\ref{algo:free_descr}). 
\section{Schur--Horn Orbitopes in AF Algebras}\label{apdx:SH_in_AF_algebras}
The permutahedra and Schur--Horn orbitopes studied in Section~\ref{sec:DS_permuta} are special cases of a more general construction arising in the field of operator algebras, and which is naturally analyzed using our framework. In this appendix, we present this unified view.
Let $\vct A_n\subseteq \vct M_n\coloneqq\CC^{m2^n\times m2^n}$ be a unital $\CC$-subalgebra with the faithful tracial state $\tau_n(x)=\mathrm{Tr}(x)/m2^n$ which induces the inner product $\langle x,y\rangle_n = \tau_n(x^*y)$, and the max-singular value operator norm $\|\cdot\|_{(n)}$. Let $\varphi_n\colon \vct M_n\hookrightarrow \vct M_{n+1}$ be the unital algebra embedding $\varphi_n(x)=x\otimes I_2$, and note that $\|x\otimes I_2\|_{(n+1)}=\|x\|_{(n)}$ and $\tau_{n+1}\circ\varphi_n = \tau_n$, so $\varphi_n$ is an isometry with respect to the above inner product. 
Thus, both the operator norm $\|\cdot\|$ and the normalized trace $\tau$ do not depend on $n$ and extend to the limit $\vct A_{\infty}=\bigcup_n\vct A_n$, and we omit their subscripts above.
Let $\overline{\vct A_{\infty}}$ be the closure of $\vct A_{\infty}$ with respect to the operator norm, which is now a $C^*$ algebra called an approximately finite-dimensional (AF) algebra~\cite[Chap.~3]{davidson1996c}. Note that $\tau$ extends continuously to $\overline{\vct A_{\infty}}$ since $|\tau(x)\|\leq \|x\|$ for all $x\in \vct A_{\infty}$, and hence $\tau$ defines a tracial state on $\overline{\vct A_{\infty}}$.

The spectral theorem for $C^*$ algebras gives an (isometric) isomorphism of algebras between continuous functions on the spectrum $\sigma(x)=\{\lambda\in\CC: x-\lambda I \textrm{ not invertible}\}$ of any self-adjoint $x\in \overline{\vct A_{\infty}}$ and the closed subalgebra it generates~\cite[Thm.~11.19]{rudinFuncAnal}. This allows us to apply continuous functions $f\colon \sigma(x)\to \CC$ to $x$ itself to obtain $f(x)\in\overline{\vct A_{\infty}}$. Using this functional calculus, for each self-adjoint $x\in \overline{\vct A_{\infty}}$ we get the positive unital linear functional on such functions sending $f\mapsto \tau(f(x))$. By the Riesz representation theorem~\cite[Thm.~6.19]{rudinRealComplex}, there is a unique probability measure $\mu_x^{\tau}$ on $\sigma(x)$ satisfying $\mu_x^{\tau}(f) = \tau(f(x))$ for all continuous $f\colon\sigma(x)\to\CC$, which is called the spectral measure of $x$ with respect to $\tau$.
Finally, if $\sigma(x)=\{\lambda_1,\ldots,\lambda_q\}$ is finite then we have a spectral decomposition $x=\sum_{i=1}^q\lambda_ip_i$ where $\{p_i\}$ are orthogonal projectors satisfying $p_ip_j=\delta_{i,j}p_i$, $\sum_ip_i=1$, and $\mu_x^{\tau} = \sum_{i=1}^q\tau(p_i)\delta_{\lambda_i}$.

Let $\vct V_n = \vct A_n\cap \vct M_n^{\mathrm{s.a.}}$ be the (finite-dimensional real) vector space of self-adjoint elements in $\vct A_n$, and note that $\varphi_n(\vct V_n)\subseteq \vct V_{n+1}$ and that $\overline{\vct V_{\infty}}$ is the collection of self-adjoints in $\overline{\vct A_{\infty}}$. Let $\msf G_n\subseteq \mathrm{U}(m2^n)$ be a group of unitaries acting on $\vct V_n$ by conjugation such that $\msf G_n\subseteq \msf G_{n+1}$ under the embedding $g\mapsto g\otimes I_2$, so that $\{\vct V_n\}$ is a consistent sequence of $\{\msf G_n\}$-representations.
\begin{lemma}
    Let $\mc P_n\colon \vct V_{\infty}\to \vct V_n$ be orthogonal projections. Then $\|\mc P_nx\|\leq \|x\|$ and $\tau\circ \mc P_n = \tau$ for all $n$.
\end{lemma}
\begin{proof}
Note that $\varphi_n^*\colon \vct V_{n+1}\to \vct V_n$ satisfies
\begin{equation}\label{eq:AF_orth_projs}
    \varphi_n^*(x) = \frac{1}{2}\sum_{i=1}^2(I\otimes e_i)^\top x(I\otimes e_i),\qquad \textrm{ hence } \|\varphi_n^*x\|\leq \|x\|,
\end{equation}
for all $x\in\vct V_{n+1}$. Because the orthogonal projection $\mc P_n\colon \vct V_{\infty}\to \vct V_n$ satisfies $\mc P_n|_{\vct V_N}=\varphi_n^*\circ\ldots\circ \varphi_{N-1}^*$ for each $N>n$, we conclude that $\|\mc P_nx\|\leq \|x\|$ for all $x\in \vct V_{\infty}$. 
Also, for any $x\in \vct V_{\infty}$ we have $\tau(\mc P_nx) = \langle \mc P_nx, 1\rangle = \langle x,\mc P_n1\rangle = \langle x,1\rangle = \tau(x)$, hence the second claim follows.
\end{proof}
In particular, the norm $\|\cdot\|$ satisfies the condition of Definition~\ref{def:cont_lims}.
Let $\lambda\in\RR^q$ and $\widetilde \lambda=[\lambda_1\mathbbm{1}_{m_i}^\top,\ldots,\lambda_q\mathbbm{1}_{m_q}^\top]^\top\allowbreak\in\RR^m$ has entry $\lambda_i$ repeated $m_i$ times (so $m=\sum_im_i$). Then $x_1 = \mathrm{diag}(\widetilde\lambda)$ has spectral measure $\mu_{x_1}^{\tau} = \sum_{i=1}^q\frac{m_i}{m}\delta_{\lambda_i}$.
The Shur-Horn orbitopes associated to $x_1$ is then
\begin{equation*}\begin{aligned}
    \mathrm{SH}(x_1)_n = \mathrm{conv}\{x\in \vct V_n: \mu_x^{\tau} = \mu_{x_1}^{\tau}\} =\mathrm{conv}\left\{\sum\nolimits_{i=1}^q\lambda_ip_i: p_i\in \vct V_n,\ p_ip_j=\delta_{i,j}p_i,\ \sum\nolimits_ip_i=1,\ \tau(p_i)=\frac{m_i}{m}\right\},
\end{aligned}\end{equation*}

To obtain conic descriptions for the above orbitopes, let $\cvx K_n = \{xx^*: x\in \vct A_n\}\subseteq \vct V_n$ be the cone of positive-semidefinite Hermitian matrices in $\vct A_n$, so that $\cvx K_n\subseteq K_{n+1}$ and $\overline{\cvx K_{\infty}}=\{xx^*:x\in\overline{\vct A_{\infty}}\}$. 
We also have $\mc P_n(K_{n+1}) = \cvx K_n$ by~\eqref{eq:AF_orth_projs}, hence $\mc P_n(\overline{\cvx K_{\infty}}) = \cvx K_n$.
Therefore, the finite-dimensional description of Schur--Horn orbitopes~\cite[Eq.~(19)]{finding_planted} reads
\begin{equation}\label{eq:finite_SH_description}
    \mathrm{SH}(x_1)_n = \left\{\sum\nolimits_{i=1}^q\lambda_ip_i: p_i\in \cvx K_n,\ \sum\nolimits_ip_i=1,\ \tau(p_i)=m_i/m\right\}.
\end{equation}
Our Theorem~\ref{thm:descriptions_of_lims} now yields the following infinite-dimensional extension of these descriptions.
\begin{proposition}\label{prop:infinite_SchurHorn_descrip}
    Let $\lambda_1,\ldots,\lambda_q$ are distinct real numbers, and let $m_1,\ldots,m_q\in \NN$ with sum $m=\sum_im_i$. Then $\overline{\mathrm{SH}(x_1)_{\infty}}$ with $x_1=\diag(\widetilde\lambda)$ as above is
    \begin{equation}\label{eq:AF_schurHorn_descrip}
        \overline{\mathrm{conv}}\left\{x\in \overline{\vct V_{\infty}}: \mu_x^{\tau}=\sum\nolimits_{i=1}^q\frac{m_i}{m}\delta_{\lambda_i}\right\} = \overline{\left\{\sum\nolimits_{i=1}^q\lambda_ip_i: p_i\in \overline{\cvx K_{\infty}},\ \sum\nolimits_ip_i=1,\ \tau(p_i)=m_i/m\right\}}.
    \end{equation}
\end{proposition}
\begin{proof}
    First, the left-hand side of~\eqref{eq:AF_schurHorn_descrip} is just $\overline{\mathrm{SH}(x_1)_{\infty}}$. Indeed, if $x_n\in \overline{\vct V_{\infty}}$ is a sequence converging to $x$ with $\mu_{x_n}^{\tau}=\mu_{x_1}^{\tau}$ for all $n$, then for any continuous $f\colon\RR\to\CC$ we have $\mu_x^{\tau}(f) = \tau(f(x))=\lim_n\tau(f(x_n))=\mu_{x_1}^{\tau}(f)$ so $\mu_x^{\tau}=\mu_{x_1}^{\tau}$. Conversely, if $\mu_x^{\tau}=\sum_i\frac{m_i}{m}\delta_{\lambda_i}$ then $x$ admits a decomposition $x=\sum_i\lambda_ip_i$ with $\tau(p_i)=m_i/m$. 
    Define $x_n = \mc P_nx = \sum_{i=1}^q\lambda_i(\mc P_np_i)$ and note that $\mc P_np_i\in \cvx K_n$ for all $i$, that $\sum_i\mc P_np_i = \mc P_n1 = 1$, and that $\tau(\mc P_np_i) = \tau(p_i) = m_i/m$. Thus, $x_n\in \mathrm{SH}(x_1)_n$ for all $n$ and $x_n\to x$ by Lemma~\ref{lem:projections_converge}, hence $x\in \overline{\mathrm{SH}(x_1)_{\infty}}$.

    Second, we appeal to Theorem~\ref{thm:descriptions_of_lims} to obtain equality in~\eqref{eq:AF_schurHorn_descrip}. 
    Note that the conic description~\eqref{eq:finite_SH_description} is of the form~\eqref{eq:preim_seq} with $\mscr V = \{\vct V_n\}, \mscr W = \mscr V^{\oplus q}, \mscr U = \mscr W\oplus\mscr V^{\oplus 2}\oplus \RR^q$, and 
    \begin{equation*}
        A_nx = (0,-x,0,0),\quad B_n(p_1,\ldots,p_q) = \left(p_1,\ldots,p_q,\sum\nolimits_i\lambda_ip_i, \sum\nolimits_ip_i, \tau(p_1),\ldots,\tau(p_q)\right),
    \end{equation*}
    and $u_n = (0,0,-I,-m_1/m,\ldots,-m_q/m)$. Observing that $\{A_n\},\{A_n^*\},\{B_n\},\{B_n^*\}$ are all morphisms, that $u_n=u_{n+1}$ (under our embeddings), the descriptions in~\eqref{eq:finite_SH_description} are \new{dimension-}free and satisfy the hypotheses of Proposition~\ref{prop:compatible_descriptions}(a). Furthermore, putting the norm $\|(p_1,\ldots,p_q)\|=\max_i\|p_i\|$ on $W_{\infty}$ and $\|(p_1,\ldots,p_q,y_1,y_2,v)\|\leq \max\{\|p_i\|,\|y_j\|,\|v\|_{\infty}\}$ on $\vct U_{\infty}$, we have $\overline{W_{\infty}}=\overline{\vct V_{\infty}}^{\oplus q}$ and $\overline{\vct U_{\infty}}=\overline{W_{\infty}}\oplus \overline{\vct V_{\infty}}^{\oplus 2}\oplus \RR^q$. Moreover, $\|A_n\|_{\mathrm{op}}=1$, $\|B_n\|_{\mathrm{op}}\leq \max\{\sum_i|\lambda_i|,q\}$ for all $n$. Thus, both $A_n$ and $B_n$ extend to the continuous limit and Theorem~\ref{thm:descriptions_of_lims} indeed applies and yields~\eqref{eq:AF_schurHorn_descrip}.
\end{proof}
The Schur--Horn orbitopes $\mathrm{SH}(x_1)_n$ are precisely the Schur--Horn orbitopes $\mathrm{SH}(\lambda)_n$ in~\eqref{eq:schur_horn} from Section~\ref{sec:DS_permuta} if $\vct A_n = \RR^{m2^n\times m2^n}$ with $\msf G_n=\mathrm{O}(m2^n)$. They also reduce to the permutahedra $\mathrm{SH}(x_1)_n=\mathrm{Perm}(\lambda)_n$ as in~\eqref{eq:permuta_good} from Section~\ref{sec:DS_permuta} if $\vct A_n$ is the algebra of diagonal matrices in $\vct M_n$ with $\msf G_n=\mathrm{S}_{m2^n}$. 

Proposition~\ref{prop:infinite_SchurHorn_descrip} yields a generalization of the Schur--Horn theorem to AF algebras. 
Indeed, let $\vct D_n = \{x\in \vct V_n: x \textrm{ diagonal}\}$ which is a unital subalgebra of $\vct A_n$ embedded in $\vct D_{n+1}$ under $\varphi_n$. 
For $x_1\in \vct D_1$, define
\begin{equation*}
    \mathrm{Perm}(x_1)_n = \mathrm{SH}(x_1)_n\cap \vct D_n.
\end{equation*}
The sequence of linear maps $(\diag_n\colon \vct V_n\to \vct D_n)_n$ extracting the diagonal of a Hermitian matrix in $\vct V_n$ extend to a bounded linear map $\diag\colon \overline{\vct V_{\infty}}\to \overline{\vct D_{\infty}}$ since $\diag_{n+1}\circ\varphi_n=\varphi_n\circ \diag_n$ and $\|\diag_n\|_{\mathrm{op}} = 1$. 
The finite-dimensional Schur--Horn theorem states that $\diag_n(\mathrm{SH}(x_1)_n) = \mathrm{Perm}(x_1)_n$. 
Using Proposition~\ref{prop:infinite_SchurHorn_descrip}, we obtain the following infinite-dimensional extension.
\begin{proposition}\label{prop:infinite_SchurHorn_theorem}
    Let $x_1\in \vct D_1$ and let $\diag\colon\overline{\vct V_{\infty}}\to\overline{\vct D_{\infty}}$ be the bounded linear map extending the diagonal maps. Then $\diag(\overline{\mathrm{SH}(x_1)_{\infty}}) = \overline{\mathrm{Perm}(x_1)_{\infty}}$. 
\end{proposition}
\begin{proof}
Since $\overline{\mathrm{Perm}(x_1)_{\infty}}\subseteq \overline{\mathrm{SH}(x_1)_{\infty}}$ and $\diag$ is a projection onto $\overline{\vct D_{\infty}}$, we get $\diag(\overline{\mathrm{SH}(x_1)_{\infty}}) \supseteq \overline{\mathrm{Perm}(x_1)_{\infty}}$. Conversely, Proposition~\ref{prop:infinite_SchurHorn_descrip} and continuity of $\diag$ yields
\begin{equation*}\begin{aligned}
    \diag(\overline{\mathrm{SH}(x_1)_{\infty}}) &\subseteq \overline{\left\{\sum\nolimits_{i=1}^q\lambda_i\diag(p_i): p_i\in \overline{\cvx K_{\infty}},\ \sum\nolimits_ip_i=1,\ \tau(p_i)=m_i/m\right\}}\\ 
    &= \overline{\left\{\sum\nolimits_{i=1}^q\lambda_ip_i: p_i\in \overline{\cvx K_{\infty}}\cap \overline{\vct D_{\infty}},\ \sum\nolimits_ip_i=1,\ \tau(p_i)=m_i/m\right\}}\\
    &= \overline{\mathrm{Perm}(x_1)_{\infty}},
\end{aligned}\end{equation*}
where the last equality follows from Proposition~\ref{prop:infinite_SchurHorn_descrip} applied to $\overline{\vct D_{\infty}}$ instead of $\overline{\vct V_{\infty}}$.
\end{proof}
The finite-dimensional Schur--Horn theorem yields $\mathrm{Perm}(x_1)_{\infty} = \mathrm{diag}(\mathrm{SH}(x_1)_{\infty})$, and taking closures yields a weaker statement since it involves the closure of the image of diag. Proposition~\ref{prop:infinite_SchurHorn_theorem} shows that this closure can be removed. 

There is a considerable literature on extending the Schur--Horn theorem to infinite-dimensional setting, including the extension in~\cite{bownik2015schur} for operators with a finite spectrum, and in~\cite{ravichandran2012schur} for von Neumann algebras. As explained in~\cite[\S1]{ravichandran2012schur}, removing the closure over the image of diag has been a major challenge in these more general settings. Our descriptions of limiting convex sets from Theorem~\ref{thm:descriptions_of_lims} can be seen as resolving this challenge in our simpler setting of AF algebras.

\end{document}